\documentclass{amsart}
\usepackage{amsmath,amsthm,amssymb,amsfonts,ifpdf,bbm,calc,mathtools}
\usepackage{mathabx}
\usepackage{enumitem}
\usepackage{graphicx}
\usepackage{ifthen}

\usepackage[usenames,dvipsnames]{color}

\ifpdf
\usepackage[pdftex,pdfstartview=FitH,pdfborderstyle={/S/B/W 1},%
colorlinks=true, linkcolor=blue, urlcolor=blue, citecolor=blue,%
pagebackref=true]{hyperref}
\else
\usepackage[dvips]{hyperref}
\fi
\numberwithin{equation}{section}

\newtheorem {theorem}    {Theorem}[section]
\newtheorem {lemma}      [theorem]    {Lemma}

\newtheorem {proposition}[theorem]    {Proposition}

\theoremstyle{definition}
\newtheorem {definition} [theorem]    {Definition}

\newcounter{AbcT}

\numberwithin{equation}{section}

\newcommand {\goesto} {\longrightarrow}

\newcommand {\C} {{\mathbb C}}

\newcommand {\K} {{\mathbb K}}
\newcommand {\N} {{\mathbb N}}
\newcommand {\Q} {{\mathbb Q}}
\newcommand {\R} {{\mathbb R}}

\newcommand {\T} {{\mathbb T}}
\renewcommand {\H} {{\mathbb H}}
\newcommand {\Z} {{\mathbb Z}}

\renewcommand{\liminf}{\varliminf}

\DeclareMathOperator{\supp}{supp}

\DeclareMathOperator{\SL}{SL}

\DeclareMathOperator{\PGL}{PGL}

\newcommand {\IGNORE}[1]  {}

\newcommand {\absolute}[1] {\left| {#1} \right|}
\newcommand {\norm}[1] {\left\| {#1} \right\|}

\newcommand{\G}{\mathbf{G}}
\newcommand{\GG}{\mathbb{G}}
\newcommand{\TT}{\mathbb{T}}
\newcommand{\Adel}{\mathbb{A}}

\renewcommand{\sl}{\mathfrak{sl}}

\newcommand{\dee}{\operatorname{d}\!}
\newcommand{\fancyX}{\mathcal X}
\newcommand{\fancyY}{\mathcal Y}
\newcommand{\fancyZ}{\Omega}
\newcommand{\order}{\mathcal O_S}
\newcommand\vol{\operatorname{vol}}

\newcommand\cA{\mathcal{A}}

\newcommand{\LL}{\mathbb{L}}
\newcommand{\HH}{\mathbb{H}}

\newcommand{\UU}{\mathbb{U}}

\newcommand{\ZZ}{\mathbb{Z}}

\newcommand{\QQ}{\mathbb{Q}}
\newcommand{\KK}{\mathbb{K}}
\newcommand{\GGK}{{\GG^\KK}}
\newcommand{\sa}[1]{{\mathsf a({#1})}}
\newcommand{\ssa}[1]{{\mathring{\mathsf a}({#1})}}
\newcommand{\myentry}[1]{\ifthenelse{\equal{#1}{.}}{}{#1}}
\newcommand{\twobytwo}[4]{\left(\!\begin{smallmatrix}%
   \myentry{#1} & \myentry{#2} \\ \myentry{#3} & \myentry{#4}%
\end{smallmatrix}\!\right)}

\begin{document}

	\title[Rigidity of torus actions and unipotent recurrence]{Rigidity of 
	non-maximal torus actions, unipotent quantitative recurrence, and Diophantine approximations}
	\author[M. Einsiedler]{Manfred Einsiedler}
	\address[M. E.]{ETH Z\"urich, R\"amistrasse 101
	CH-8092 Z\"urich
	Switzerland}
	\email{manfred.einsiedler@math.ethz.ch}
	\author[E. Lindenstrauss ]{Elon Lindenstrauss}
	\address[E. L.]{The Einstein Institute of Mathematics\\
	Edmond J. Safra Campus, Givat Ram, The Hebrew University of Jerusalem
	Jerusalem, 91904, Israel}
	\email{elon.bl@mail.huji.ac.il}
	\thanks{M.~E.~acknowledges the support by the SNF (Grant 200021-152819 and 200020-178958).
	E.~L.~acknowledges the support by ERC 2020 grant HomDyn (grant no.~833423). The authors gratefully acknowledge the support of
	the Israeli Institute for Advanced Studies at the Hebrew University, where this work originated.}
	\date{\today}
	\begin{abstract}
We present a new argument in the study of positive entropy measures for higher rank diagonalisable actions.  The argument relies on a quantitative form of recurrence along unipotent directions (that are not known to preserve the measure). 
Using this argument we prove a classification of positive entropy measures for any higher rank action on an irreducible arithmetic quotient 
of a form of $\SL_2$. 
We also provide an Adelic version of this classification result where no entropy assumption is needed. These results can also be used to prove new results regarding Diophantine approximations of integer multiples of an arbitrary element $\alpha\in\R$.
\end{abstract}
	\maketitle

\section{Introduction}\label{introduction}

\subsection{Historical background}
Higher rank diagonalizable actions have subtle rigidity properties, which are remarkable because the action of each individual element of the action has no rigidity. 
Furstenberg initiated this area of research in his influential paper \cite{Furstenberg-disjointness-1967} by studying the rigidity properties of the action of a non-lacunary multiplicative semigroup of integers on $\R / \Z$, which is closely connected to the type of actions we consider here, and motivated much of the research in the area.

We concern ourselves here with the problem of classifying invariant measures under higher rank abelian groups. This measure classification question was raised by Furstenberg following his work \cite{Furstenberg-disjointness-1967}, and conjectures in this vein in the context of actions on homogeneous spaces were made by Katok and Spatzier \cite[Main Conj.]{KatokSpatzier96} and in a more explicit form by Margulis \cite[Conj.\ 2]{Margulis-conjectures}; cf.\ also \cite[Conj. 2.4]{Einsiedler-Lindenstrauss-ICM}. The corresponsing question for orbit closures dates even further back to the work of Cassels and Swinnerton-Dyer \cite{Cassels-Swinnerton-Dyer}. Whether in the original context of actions of nonlacunary semigroups of integers or in the case of the action of higher rank abelian groups on homogeneous spaces all progress regarding this measure classification, e.g.~\cite{Rudolph-2-and-3,KatokSpatzier96,Lindenstrauss-03,Einsiedler-Katok-Lindenstrauss},
has been obtained under the assumption of positive entropy. Even assuming positive entropy, the state-of-the-art regarding classifications of invariant measures on quotients of semisimple groups requires --- even for
quotients arising from forms of $\operatorname{SL}_2^k$ --- that the acting group be a maximal split torus in a normal subgroup \cite{full-torus-paper}. Our goal here is to present an argument that does not need this strong maximality assumption on the 
action (or the 
assumption that the invariant measure is a joining as in \cite{Einsiedler-Lindenstrauss-joinings-2}).
We note that the results presented here are new even in the case
of homogeneous spaces of real groups, but as 
many number theoretic
applications use dynamics 
on~$S$-arithmetic locally homogeneous spaces 
we give the results in this setting.

We note that the joining classification results from \cite{Einsiedler-Lindenstrauss-joinings-2}, even for the case of quotients arising from forms of $\operatorname{SL}_2^k$ we consider here, have had several noteworthy number theoretic applications, e.g. \cite{Aka-Einsiedler-Shapira,Aka-Einsiedler-Wieser,Khayutin-CM-joinings}, and we hope this new measure classification theorem will have similar applications.

\subsection{Statement of the main result}

Let~$G$ be a locally compact group, let~$\Gamma<G$ be a lattice, and let $X=\Gamma\backslash G$ be the corresponding quotient space.
The action of~$G$ on~$X$ is defined by~$g.x=xg^{-1}$ for~$x\in X$ and~$g\in G$.
A probability measure~$\mu$ on~$X$ is called \emph{homogeneous} if there exists a
subgroup~$H<G$ and a point~$x\in X$ such that~$\Lambda=\operatorname{Stab}_H(x)$ is a lattice in $H$, and~$\mu$ is the unique~$H$-invariant
probability measure on~$xH\cong\Lambda\backslash H$.

The setting we consider in this paper is that of an~$S$-arithmetic quotient defined using a~$\Q$-group~$\GG$ and a finite non-empty set of places
\[
 S\subseteq\{\infty\}\cup\{p\mid p\in\N\text{ is a prime}\}
\]
as follows. For any~$v\in S$ we let~$G_v=\GG(\Q_v)$ be the set of~$\Q_v$-points of~$\GG$ (where~$\Q_\infty=\R$) and 
define the~$S$-algebraic group
\[
 G=\prod_{v\in S}G_v=\GG(\Q_S),
\]
where~$\Q_S=\prod_{v\in S}\Q_v$.
We require that~$\infty\in S$.
Now let~$\Gamma<G$ be an arithmetic lattice commensurable to~$\GG(\order)$ with $\order=\Z\bigl[\frac{1}p\mid p\in S\bigr]$. In this case we say that~$\Gamma$ is an \emph{arithmetic lattice arising from the~$\Q$-structure}~$\GG$.

\begin{definition}\label{def:algebraic measure}
We say that a measure $\mu$ on an $S$-arithmetic quotient $\Gamma\backslash G$ as above is \emph{algebraic over $\Q$} if it is homogeneous and moreover
there exists 
a~$\Q$-subgroup~$\HH\leq\GG$ and a finite index subgroup~$H\leq\HH(\Q_S)$
such that~$\mu$ is the normalized Haar measure on a single orbit~$\Gamma H g$
for some~$g\in G$.
\end{definition}

Our main concern are probability measures that are invariant under the action of a diagonalizable subgroup of the following type. 

\begin{definition} \label{def:class-A}
A closed commutative subgroup $A < G$ is said to be of \emph{class-$\cA'$} if it can be simultaneously diagonalized 
and
\begin{itemize}
\item
for every $a \in A$ the projection of $a$ to $G_\infty=\GG(\R)$ has positive real eigenvalues,
\item for each prime $p\in S$ there exists some $\lambda_{p}\in\Q_p^{*}$
with $|\lambda_{p}|_{p}\neq 1$ so that the projection of every $a\in A$ 
to $G_p=\GG(\Q _ p)$ satisfies that all of its eigenvalues
are integer powers of $\lambda_{p}$.
\end{itemize} 
An element $g \in G \setminus\{e\}$ is said to be of class-$\cA'$ if the group it generates is of class-$\cA'$. 
A closed subgroup~$A<G$ is of \emph{higher rank} if it contains a subgroup generated by semisimple elements that is
in the category of topological groups isomorphic
to~$\Z^2$. 
\end{definition}

Our definition of class-$\mathcal{A}'$ subgroups and elements is taken 
from \cite{Einsiedler-Lindenstrauss-joinings-2} and is more general than the notion of class-$ \mathcal{A}$ elements introduced by Margulis and Tomanov in \cite{Margulis-Tomanov}, though for the purposes of \cite{Margulis-Tomanov,Margulis-Tomanov-almost-linear} one can use this more general class equally well.

We recall that an algebraic group $\mathbb{G}$ over $\mathbb{Q}$ is \emph{a form of} $\operatorname{SL}_2^k$ (for some $k\in\mathbb{N}$) if it is
isomorphic to $\operatorname{SL}_2^k$ over $\overline{\mathbb{Q}}$. Also recall that  $\Q$-almost simplicity of $\GG$ is equivalent to assuming that the arithmetic lattices arising from the $\Q$-structure $\GG$ are irreducible.
Any $\Q$-almost simple form of $\operatorname{SL}^k$ arises through restriction of scalars from an algebraic group that is a form of $\SL_2$
and is defined over a number field $\KK$ of degree $k$. These in terms can be described very concretely in terms of quaternion algebras as follows. Let $\KK$ be a number field with $k=[\KK:\Q]$, and let $\mathbb D$ be a quaternion algebra over $\KK$, that is either the algebra of $2\times2$ matrices over $\KK$ or a division algebra of dimension 4 over $\KK$. Let $\GGK$ denote the $\KK$-group of elements of $\mathbb D$ of reduced norm 1. Then the $\Q$-group $\mathbb G = \operatorname{Res}_{\KK\mid\Q}(\GGK)$, is a form of~$\SL _ 2 ^ k$.
For any set of places $S \ni \infty$ for $\Q$, we can canonically identify
\[
\GG(\Q_S) = \prod_{v \in S}\prod_{\sigma|v} \GGK (\KK_\sigma),
\]
where in the second product $\sigma$ runs over all places of $\KK$ lying over $v$.
If we take $\Gamma$ to be an arithmetic lattice in $\GG(\Q_S) $, i.e. a lattice commensurable to $\GG(\order)$, then 
$\Gamma \backslash\GG(\Q_S) $ is non-compact iff $\mathbb D$ is isomorphic to the the algebra of $2\times2$ matrices over $\KK$ in which case $\GGK=\SL_2$ and 
$\GG=\operatorname{Res}_{\KK\mid\Q}\SL_2$.

Similarly, if we take $\GGK$ denote the $\KK$-group of invertible elements of $\mathbb D$ divided by scalars, then $\mathbb G =\operatorname{Res}_{\KK\mid\Q}(\GGK)$ is a form of $\PGL _ 2 ^ k$.
Moreover, any $\Q$-almost simple form of $\SL _ 2 ^ k$ or $\PGL _ 2 ^ k$ is isomorphic as a $\Q$-group to such a group; see \cite[Ch.~4]{Platonov-Rapinchuk}. 

The following theorem represents the main result of this paper.

\begin{theorem}[Forms of~$\operatorname{SL}_2^k$]\label{sl2-thm-final}
	Let~$\GG$ be an algebraic group over~$\Q$ that is~$\Q$-almost simple
	and a form of~$\operatorname{SL}_2^k$ or of~$\PGL_2^k$ with~$k\geq 1$, and let~$\Gamma$
	be an arithmetic lattice in~$G=\GG(\Q_S)$ arising from~$\GG$. 
	Let~$A<G$ be a closed abelian class-$\cA'$ subgroup of higher rank.  Let~$\mu$
 be an~$A$-invariant and ergodic probability measure on~$X=\Gamma\backslash G$ such that~$h_\mu(a)>0$
for some~$a\in A$.  Then one of the following holds:
\begin{enumerate}[label=\textup{(\textit{\roman*})},leftmargin=*]
	\item \label{item-algebraic}\textup{\bf (Semi-Simple)} The measure~$\mu$ is algebraic over~$\Q$. Moreover, if $\HH$ is a $\Q$-group so that $\mu$ is the uniform measure on $\Gamma Hg$ for $H\leq \HH(\Q_S)$ as in Definition~\ref{def:algebraic measure}, then $\HH$ is semisimple and $A<g^{-1}\H(\Q_S)g$. Moreover, if $\GG$ is a form of $\SL_2^k$ then $H=\H(\Q_S)$.
	
    \item\label{item-solvable} \textup{\bf (Solvable)} The space~$X$ is non-compact.  
	There exists a nontrivial unipotent subgroup~$L$
such that~$\mu$ is invariant under~$L$
and the measure $\mu$ is
supported on a compact orbit~$xM$ of a solvable subgroup $M<G$ containing $A$ and normalizing $L$.
Moreover, $L=g^{-1}\LL(\Q_S)g$ for a unipotent
$\Q$-group $\LL$ and a representative $g\in G$
of $x=\Gamma g$. 
\end{enumerate}
\end{theorem}

\medskip 

In the non-compact case the second possibility can indeed occur.
In fact, on these locally homogeneous spaces
some cases of the measure classification problem
for actions on the~$k$-dimensional
torus subgroup~$(\R/\Z)^k$ via units of the number field~$\K$ 
or~$S$-units on some solenoids (as considered e.g.~in \cite{EL-action-on-torus})
can be embedded in the measure classification
problem for the action of~$A$ on~$X$. 
A hypothetical counterexample to Furstenberg's conjecture regarding $\times 2$, $\times 3$-invariant measure measure on $\R/\Z$ would give an example of an $A$-invariant and ergodic measure on $X$ that is not homogeneous supported on a closed orbit of a solvable group as in Theorem~\ref{sl2-thm-final}\ref{item-solvable}.

We also note that with a bit more work
it is possible to obtain a stronger conclusion in the solvable case described in Theorem~\ref{sl2-thm-final}\ref{item-solvable} either by the methods of this paper or by applying the methods of \cite{EL-action-on-torus}. Specifically, one can show that $L$ may be chosen so that on the quotient of $M/L$ by the appropriate lattice any $a\in A$ will act with zero entropy with respect to the measure induced by $\mu$, i.e.\ all of the entropy for elements of $A$ is ``explained'' by the unipotent subgroup $L$. 
We omit the details as e.g.\ using the techniques of  \cite{EL-action-on-torus} would require a nontrivial increase in the length of this paper, for a rather modest improvement.

\medskip

The requirement that $A$ be of class-$ \mathcal{A} '$ is necessary for Theorem~\ref{sl2-thm-final} to hold, but omitting this requirement gives only a slightly weaker result.

\begin{definition}
Let $G$ be a locally compact group, $A$, $A'$, and $\Gamma$ closed subgroups of $G$. Let $\mu$ resp.\ $\mu'$ be $A$-invariant resp.\ $A'$-inviarnat probability measures on $\Gamma \backslash G $. We say $\mu'$  is \emph{equivalent-by-compact} to $\mu $ if there is a closed subgroup $A''\leq G$ so that
\begin{itemize}
    \item both $A$ and $A'$ are cocompact normal subgroups of $A''$, and 
    \item $\int_{A''/A} a.\mu\, da = \int_{A''/A'} a.\mu' da$,
    where the integrals is w.r.t.\ the normalized Haar measures on $A''/A$ resp.\ $A''/A'$. 
\end{itemize}
\end{definition}
While the definition makes sense more generally, we will only need here the case where $A,A'$ and $A''$ are abelian.

\begin{theorem}[Without the class-$\cA'$ assumption]\label{sl2-thm-without}
	Let~$\GG$ and $\Gamma$ be as in Theorem~\ref{sl2-thm-final}, and
	let~$A<G$ be a closed abelian diagonalizable 
 subgroup of higher rank.  Let~$\mu$
 be an~$A$-invariant and ergodic probability measure on~$X=\Gamma\backslash G$ such that~$h_\mu(a)>0$
for some~$a\in A$.  Then $\mu$ is equivalent-by-compact to an $A'$-invariant measure $\mu'$ with $A'$ of class-$\cA'$ and $\mu'$ satisfying either \ref{item-algebraic} or \ref{item-solvable} of Theorem~\ref{sl2-thm-final} for $A'$.
\end{theorem}

\noindent
(In the case $\GG=\operatorname{Res}_{\KK\mid\Q}\SL_2$, this follows from Theorem~\ref{positive entropy one} below. The case for general $\GG$ is similar and is left to the reader.)

\subsection{Adelic Measure Classification Theorems}\label{s:adelic}

In the adelic setup we can give a complete classification result avoiding an entropy assumption. Let us write $\Adel_\K$
for the adeles over a number field $\K$, and $S_\K$ for the set of places of $\K$ (including infinite ones). If $v | p$, we normalize the absolute values $|\cdot|_v$ so that they extend $|\cdot|_p$ and let $d_v = [K_v:\Q_p]$; then for any $k\in \KK^\times$
\[
\prod_{v \in S_\K} |k|_v^{d_v}=1. 
\]
Recall that for $k \in \KK$, the logarithmic height is defined by
\[
h (k) = \sum_ {v \in S_\KK} d_v\log ^ + \absolute k _ v
\]
where $\log ^ + (x) = \max (0, \log x)$, hence for $k\in \KK^\times$ \[h(k) = \tfrac12 \sum_ {v \in S_\KK} d_v\absolute{\log |k| _ v}.\]
We extend this definition to any $y =(y_v)_{v\in S_\KK} \in \Adel^\times_\KK$: for such a $y$
\[
h(y) =  \tfrac12 \sum_ {v \in S_\KK} d_v\absolute{\log |y_v| _ v}.
\]

Let $\KK$ be a number field, and let $\Lambda <\Adel_\KK^\times$ be a discrete subgroup so that
\begin{enumerate}[label=\textup{(A\arabic*)}]
\item \label{item:big} There are $C,\kappa>0$ so that for any $t>0$,
\[
|\{ k \in \Lambda: h(k)\leq t \}| \geq C e^{\kappa t}.
\]
\end{enumerate}
We also consider the following two additional possible assumptions on $\Lambda$:
\begin{enumerate}[resume,label=\textup{(A\arabic*)}]
\item \label{item:indep-1} For any (possibly infinite) set of places $S \subsetneq S_\KK$, there is an element $k \in \Lambda$ so that $\prod_{\sigma\in S}|k|_\sigma ^{d_\sigma}\neq 1$;
 \item\label{item:isolate a place-1} there is a $\sigma \in S_\K$ lying over some $v \in S_\Q$ and a coarse Lyapunov subgroup (see \S\ref{sec:dynamicsintro}) $U^{[\alpha]}$ of $\SL_2(\Adel_K)$ so that (\textit{a}) $\KK_\sigma =\Q_v$, \ (\textit{b}) $U^{[\alpha]}$ intersects $\SL_2(\K_\sigma)$ non-trivially, and (\textit{c}) $U^{[\alpha]}$ does \emph{not} intersect $\SL_2(\K_{\sigma'})$ for any other $\sigma' |v$.
\end{enumerate}
For instance, it is easy to see that for any $\ell \geq 1$ the group $\Lambda = \left\{ k ^ \ell: k \in \KK ^ \times \right\}$ 
(diagonally embedded in $\mathbb{A}_\KK^\times$) satisfies the above three conditions.

\begin{theorem}[Adelic measure classification]\label{thm:adelic}
Let $\K$ be a number field, $G=\SL_2(\Adel_\K)$ and $\Gamma=\SL_2(\KK)$. Set
\[
X_{\Adel_\K}=\SL_2(\K)\backslash \SL_2(\Adel_\K).
\]
Let $\TT<\SL_2$ be the group of diagonal matrices; we identify $\TT(\K)\backslash \TT(\Adel_K)$ with its image in $X_{\Adel_\K}$. For $k \in \Adel_\K^\times$, let $\sa k=\left(\!\begin{smallmatrix}
k&\\&k^{-1}
\end{smallmatrix}\!\right)$.
Let $\Lambda < \Adel_\K^\times$ satisfy condition \ref{item:big}, and let $A = \left\{ \sa k: k \in \Lambda \right\}$.
Then any $A$-invariant and ergodic probability measure $\mu$ on $X_{\Adel_\K}$ is equivalent-by-compact to an $A'$-invariant measure $\mu'$ so that either
\begin{enumerate}[label=\textup{(\textit{\roman*})},leftmargin=*]
	\item \label{item-algebraic-pos-adelic-1}There is a semisimple $\Q$ algebraic subgroup $\HH \leq \GG=\operatorname{Res}_{\K|\Q}\SL_2$ and $g \in G$ so that $\mu'$ is the uniform measure on $\Gamma \HH(\Adel_\Q)g$.
	
    \item\label{item-solvable-pos-adelic-1} $\mu'$ is supported on a single orbit $\Gamma N^1_G (\UU) g$, where $\UU$ is the $\KK$-group $\twobytwo1*.1$ and $g \in G$.
    \item There is a $\KK$-torus $\HH < \SL _ 2 $ and $g \in G$ so that 
\begin{equation*}
\supp \mu ' \subseteq \Gamma \HH (\Adel_\KK) g \text{ \ and \ } g  A ' g ^{-1} \subseteq \HH (\Adel_{\KK})
.\end{equation*}
\end{enumerate}
\end{theorem}

Additional information about the case \ref{item-solvable-pos-adelic-1} is given in Theorem~\ref{positive entropy one} below. If $\Lambda$ also satisfied \ref{item:indep-1} and \ref{item:isolate a place-1} we obtain a stronger result:

\begin{theorem}[Adelic measure classification]\label{thm:adelic 2}
Let the notations be as in Theorem~\ref{thm:adelic}. Assume furthermore that $\Lambda$ satisfied \ref{item:indep-1} and \ref{item:isolate a place-1}. Let $\mu$ be an $A$-invariant and ergodic probability measure on $X_{\Adel_\K}$. Then one of the following must hold:
\begin{enumerate}[label=\textup{(\textit{\roman*})},leftmargin=*]
\item $\mu$ is the uniform Haar measure on $X_{\Adel_\K}$.
\item $\mu$ is the uniform Haar measure on the closed orbit $\Gamma\UU(\Adel_\KK) a$, where $a \in\UU(\Adel_\K)$ and $\UU$ is is one of the two unipotent $\K$-subgroups  $\twobytwo 1 * . 1$ or $\twobytwo 1 . * 1$.
\item There is a $x _ 0 \in \T (\KK) \backslash \T (\Adel_{\KK})$ so that $\mu = \delta _ {x _ 0}$. 
\end{enumerate}
\end{theorem}

We note that the above refines the unique ergodicity theorem
for the full adelic diagonal subgroup proven in \cite[Thm.~1.4]{Lindenstrauss-adelic}. A crucial point is that similar to that theorem, \emph{we can avoid an entropy assumption in the adelic context}: every 
prime provides an independent transformation on our space,
the prime number theorem makes this a very rich dynamical system, whose recurrence properties gives strong information about the measure.

\subsection{A Diophantine Application}\label{diophantus}

Let~$\alpha\in\R$ and let us denote by
\[
\langle\alpha\rangle=\min_{m\in\Z}|\alpha-m|
\]
the distance to the nearest integer. Then Dirichlet's theorem states that
\[
 Q\cdot\!\!\!\min_{q\in\N\cap[1,Q]}\langle q\alpha\rangle\leq 1
\]
for all~$Q\geq 1$.
Moreover, for a.e.\ $\alpha\in\R$ Khintchin's theorem gives that this quantity approaches zero
along a subsequence of~$Q$. On the other hand there is a large set of numbers $\alpha$, known as the set of 
badly approximable numbers, for which this quantity stays bounded away from zero.

In personal communications Bourgain asked whether this statement could be improved by allowing
a denominator that is a product of two numbers~$q_1,q_2\leq Q$. In other words, what is the behavior
of the following function
\[
 Q\in[0,\infty)\mapsto f(Q)=Q\cdot\!\!\!\min_{q_1,q_2\in\N\cap[1,Q]}\langle q_1q_2\alpha\rangle
\]
for an arbitrary $\alpha\in\R$?

If~$Q+1=p$ is a prime number and~$|\alpha-\frac{m}p|<\frac12Q^{-3}$ for some
nonzero~$m\in\Z$, then~$q_1q_2\alpha$ has distance less than~$\frac12Q^{-1}$ to~$\frac{q_1q_2m}p\in\Q\setminus\Z$ 
and so distance at least~$\gg Q^{-1}$ to the nearest integer. This implies~$f(Q)\gg 1$. 
Using the Baire Category Theorem it is now easy to see that there exists a dense $G_\delta$-set of real
number~$\alpha$ for which~$\operatorname{limsup}_{Q\to\infty}f(Q)>0$. For $\alpha = \sqrt r $ with $r$ rational or for almost every $\alpha$ Blomer, Bourgain, Radzwi\l{}\l{} and Rudnick \cite{Blomer-Bourgain-Radziwill-Rudnick} show that $f(Q) \ll_\epsilon Q^{-1+\epsilon}$ which is essentially optimal (see also \cite{Carmon-quadratic}).

Even thought one cannot guarantee that $\operatorname{limsup}_{Q\to\infty}f(Q)=0$ for \emph{all} $\alpha\in\R$, allowing products of denominators does give something: Theorem~\ref{dioph-cor} implies that 
for every real $\alpha$ we have $f(e^t)\to 0$ 
``for almost all scales''~$t\to\infty$. Indeed, we prove the following significantly stronger uniform statement:

\begin{theorem}[An improved Dirichlet-type theorem]\label{dioph-cor}
For any $\varepsilon>0$, and integer $\ell>0$ there exists an $N\in\N$ so that for every $\alpha\in\R$,
for every $Q>10$ there exist some integers $n\leq N$ and $q\leq Q$ for which
\[
\langle qn^{2\ell}\alpha\rangle\leq \min\bigl(\tfrac{\varepsilon}q,\tfrac{1+\varepsilon}Q\bigr)
\]
\end{theorem}

We note that Theorem \ref{dioph-cor} is a strengthening of the following theorem \cite[Thm.~1.11]{EFS} by the first named author, L.~Fishman, and U.~Shapira.

\begin{theorem}[Einsieder-Fishman-Shapira]\label{EFStheorem}
    Let $\alpha\in\R$. Then
    \[
    \liminf_{n\to\infty}\liminf_{q\to\infty}q\langle qn\alpha\rangle=0.
    \]
\end{theorem}

We also note that both Theorem \ref{EFStheorem} and Theorem \ref{dioph-cor} rely 
on adelic measure classification results. Theorem \ref{EFStheorem} uses \cite[Thm.~1.4]{Lindenstrauss-adelic} due to the second named author
while Theorem \ref{dioph-cor} instead uses the results from Section \ref{s:adelic}.

\subsection{Outline of the proofs}\label{outline}
We briefly outline the argument of the principle result of this paper,  Theorem~\ref{sl2-thm-final}, from which all other results are derived. For simplicity
of exposition we discuss only  
a special case of a quotient $X=\Gamma\backslash G$ of $G=\SL_2(\R)^4$ 
by an arithmetic irreducible lattice $\Gamma$. 
We suppose that $a\in G$ is the direct product of nontrivial positive 
diagonal matrices in the first two factors and the identity in the last two factors. Let $U$ be the product of the upper unipotent subgroup in the first two factors and suppose that $a$ expands $U$. 
Similarly suppose $b\in G$ is the product of the identity
in the first two factors and nontrivial positive diagonal matrices in the last two factors (hence $b$ commutes with $U$). Let $V$ be the product of the upper unipotent subgroup
in the last two factors and suppose that $b$ expands~$V$.

The results of this paper cover much more generally the actions of subgroups  of the full diagonal group that have rank at least two. The concrete choice of acting group $A=\langle a,b\rangle$ given here is a simple example of a case which we treat here but is not handled by the low-entropy method of~\cite{Lindenstrauss-03,Einsiedler-Katok-Lindenstrauss,Low-entropy}. 

Similarly to the low entropy method, in order to study $\mu$
we wish to use dynamics along the direction~$U$ --- a direction in which $\mu$ is neither invariant nor quasi-invariant! However, whereas in the low entropy method we use it in a way inspired by Ratner's work in unipotent dynamics, here we do it in a rather different way: our argument is closer to the way positive entropy for quantum limits was established by Bourgain and the second named author \cite{Bourgain-Lindenstrauss}, combined with some ideas of Boshernitzan \cite{Boshernitzan} regarding rates of recurrence.

\subsubsection{Nontrivial leafwise measures}
Crucial to our approach are the leafwise measures for certain unipotent subgroups (like $U$ and $V$), see \S \ref{sec:prelim}--\ref{anewconstruct} for their basic properties and a detailed construction.

Assume that~$a\in A$ has positive entropy with respect to an~$A$-invariant
and ergodic measure~$\mu$ on~$X$. 
By standard arguments the positive entropy assumption translates to $U$ having nontrivial leafwise measures $\mu_x^U$ a.s. However this does not give directly any information about the leafwise measures $\mu_x^V$ nor on the entropy of $\mu$ under $b$.

\subsubsection{Polynomial recurrence}
Even though $\mu$ is not $U$-invariant, one can obtain
fairly general quantitative 
recurrence statement
once the leaf-wise measures~$\mu_x^{U}$ are nontrivial a.s., analoguous to the above mentioned results \cite{Boshernitzan} of Boshernitzan.

In fact we will show (cf.\ Proposition~\ref{simplerrecurrence}) that for a.e. $x\in X$, if $n$ is large enough there are ``$\mu_x^U$-typical'' elements $u\in U$
of size at most $e^n$ so that $u.x$ is in a ball of radius 
$e^{-\kappa n}$ around $x$ for some $\kappa>0$ (depending on the dimension 
of the space and the entropy of $a$).

\subsubsection{Using the arithmeticity}\label{SMB-step}
The above mentioned recurrence to shrinking balls (cf. Proposition~\ref{simplerrecurrence}) is not strong enough for our needs. Instead we wish to prove positive entropy also under $b$, and ultimately getting additional invariance for $\mu$.

Let us first try to establich that $\mu$ has to have positive entropy also under $b$. Assume in contradiction $h_\mu(b)=0$.
Then the measure of the Bowen balls of $b$
shrinks only sub-exponentially. Combining this with the polynomial recurrence techniques of Proposition~\ref{simplerrecurrence} leads to
an unreasonably high speed for the recurrence claim -- unreasonable enough to be in contradiction to the Diophantine properties of elements in the arithmetic lattice $\Gamma$. 
Thus $b$
must have positive entropy. This step is somewhat similar to 
the proof in~\cite{Bourgain-Lindenstrauss}.

\subsubsection{Unipotent invariance}\label{conditioning}
To upgrade the production of additional directions where we have positive entropy to unipotent invariance we employ a crucial, but quite simple, observation: 
Using the product structure of the leafwise measure (see \S\ref{sec;product} below or \cite{Einsiedler-Katok}, \cite[\S 6]{Lindenstrauss-03} and \cite[\S 8]{Pisa-notes} for more details) the above argument can be carried out conditionally with respect to the 
measurable factor of the $A$ action on $\Gamma\backslash G$
determined by the~$\sigma$-algebra generated by the function~$x\mapsto\mu_x^{V}$ (we view these leafwise measures $\mu_x^{V}$ as elements in a space of locally finite measures on $V$, and $A$ acts on such measure through the action of $A$ on $V$ by conjugation). 
In fact we establish that $\mu$ has positive entropy with respect to~$ b  $ even if we 
condition on this~$\sigma$-algebra. This
implies, by standard properties of entropy and 
leaf-wise measures,  the statement that there are many 
points on typical~$V$-orbits that have the same leafwise measure for $V$. 
This in turn gives that~$\mu_x^{V}$ is invariant under a nontrivial subgroup $I_x<V$. Using $A$-invariance of $\mu$, this invariance subgroup $I_x$ must satisfy an equivariance property and Poincar\'e recurrence. Using that $A$ is of class-$\mathcal{A}'$ and $A$-ergodicity of $\mu$ we obtain that $I_x=I<V$ is a.s.\ constant, which finally implies that $\mu$ must be invariant under $I$. 

Now that we have established that $\mu$ is invariant under a non-trivial unipotent group we can apply Ratner's measure classification theorem for measures invariant under unipotent flows, or more generally the $S$-arithmetic extensions of this theorem (\cite{Ratner-Annals,Ratner-padic, Margulis-Tomanov, Margulis-Tomanov-almost-linear}). Analysing the
possible outcomes of the measure classification theorem using the invariance under $A$ allows us to establish Theorem~\ref{sl2-thm-final}.

\section{Preliminaries and notation}\label{sec:prelim}

We will recall here briefly the main properties of leaf-wise measures and their connection to entropy 
for diagonalizable actions on homogeneous spaces and refer to~\cite{Pisa-notes}
for the details. We will also describe an alternative construction
of the leaf-wise measure in \S\ref{anewconstruct}. Given two measures~$\nu_1,\nu_2$
we write~$\nu_1\propto\nu_2$ if there exists some~$c>0$ with~$\nu_1=c\nu_2$. 
For a given measure $\nu$ we write $[\nu]$ for the equivalence class
with respect to $\propto$.

\subsection{The space}\label{sec: the space}
Let~$G=\GG(\Q_S)=\prod_{v\in S}\GG(\Q_v)$ be 
the group of $\Q_S$-points of a~$\Q$-group~$\GG$. 
Let~$\Gamma<G$ be an arithmetic lattice arising from the~$\GG$, 
i.e.\ 
setting~$\order=\Z[\frac1p:p\in S]$ the lattice~$\Gamma$ is commensurable with~$\GG(\order)$ diagonally embedded
into~$\GG(\Q_S)$. 
Let $X=\Gamma\backslash G$ be the corresponding $S$-arithmetic quotient. The group $G$ acts on $X$ by right translations, i.e.~$g.x=xg^{-1}$.
Now fix some  left invariant  metric~$d_G(\cdot,\cdot)$ on~$G$
and recall that it induces a metric
\[
 d(\Gamma g_1,\Gamma g_2)=\inf _{\gamma_1,\gamma_2\in\Gamma} d_G(\gamma_1 g_1,\gamma_2g_2)=
\inf _{\gamma\in\Gamma} d_G(\gamma g_1,g_2)
\]
for~$\Gamma g_1,\Gamma g_2\in X$. We will write~$B_r^G$ for the ball
of radius~$r>0$ in~$G$ centered at the identity and~$B_r(x)=B_r^X(x)=B_r^G.x$ for the ball
of radius~$r>0$ centered at~$x\in X$. 

Since~$\Gamma$ is a discrete subgroup there exists for every~$x\in X$ some~$r>0$, called an \emph{injectivity radius at~$x$}, such that~$g\in B_r^G\mapsto xg\in X$ is injective and an isometry.
If $H<G$ is a subgroup we will always use the restriction of $d(\cdot,\cdot)$ to $H\times H$ as the left-invariant metric on $H$.
Finally recall~$g.x=xg^{-1}$ defines an action of~$g\in G$ on~$x\in X$.

The discussion up to and including \S\ref{secunipotent} is valid for general algebraic $\Q$-groups $\GG$,
and up to and including \S\ref{recurrencesection} 
with only minor modifications also for general algebraic groups $\GG$
over global fields with positive characteristic.
In \S\ref{sec:not so fast} and \ref{easy-end} essential use is made of the rather strong assumptions we make on $\GG$, namely that it is 
a $\Q$-almost simple form of $\operatorname{SL}_2^k$ or of~$\PGL_2^k$.

\subsection{Leaf-wise measures}\label{characterizingprops}
Let~$\mu$ be a locally finite measure on~$X$ and let~$H<G$ be a closed subgroup. Assume that for~$\mu$-a.e.~$x\in X$ the map~$h\in H\mapsto h.x=xh^{-1}$ is injective. In that case there exists 
a $H$-invariant set~$X'\subset X$ of full measure, and for every~$x\in X'$ a 
\emph{leaf-wise measure}~$[\mu_x^H]$ that is a proportionality class of locally finite measure on~$H$, and ``describes~$\mu$ along
the~$H$-orbit of~$x$'' in the following way: Let~$Y\subset X$ be measurable with~$\mu(Y)<\infty$
and let~$\mathcal A$ be  a countably generated~$\sigma$-algebra on~$Y$ that is~\emph{$H$-subordinate}, 
i.e.\ for a.e.~$y\in Y$ the atom has the form~$[y]_{\mathcal A}=V_y. y$
for some bounded neighborhood~$V_y\subset H$ of the identity in $H$. 
We will refer to $V_y$ as the \emph{shape of the atom}. 
Then for a.e.~$y\in Y$ the proportionality class of the conditional measure on atoms of $\mathcal A$ satisfies
\[
 [\mu_y^{\mathcal A}]=[\mu_y^H|_{V_y}].y,
\]
i.e.~we restrict the leaf-wise measure to the shape of the atom, push the restriction forward
under the orbit map~$h\in H\mapsto h.y$ and obtain the proportionality class of the conditional measure.

On occasion it will be convenient to pick a specific locally finite measure from the proportionality class. 
We will exhibit later a strictly positive function $f_H$ on $H$ so that for a.e.~$y$, \  $\int f_H d\nu < \infty$  for any (or some) $\nu \in  [\mu_y^H]$, and then chose~$\mu_x^H$ as the element of the proportionality class $[\mu_y^H]$  so that $\int f_H d\mu_y^H =1$. Even before we construct this $f_H$ we will use  $\mu_y^H$ to denote an element in the proportionality class $[\mu_y^H]$, taking care to use it only when the actual proportionality constant is immaterial.

A key property of the leaf-wise measures is given by the 
\emph{compatibility formula}
\begin{equation}\label{shift-formula}
  [\mu_{h.y}^H.h]=[\mu_y^H]
\end{equation}
whenever~$h\in H$ and~$y\in X'$; this simply means that the (somewhat pathological) proportionality class of measures on $X$ obtained by taking for any $y \in X'$ the pushforward of $\mu_y^H$ under the map $h \mapsto h.y$ is constant along $H$-orbits.
Additional important properties are that~$\mu$ is~$H$-invariant if and only if~$\mu_x^H$ is a left Haar measure
on~$H$ for a.e.~$x\in X$, and that if~$\mu$ is a finite measure then~$\mu$ is~$H$-recurrent
if and only if~$\mu_x^H(H)=\infty$ for a.e.~$x$ (see \S\ref{recurrencesection}
where we will study refinements of this fact).

\subsection{Dynamics of~\texorpdfstring{$A$}{A}}\label{sec:dynamicsintro}
Let us now fix some closed abelian class-$\cA'$ subgroup~$A<G$ (see Definition~\ref{def:class-A}),
and let~$\mu$ an~$A$-invariant and ergodic probability measure on~$X$.
The dynamics of~$a\in A$ on~$X$ can locally be modelled by the conjugation map~$\theta_a(g)=aga^{-1}$ for~$g\in G$.

For any~$a\in A$ let~$G_a^+=\{g\in G\mid a^nga^{-n}\to e$ as~$n\to-\infty\}$
denote the full unstable horospherical subgroup (where~$e$ denotes the identity element of~$G$). Let~$\mathfrak g_a^+$ be
its Lie algebra.
More precisely, projecting~$a$ to~$G_v$ for some~$v\in S$ we obtain~$a_v\in G_v$ and
can define the group~$(G_v)_{a_v}^+=G_v\cap G_a^+$. This group
equals the group of~$\Q_v$-points of a unipotent algebraic group.
As such it has a Lie
algebra over the field~$\Q_v$. The group~$G_a^+$ equals the direct product~$\prod_{v\in S}G_{a_v}^+$
and the Lie algebra of~$G_a^+$ is defined as the direct sum of the Lie algebras of~$(G_v)_{a_v}^+$,
we will consider it as a Lie algebra over the ring~$\Q_S$.
We also define the full stable horospherical subgroup~$G_a^-=G^+_{a^{-1}}$ and note that the above and following discussion holds
equally well for~$G_a^-$.

We will frequently consider \emph{Zariski closed unipotent subgroups}~$U<G_a^+$, 
i.e.~subgroups that are themselves
products~$U=\prod_{v\in S}U_v$ where~$U_v$ is the group of~$\Q_v$-points
of a Zariski-closed subgroup of~$\GG_{a_v}^+$ defined over~$\Q_v$
for every~$v\in S$. 

\medskip

Let~$a_0\in A$ and~$U<G_{a_0}^+$ be an~$A$-normalized Zariski closed unipotent subgroup. Then 
invariance of~$\mu$ under~$A$ implies for the leaf-wise measures with respect to~$U$
the equivariance property
\begin{equation}\label{eq:equivariance}
 [\mu_{a.x}^U]= (\theta_{a})_*[\mu_{x}^U] 
\end{equation}
for a.e.~$x$, i.e.\ the leaf-wise measure at~$a.x$ is up to a scalar
the pushforward of the leaf-wise measure at~$x$ under the conjugation map $
\theta_{a}$ by~$a\in A$ 
restricted to~$U$.

The \emph{entropy contribution} of~$U<G_a^+$ at~$x$ for~$a\in A$
is defined by
\begin{equation}\label{volume decay}
\vol_\mu(a, U,x)=\lim_{|n|\rightarrow\infty}
\frac{\log\mu_{x}^U \bigl(\theta_a^n(B_1^ U)\bigr)}{n}.
\end{equation}
This limit exists a.s., defines an~$A$-invariant function, and so is constant a.e.~by ergodicity of~$\mu$.
The almost sure value is the entropy contribution of~$U$ and is denoted by~$h_\mu(a,U)$. 
We have~$h_\mu(a,U)=0$ if and only if~$\mu_x^{U}$ equals a.s.~the Dirac measure at the identity.
Moreover, if~$U=G_a^+$ then~$h_\mu(a,G_a^+)=h_\mu(a)$. 
This connection is well known and
has been developed in the homogeneous context in \cite[Sect.~9]{Margulis-Tomanov} 
(see also~\cite[Sect.~7]{Pisa-notes} for an introduction).

By the assumption that~$A$ is of class-$\cA'$ there exists for every prime~$p\in S$
a number~$\lambda_p\in\Q_p^\times$ (with~$\|\lambda_p\|_p>1$) and finitely many homomorphisms~$\chi:A\to\Z$
such that the eigenspaces of the adjoint action of~$A$ on~$G_p$ are given by
\[
 \mathfrak g_p^\chi=\{w\in\mathfrak g_p\mid\operatorname{Ad}_a(w)=\lambda_p^{\chi(a)}w\mbox{ for all }a\in A\}.
\]
For the case of~$v=\infty\in S$ we set~$\lambda_v=e$ and allow the homomorphisms~$\chi$ on~$A$ to take real values. 
Hence we can decompose~$\mathfrak g_\infty$ into eigenspaces of the form
\[
\mathfrak g_\infty^\chi=\{v\in\mathfrak g_\infty\mid\operatorname{Ad}_a(v)=\lambda_\infty^{\chi(a)}v\mbox{ for all }a\in A\}
\]
for finitely many homomorphisms~$\chi:A\to\R$. 

We will refer to the homomorphisms~$\chi:A\to\R$ with nontrivial eigenspaces (at some place) as the \emph{Lyapunov weights} 
and to~$\mathfrak g^\chi=\prod_{v\in S}\mathfrak g_v^\chi$ as the \emph{Lyapunov weight spaces}.
Note that a homomorphism $\chi$ of $A$ might have nontrivial eigenspaces at two or more different places if $\chi:A\to\Z$. 
The trivial Lyapunov weight is given by~$\chi(a)=0$ for all~$a\in A$, the corresponding Lyapunov
weight space is the Lie algebra of the centralizer of~$A$. 
Two nontrivial Lyapunov weights~$\chi_1,\chi_2$ are called equivalent if there exists some~$s>0$ such
that~$\chi_2=s\chi_1$. The equivalence classes are called \emph{coarse Lyapunov weights}
and denoted by~$[\chi]$. The sum of the Lyapunov weight spaces over all weights equivalent to~$\chi$
is a Lie algebra over~$\Q_S$, is called the \emph{coarse Lyapunov weight space}, and is denoted by~$\mathfrak g^{[\chi]}$.

We define~$G^{[0]}=C_G(A)$ and for a nontrivial Lyapunov 
weight~$\chi$ we define the corresponding 
\emph{coarse Lyapunov subgroup} 
by~$G^{[\chi]}=\exp\mathfrak g^{[\chi]}$. We denote the corresponding leaf-wise measures by~$\mu_x^{[\chi]}=\mu_x^{G^{[\chi]}}$ 
for a.e.~$x\in X$.
We note that~$G^{[\chi]}$ is a Zariski closed unipotent subgroup of~$G$ for any nontrivial~$\chi$. 

\subsection{Product lemma}\label{sec;product} 
The coarse Lyapunov subgroups are dynamically significant for the dynamics of~$A$ as above. 
One instance of this is the following product structure. Given some~$a\in A$ the unstable horospherical
subgroup is the product of coarse Lyapunov subgroups, i.e.
\begin{equation}\label{productofcoarse}
 G_a^+=\prod_{[\chi]:\chi(a)>0}G^{[\chi]},
\end{equation}
for which we fix some enumeration of the coarse Lyapunov weights
and the corresponding subgroups.
Then by \cite[Thm.~8.4]{EinsiedlerKatokNonsplit} (cf.~also \cite[Thm.~9.8]{Pisa-notes}) there exists a set~$X'$ of full measure such that
\begin{equation}\label{leafwiseproductofcoarse}
 \mu_x^{G_a^+}\propto\prod_{[\chi]:\chi(a)>0}\mu_x^{[\chi]}
\end{equation}
for all~$x\in X'$, where the the right hand side
is defined as the push-forward of the product measure
under the product map in \eqref{productofcoarse}.
An immediate consequence of this is that the entropy contributions add up to the full entropy, i.e.
\begin{equation}\label{eq:entropycontributionsaddup}
h_\mu(a)=\sum_{[\chi]:\chi(a)>0}h_\mu\bigl(a,G^{[\chi]}\bigr).
\end{equation}
In the particular cases we consider here the coarse Lyapunov subgroups $\{G^{[\chi]}:\chi(a)>0\}$ commute which each other, but \eqref{leafwiseproductofcoarse} is valid in general whichever order one imposes on the coarse Lyapunovs $\{[\chi]:\chi(a)>0\}$; when the $\{G^{[\chi]}:\chi(a)>0\}$ do not commute this imposes restrictions on the leaf-wise measures, restrictions that are expoited by the \emph{high entropy method} of establishing additional invariance \cite{Einsiedler-Katok,EinsiedlerKatokNonsplit} as 
presented in \cite[Sect.~8]{Pisa-notes}.

While not stated there explicitly, the above also holds conditionally on an $A$-invariant $\sigma$-algebra: i.e. if $\mathcal C$ is an $A$-invariant $\sigma$-algebra then
\[
h_\mu(a|\mathcal C)=\sum_{[\chi]:\chi(a)>0}h_\mu\bigl(a,G^{[\chi]}\bigm|\mathcal C\bigr)
\]
and for $\mu$-a.e.~$x \in X$, for $\mu_x^{\mathcal C}$-a.e. $y$,
\begin{equation}\label{eq:conditional product of coarse}
 \bigl(\mu_x^{\mathcal C}\bigr)^{G_a^+}_y\propto\prod_{[\chi]:\chi(a)>0} \bigl(\mu_x^{\mathcal C}\bigr)_y^{[\chi]}.
\end{equation}

\section{Fubini construction of leaf-wise measures}\label{anewconstruct}

In this section we present an alternative construction for the leaf-wise measures, 
which simplifies some of the
upcoming proofs.

\subsection{Extended space and countably generated \texorpdfstring{$\sigma$}{sigma}-algebra}\label{extended-space}
Since the type of subgroups~$H<G$ that we are interested in studying often act ergodically
on~$X=\Gamma\backslash G$ (e.g.\ with respect to the Haar measure on~$X$), there is no countably 
generated~$\sigma$-algebra whose atoms are~$H$-orbits. Because of this the 
standard construction of the leaf-wise
measures invokes
the existence of $H$-subordinate~$\sigma$-algebras or $\sigma$-rings in $X$
(as done e.g.\ in~\cite{Lindenstrauss-03,Pisa-notes}). 

We present here an alternate construction of the leafwise measures. We note that in this treatment we do not need to assume
that $H$ is unimodular.
Consider the product space~$X\times H$ and define the~$\sigma$-algebra
\[
 \mathcal C_H=\Psi_H^{-1}\mathcal B_X,
\]
where~$\Psi_H:X\times H\to X$ is the map defined 
by~$\Psi_H(x_0,h_0)=h_0^{-1}.x_0$ for~$(x_0,h_0)\in X\times H$
and~$\mathcal B_X$ is 
the Borel~$\sigma$-algebra of~$X$. Since~$\mathcal B_X$ is countably generated, the $\sigma$-algebra~$\mathcal C_H$ is a countably generated algebra of Borel subsets of $X\times H$. 

To describe the atom of a point in~$X\times H$ we define~$\Delta_H=\{(h,h)\mid h\in H\}$
and let~$\Delta_H$ act on~$X\times H$ by setting
\[
 (h,h).(x_0,h_0)=(h.x_0,hh_0)
\]
for all~$x_0\in X$ and~$h_0,h\in H$. For $h \in H$, let $\Delta(h) =(h,h)\in \Delta_H$.
We note that 
$$
\Psi_H\bigl(\Delta(h).(x_0,h_0)\bigr)=
\Psi_H\bigl(h.x_0,hh_0\bigr)=
(hh_0)^{-1}h.x_0=h_0^{-1}.x_0=\Psi_H(x_0,h_0)
$$
for all $x_0\in X$ and $h_0,h\in H$.
Clearly the atom~$[(x_0,h_0)]_{\mathcal C_H}$ of~$(x_0,h_0)\in X\times H$ 
consists of all~$(x_1,h_1)\in X\times H$ with~$h_0^{-1}.x_0 =h_1^{-1}.x_1$.
Setting $h=h_1h_0^{-1}$ gives $\Delta(h).(x_0,h_0)=(x_1,h_1)$.
This shows that
\[
[(x_0,h_0)]_{\mathcal C_H}=\Delta_H.(x_0,h_0).
\]
In other words the~$\sigma$-algebra~$\mathcal C_H$
on~$X\times H$ has atoms that consist of orbits of the group~$\Delta_H \cong H$.

\begin{figure}[ht]
\centering
\includegraphics[width=8.8cm]{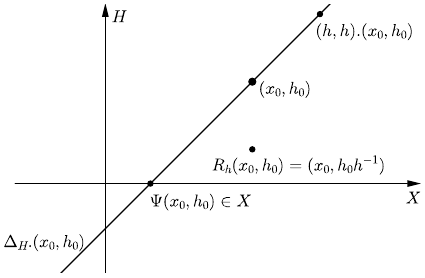}
\caption{We draw $X$ horizontally, $H$ vertically, so that the $\Delta_H$-orbits are represented as diagonal lines. Below we will also use the action $R_h$ of $h\in H$ on $X\times H$ by right translation on the second component.}
\end{figure}

Central to our argument is studying not just the leaf-wise measures but also relative version of this construction. Let~$\fancyZ$ denote any Borel space,  set~$\fancyX=X\times \fancyZ$. We extend the $H$ action to $\fancyX$ by letting~$H$ act on~$X$ as before but trivially on~$\fancyZ$. Consider
the product space~$\fancyX\times H$, and set~$\mathcal C_H=\Psi_H^{-1}\mathcal B_{\fancyX}$, where
\[
 \Psi_H\bigl((x_0,\omega),h_0\bigr)=h_0^{-1}.(x_0,\omega)=(h_0^{-1}.x_0,\omega)
\]
for all~$h_0\in H$ and~$(x_0,\omega)\in\fancyX=X\times \fancyZ$.
We also define~$\Delta(h).((x_0,\omega),h_0)=((h.x_0,\omega),hh_0)$ for all~$(x_0,\omega)\in\fancyX$ and~$h_0,h\in H$, 
so that the atoms for~$\mathcal C_H$ are  of the 
form
\[ 
 \bigl[\bigl((x_0,\omega),h_0\bigr)\bigr]_{\mathcal C_H}=\Delta_H.\bigl((x_0,\omega),h_0\bigr)\subset (X\times\{\omega\})\times H.
\]

\subsection{Extending the measure}
Let now~$\mu$ be a probability measure on~$X$ for which we wish to construct the leaf-wise measures.
We equip the product space $X\times H$ with the product
measure~$\mu\times m_H$ where~$m_H$ 
is the left Haar measure on~$H$.

If~$\phi:X\to \fancyZ$ is a Borel measurable map and we wish to study the leaf-wise measures relative to the 
map~$\phi$, we replace~$X$ by~$\fancyX=X\times \fancyZ$ and replace~$\mu$ by 
the push-forward of~$\mu$ under the map~$\tilde\phi=(\operatorname{id}\times\phi):X\to\fancyX$. 
Below we will always work with this more general setup by considering a probability
measure~$\mu$ on~$\fancyX$. Keeping the notation at its minimum
we will simply write~$x\in\fancyX$, so that~$x$ is a tuple belonging to~$X\times \fancyZ$.

We recall e.g.~from \cite[Ch.~5]{ET-book} that for a Borel probability 
space $(\fancyY,\nu)$ and a~$\sigma$-algebra $\mathcal{C}$ one can define a 
system of conditional measures $y\in \fancyY'\mapsto \nu_y^{\mathcal{C}}$ 
defined on a set $\fancyY'\in\mathcal{C}$
of full measure in a~$\mathcal{C}$-measurable way so that 
$\nu_y^{\mathcal{C}}$ is a probability measure on $\fancyY$ and $\nu=\int_{\fancyY'} \nu_y^{\mathcal{C}}\dee\nu(y)$. 
Moreover, if $\mathcal{C}$ is countably generated one furthermore 
has $\mu_y^{\mathcal{C}}([y]_{\mathcal{C}})=1$, where
$[y]_{\mathcal{C}}=\bigcap_{y\in C\in\mathcal{C}}C$ 
denotes the \emph{atom} of $y\in \fancyY$.
This is also referred to as disintegration of the measure 
$\nu$, and easily 
extends to finite Borel measures on $\fancyY$, where we keep the requirement that the conditional measures 
are probability measures.

We now extend this construction to locally finite measures.
Suppose that $\nu$ is a locally finite 
measure on a~$\sigma$-compact 
locally compact metric space~$\fancyY$ and that ~$\mathcal C$
is a countably generated~$\sigma$-algebra. 
In this case one can define a type of infinite 
conditional measure on the atoms~$[y]_{\mathcal C}$ 
for $y\in\fancyY$ in the following way ---
we will refer to these measures as \emph{fibre measures} in order to distinguish
them from regular conditional measures (which are probability measures by definition). First choose 
a strictly positive continuous integrable function $f_0$ on $\fancyY$ with $\int  f_0 \,d \mu= 1$,
and replace~$\nu$ by the probability measure~$\nu_{\textrm{prob}}$ with~$d\nu_{\textrm{prob}}=f_0d\nu$.
Then find the conditional measures 
$(\nu_{\textrm{prob}})_y^{\mathcal C}$ for~$\nu_{\textrm{prob}}$ 
with respect to $\mathcal{C}$ and define
the fibre measure
\[
 d\nu_y^{\mathcal C}=\frac1{f_0}d(\nu_{\textrm{prob}})_y^{\mathcal C}
\] 
for a.e.~$y\in \fancyY$.
From this definition we see 
that~$\nu_{y_1}^{\mathcal C}=\nu_{y_2}^{\mathcal C}$
whenever~$[y_1]_{\mathcal C}=[y_2]_{\mathcal C}$.
Moreover, using the assumptions on $f_0$ and properties of conditional measures it 
follows that~$\nu_y^{\mathcal C}$
is a locally finite measure on~$\fancyY$. A different choice (say $f_0'$) of such a strictly positive continuous function instead $f_0$ will give a new system of fiber measures (say $\nu_y'$) that a.e.\ are in the same proportionality class as the the original system of fiber measures: whereas $\nu_y^{\mathcal C}$ satisfy that $\int f_0 d\nu^{\mathcal C}_{y}=1$ a.e., $\nu_y'$ satisfy $\int f_0' d\nu_{y}'=1$. Since we will mostly care only about the proportionality class, we keep the dependence of $\nu^{\mathcal C}_{y}$ on the choice of $f_0$ implicit.

If~$K\subseteq \fancyY$ is a subset of finite~$\nu$-measure, then the conditional measure
of~$\nu|_K$ with respect to~$\mathcal C$
can be obtained by taking the normalized restriction of the fibre measures, i.e.
\[
 (\nu|_K)_y^{\mathcal C}\propto \nu_y^{\mathcal C}|_K
\]
for a.e.~$y\in K$. This property characterises the proportionality class of the fibre measures
up to sets of zero measure.

Moreover, if~$T:\fancyY\rightarrow \fancyY$ is measure-preserving, or more generally if~$T_*\nu\propto\nu$,
then
\begin{equation}\label{easytccformula}
 T_*\nu_y^{T^{-1}\mathcal C}\propto \nu_{Ty}^{\mathcal C}
\end{equation}
for a.e.~$y$.

\begin{proposition}\label{proposition about sub sigma algebras}
    If $\mathcal C'\supseteq\mathcal C$ are two countably generated $\sigma$-algebras, then for $\nu$-a.e. $y \in \fancyY$ and for $\nu_y^{\mathcal C}$-a.e.\ $z$ it holds that $
(\nu_y^{\mathcal C})_z^{\mathcal C'} \propto \nu_z^{\mathcal C'}$. In particular, if~$\mathcal C'$ satisfies
that the~$\mathcal C$-atoms are countable unions of~$\mathcal C'$-atoms, then the 
fibre measures~$\nu_y^{\mathcal C'}$ can be obtained by restricting the fibre measures~$\nu_y^{\mathcal C}$
to the atoms for~$\mathcal C'$.
\end{proposition}

\begin{proof}
The first part of the statement follows readily from the analoguous statement for conditional measures, e.g. \cite[Prop.~5.20]{ET-book}. The second statement follows from the first since for countable partitions, the fibre measures are simply the proportionality class of the restrictions.  
\end{proof}

\subsection{Fibre measures for the product}
Combining the above preparations we set~$\fancyY=\fancyX\times H$,
let~$\mu$ be a probability measure on~$\fancyX$, set $\nu=\mu\times m_H$
and obtain the fibre measures
\[
\nu_{(x_0,h_0)}^{\mathcal C_H}= (\mu\times m_H)_{(x_0,h_0)}^{\mathcal C_H}\mbox{ on the atoms }[(x_0,h_0)]_{\mathcal C_H}=\Delta_H.(x_0,h_0)
\] 
of a.e.~$(x_0,h_0)\in \fancyX\times H$. 
It will be helpful to give a more concrete choice of integrable function $f_0$ in the construction of fibre measures: we take $f_0$ to be a strictly positive continuous integrable function on $H$, and implicitly identify it with the corresponding function on $\fancyX\times H$ depending only on the second coordinate.

The 
right action of~$H$ defined by
\[
 R_{h}(x_0,h_0)=(x_0,h_0h^{-1})
\]
for~$(x_0,h_0)\in \fancyX\times H$ and~$h\in H$ commutes with the the left action of $\Delta_H$ and
maps~$\mu\times m_H$ to a multiple of itself.  Moreover, we have $\Psi_H\bigl(R_h(x_0,h_0)\bigr)=h.\Psi_H(x_0,h_0)$ 
for all~$(x_0,h_0)\in\fancyX\times H$, 
which implies in particular~$R_{h}^{-1}\mathcal C_H=\mathcal C_H$. 
This allows us to use~\eqref{easytccformula} to see
that
\begin{equation}\label{goodprop} 
 (R_{h})_*(\mu\times m_H)_{(x_0,h_0)}^{\mathcal C_H}\propto(\mu\times m_H)_{(x_0,h_0h^{-1})}^{\mathcal C_H}
\end{equation}
for any fixed~$h\in H$ and~$\mu\times m_H$-a.e.~$(x_0,h_0)$,
where the null set may depend on~$h$ and a priori is not a product set.
We will also use $R_{h}$ to denote the right action $R_h (h_0) = h_0 h^{-1}$ on $H$; hopefully it should be clear from the context which action is used at any particular point; note that when identifying a function $f$ on $H$ with functions on $\fancyX\times H$ depending only on the second coordinate, both possible interpretations of $f\circ R_h$ coincide.

\subsection{A cleaner version of the fibre measures}
Recall that by construction 
\begin{equation}\label{eq:intisnormalized}
    \int f_0 \, d  (\mu\times m_H)_{(x_0,h_0)}^{\mathcal C_H} =1\qquad\text{a.e.}
\end{equation}
We assume without loss of generality that our metric on $G$ is proper and define
\begin{equation}\label{eq:mindefinition}
f_H(h)=\inf_{h'\in B_1^H}f_0(hh'). 
\end{equation}
By our assumptions on $f_0$ it follows that $ f_H$ is also a strictly positive integrable and continuous function on $H$.

The next lemma provides a cleaner version of \eqref{goodprop}.
Let $p:\fancyX\times H\rightarrow H$ be the projection $p(x,h)=h$ for $(x,h)\in \fancyX\times H$, and recall our notation $\Delta(h)=(h,h)$ for~$h\in H$. In these notations
 \begin{equation}\label{eq;Deltaisinvp}
 \Delta\bigl(R_{h_0}\circ p(x_1,h_1)\bigr).(x_0,h_0)=(x_1,h_1)
 \end{equation}
for all~$(x_1,h_1)\in\Delta_H.(x_0,h_0)$ and~$(x_0,h_0)\in\fancyX\times H$. 

Define
\begin{equation}\label{eq:shorthand12}
  \mu_{(x_0,h_0)}=(R_{h_0}\circ p)_*(\mu\times m_H)_{(x_0,h_0)}^{\mathcal C_H}.
\end{equation}
Since~$p$ is a 
homeomorphism when restricted to the~$\Delta_H$-orbit, the measure $\mu_{(x_0,h_0)}$
is a locally finite measure on~$H$. Moreover, due to~\eqref{eq;Deltaisinvp}
the push-forward of~$\mu_{(x_0,h_0)}$ under the map~$h\in H\mapsto \Delta(h).(x_0,h_0)$
is again~$ (\mu\times m_H)_{(x_0,h_0)}^{\mathcal C_H}$. 

Also note that $p\circ R_h=R_h\circ p$ for $h\in H$ (here $R_h$ denotes the action of $H$ on $\fancyX \times H$ on the left hand side and the action of $H$ on itself by right translations on on the right hand side).
Thus applying $R_{h_0h^{-1}}\circ p$
to both measures in equation~\eqref{goodprop}
we see that equation~\eqref{goodprop}
can then be written more succinctly as
\begin{equation}\label{goodprop2} 
 \mu_{(x_0,h_0)}\propto \mu_{(x_0,h_0h^{-1})},
\end{equation}
which holds for any fixed~$h\in H$ and~$\mu\times m_H$-a.e.~$(x_0,h_0)\in\fancyX\times H$.

\begin{lemma}
\label{fubinilemma}
	There exists a
	measurable conull set $\fancyX'\subset \fancyX$ 
 and a choice of the fibre 
 measures~$(\mu\times m_H)_{(x_0,h_0)}^{\mathcal C_H}$ on~$\fancyX'\times H$ 
 so that:
 \begin{enumerate}
 \item equations \eqref{goodprop} and \eqref{goodprop2} hold for all~$x_0\in\fancyX'$ and all~$h_0,h\in H$;
\item $\int_H f_H d \, \mu_{(x_0,h_0)} <\infty$ for all~$x_0\in\fancyX'$ and all~$h_0\in H$.
\end{enumerate}
\end{lemma} 

\begin{proof}
Combining \eqref{eq:intisnormalized} and 
our definition \eqref{eq:shorthand12} we have
\[
 \int_H f_0(hh_0) d \mu_{(x_0,h_0)}(h)=1,
\]
for a.e.~$(x_0,h_0)$. For $f_H$ as in \eqref{eq:mindefinition} we therefore obtain
that
\begin{equation}\label{eq:f1intable}
    \int_Hf_Hd \mu_{(x_0,h_0)}<\infty
\end{equation}
for a.e.~$(x_0,h_0)\in \fancyX\times B_1^H$.

Let $H'<H$ be a countable dense subgroup. 
	Then there exists a measurable conull 
	set~$\fancyY'\subseteq\fancyX\times H$ that is invariant
	under the right action of~$H'$ on the second factor
  and such that~\eqref{goodprop2} holds for all~$(x_0,h_0)\in \fancyY'$ and $h\in H'$
  and that $\eqref{eq:f1intable}$
  holds for all $(x_0,h_0)\in \fancyY'\cap(\fancyX\times B_1^H)$. 
	Recall that the right action of~$H'$ on~$H$ is ergodic with respect to~$m_H$. Hence the~$H'$-invariant
	sets~$\fancyY_{x_0}'=\{h_0\in H\mid (x_0,h_0)\in \fancyY'\}$ must be null or 
	conull for every~$x_0\in \fancyX$. 
	Using Fubini's theorem we may remove a null set from~$\fancyY'$ and assume that~$\fancyY_{x_0}'$
	is conull for any~$(x_0,h_0)\in \fancyY'$.

For $(x_0,h_0)\in \fancyY'$ there exists now some $h\in H'$ so that $h_0h^{-1}\in B_1^H$. Hence we may use \eqref{goodprop2} and \eqref{eq:f1intable} at $(x_0,h_0h^{-1})$ to see that \eqref{eq:f1intable}
also holds at any $(x_0,h_0)\in \fancyY'$. 
Therefore it makes sense to define the measure
\[
 \tilde{\mu}_{(x_0,h_0)}=\left(\int f_Hd\mu_{(x_0,h_0)}\right)^{-1}\mu_{(x_0,h_0)}
\]
for all $(x_0,h_0)\in \fancyY'$.
We note that
$(x_0,h_0)\in \fancyY'\mapsto \tilde{\mu}_{(x_0,h_0)}$
takes values in a compact metric space
(consisting of all measures $\nu$ with $\int_Hf_H d\nu\leq 1$ and equipped with the
weak* topology). Using the ergodicity of the right
action of $H'$ on $H$ again we see that $h_0\in H\mapsto\tilde{\mu}_{(x_0,h_0)}$
is almost surely constant for any fixed $x_0$. Equivalently we have that $\tilde{\mu}_{(x_0,h_0)}$ equals 
\[
 \mu_{x_0}=m_H(B_1^H)^{-1}\int_{B_1^H}\tilde{\mu}_{(x_0,h)}dm_H(h) 
\]
for a.e.~$(x_0,h_0)\in \fancyY'$. 

Using Fubini's theorem we can fix some $\tilde{h}\in H$ so that $\fancyX'=\{x\mid (x,\tilde{h})\in \fancyY'\}$ is conull in $\fancyX$.
Summarizing the above we have that $\mu_{(x_0,h_0)}\propto\tilde{\mu}_{(x_0,h_0)}$ for all $(x_0,h_0)\in \fancyY'$,
and that $\tilde{\mu}_{(x_0,h_0)}=\mu_{x_0}$
for a.e.~$(x_0,h_0)\in \fancyY'$. 
This allows us to replace the original $\mu_{(x_0,h_0)}$ defined for $(x_0,h_0)\in \fancyY'$
by $\mu_{x_0}$ defined for $(x_0,h_0)\in\fancyX'\times H$.
\end{proof}

\subsection{Completing the construction}
Lemma~\ref{fubinilemma} is saying that, up to proportionality,
the fibre measure depends on the second coordinate~$h\in H$ only in terms of a shift
of the position of the measure. We can now use this property to obtain an
alternative construction of the leaf-wise measure.

\begin{proposition}[Fubini-construction of leaf-wise measures]\label{straighten}
 Let~$H<G$,~$\fancyX$,~$\mathcal C_H$, and~$\mu$ be as above. 
Assume in addition that for~$\mu$-a.e.~$x\in H$ the map~$h\in H\mapsto h.x$
is injective.
Let~$\fancyX'\subset \fancyX$ and $\mu_{(x,h)}$ for $(x,h)\in\fancyX'\times H$ be as in Lemma~\ref{fubinilemma}.
Then the measures defined by
\begin{equation}\label{eq;fubiniformula}
\mu_x^H=\mu_{(x,e)}
\end{equation}
for all~$x\in \fancyX'$ satisfy the characterising properties of leaf-wise measures
discussed in \S\ref{characterizingprops}.
\end{proposition}

We note that this definition of the leaf-wise measure has the 
compatibility formula~\eqref{shift-formula} as an immediate corollary:
If~$x,h.x\in\fancyX'$, then~$(x,e)$ and~$(h,h).(x,e)=(h.x,h)$ are in the same~$\Delta_H$-orbit
and as such have proportional fibre measures. 
Applying the push forward 
under~$p$ to this fibre
measure we then obtain using \eqref{eq;fubiniformula},
\eqref{goodprop2}, and \eqref{eq:shorthand12} that
\begin{align*}
 \mu_{h.x}^H&=\mu_{(h.x,e)}\propto\mu_{(h.x,h)}\propto(R_h\circ p)_*(\mu\times m_H)_{(h.x,h)}^{\mathcal C_H}\\
 &\propto (R_h\circ p)_*(\mu\times m_H)_{(x,e)}^{\mathcal C_H}
 \propto\mu_{(x,e)}h^{-1}=\mu_x^Hh^{-1}.
\end{align*}

\begin{proof} [Proof of Proposition~\ref{straighten}]
	Suppose~$Y\subset \fancyX$ is measurable  and 
	that~$\mathcal A$ is an~$H$-subordinate~$\sigma$-algebra on~$Y$
	as in~\S\ref{characterizingprops}. We may assume without loss of generality that
	the injectivity requirement for~$\mu$ in the proposition holds for all~$y\in Y$. 
	Recall that~$V_y$ denotes the ``shape of the atom'' for~$y\in Y$
	in the sense that~$V_y.y=[y]_{\mathcal A}$ and that~$V_y\subset H$ is assumed to be 
	a bounded neighborhood of the identity in $H$ for all $y \in Y$.

    We now consider four $\sigma$-algebras of subsets of $Y \times H$. The first $\sigma$-algebra is the restriction of $\mathcal C_H$ to ${Y \times H}$, for simplicity again denoted by $\mathcal{C}_H$, whose atoms are given by $\Delta_H$-orbits intersected with $Y\times H$. 
    The second is $\mathcal A_{\times H} = \{A \times H : A \in \mathcal A\}$, whose atoms are 
    \[
    [y]_\mathcal{A}\times H=(V_{y}.y)\times H
    \]
    for $y\in Y$. The third is $\mathcal A_{\mathrm{flat}} = \mathcal A \times \mathcal B_H$, whose atoms are
    \[
     [y]_{\mathcal{A}}\times\{h\}=(V_y.y)\times\{h\}
    \]
    for $(y,h)\in Y\times H$. And finally $\mathcal A_H=\mathcal C_H \vee \mathcal A_{\times H}$, whose atoms are  \begin{equation}\label{eq:fourthatom}
    [y,h]_{\mathcal{A}_H}=\Delta(V_{y}).(y,h)
    \end{equation}
    for $(y,h)\in Y\times H$.

    Our aim is to calculate the fibre measures
    $(\mu\times m_H)^{\mathcal A_H}_{(y,h)}$ on the atoms in \eqref{eq:fourthatom} for a.e.~$(y,h)\in Y\times H$ using Proposition \ref{proposition about sub sigma algebras} in two different ways.
    The first application is quite straight-forward:
    Recall that the group $H$ is second countable and 
    that for any $y \in Y$ the set $V _ y$ contains an open neighborhood. Hence each atom of $\mathcal{C}_H$ (here denoting a $\sigma$-algebra of subsets of $Y \times H$) contains only countably many atoms of $\mathcal{A}_H$. By the second part of Proposition~\ref{proposition about sub sigma algebras} and \eqref{eq:fourthatom} it follows that
\begin{equation}\label{eq:firstdoublecond}
(\mu \times m _ H)^{\mathcal{A} _ {H}}_{(y,h)} \propto (\mu \times m _ H)^{\mathcal{C}_H}_{(y,h)} |_ {\Delta (V _ y).(y,h)}
\end{equation}
for a.e.\ $(y,h)\in Y\times H$. 

For the second application of Proposition \ref{proposition about sub sigma algebras} we wish to use the $\sigma$-algebras $\mathcal A_{\times H}$ and $\mathcal A_{\mathrm{flat}}$. 
For this notice that the definition of fibre measures implies 
for $\mu\times m_H$ a.e.~$(y_0,h_0)\in Y\times H$ that
\[
(\mu \times m _ H) ^ {\mathcal{A} _ {\times H}}_{(y_0,h_0)} \propto \mu ^ {\mathcal{A}} _ {y_0} \times m _ H. 
\]
  Fix such a $(y_0,h_0) \in Y \times H$. We now argue
  conditionally on the atom $[y _ 0] _ {\mathcal{A}} \times H=(V_{y_0}.y_0)\times H$ and with respect to this fibre measure. Then the map 
 \[
  \Xi: (h'.y_0,h'h)=\Delta(h').(y_0,h) \mapsto (h'y_0, h)
 \]
  is a map from the atom $[y _ 0] _ {\mathcal{A}} \times H$ to itself preserving $(\mu \times m _ H) ^ {\mathcal{A} _ {\times H}}_{(y_0,h_0)}$. Moreover, on $[y _ 0] _ {\mathcal{A}} \times H $ our map $\Xi$ maps atoms of the $\sigma$-algebra  $\mathcal{A} _ H$
  precisely to atoms of $\mathcal{A}_{\mathrm{flat}}$ -- equivalently the preimage of $\mathcal{A}_{\mathrm{flat}}$ equals $\mathcal{A}_H$. 
Finally also notice that
\[
(\mu \times m _ H) ^ {\mathcal{A}_{\mathrm{flat}}} _ {(y,h)} \propto \mu ^ \mathcal{A} _ y \times \delta _ {h}
\]
for a.e.~$(y,h)\in Y\times H$, where $\delta_{h}$ denotes the delta measure at $h$.  By \cite[Cor.~5.24]{ET-book} and Proposition~\ref{proposition about sub sigma algebras} it follows a.s.\ that for $(\mu \times m _ H) ^ {\mathcal{A} _ {\times H}}_{(y_0,h_0)}$-a.e.\ $(y, h)=(h'.y_0,h)$ with $h'\in V_{y_0}$ we have
\begin{equation}\label{eq:seconddoublecond}
\Xi_* ((\mu \times m _ H) _ {(y,h)} ^ {\mathcal{A}_H}) = 
\mu ^ \mathcal{A} _ y \times \delta _ {(h')^{-1}h}
.\end{equation}

By construction of $\mu_y^H=\mu_{(y,\cdot)}$ for $y\in \fancyX'\cap Y$ in and before Lemma \ref{fubinilemma}, the fibre measure on the right of \eqref{eq:firstdoublecond} can be obtained by restricting $\mu_y^H$
to $V_y$ and pushing it forward under the map $h_1\in V_y\mapsto \Delta(h_1).(y,h)$. By \eqref{eq:firstdoublecond} what we obtain is proportional to 
$(\mu \times m _ H) _ {(y,h)} ^ {\mathcal{A}_H}$. Applying the push forward for $\Xi$ to the latter gives a measure proportional to $\mu ^ \mathcal{A} _ y \times \delta _ {(h')^{-1}h}$ by \eqref{eq:seconddoublecond}. Finally applying
the projection to $\fancyX$ we obtain a measure proportional to $\mu ^ \mathcal{A} _ y$.
Composing these maps we obtain that the push-forward 
of $\mu_y^H|_{V_y}$ under the map $h\in V_y\mapsto h.y\in [y]_{\mathcal {A}}$ equals $\mu_y^{\mathcal A}$ up to proportionality. 
\end{proof}

\section{Recurrence along~\texorpdfstring{$U$}{U}}\label{recurrencesection}

Poincar\'e recurrence is one of the basic results of measure preserving dynamics on probability spaces. 
In 1993 Boshernitzan \cite{Boshernitzan}
refined this basic principle by showing recurrence at a polynomial rate 
when the underlying space has finite Hausdorff dimension.

The study of recurrence along directions that are not known to preserve the measure class
played a crucial role in \cite{Lindenstrauss-03}, but there no finer information about the recurrence
was needed. 
Also in \cite{Einsiedler-Lindenstrauss-joinings-2} this recurrence property was used in an essential way along the unipotent
coarse Lyapunov subgroup but again the speed of the recurrence did not play any role. 

\subsection{\texorpdfstring{$U$}{U}-returns to sets of small diameter}\label{sec:ureturns}
We combine here
these two lines of thoughts and study the speed of recurrence along coarse Lyapunov subgroups
and relate this to the entropy of an element that commutes with the given coarse Lyapunov subgroup. 
For this it will be convenient to define the notion of returns to subsets of locally homogeneous spaces
of ``small diameter''.

Let~$X=\Gamma\backslash G$
which we endow with the metric induced from a left-invariant metric on~$G$.
We also use this metric on subgroups of $G$ and assume that $B_r^G$ has compact closure for any $r>0$.
Recall that this implies that for every~$x_0\in X$
there exists an \emph{injectivity radius}~$r>0$ for which the map~$g\in B_r^G\mapsto x_0g\in X$ is an isometry. 
We say that a 
$Y\subset X$ has  \emph{small diameter} 
if there exists $x_0\in X$ and $r\in(0,\frac14)$
with $Y\subseteq B_r^X(x_0)$ so that~$10r$ is an injectivity radius at~$x_0$. Note
that~$9r$ is then an injectivity radius for all points~$x\in B_{r}^X(x_0)$ and moreover $Y\subseteq B_{2r}^X(x)$.

Given a subset~$Y$ of small diameter,
a subgroup~$U<G$, and subset~$Q\subseteq U$ with compact closure, 
a~\emph{$U$-return within~$Q$ of~$x\in Y$ to~$Y$}
is an element~$u\in Q\setminus B_{2r}^U$ such that~$u.x\in Y$. 

\subsection{A simpler version}

Let us start as a warm-up with a simpler version of the polynomial recurrence phenomenon along
a unipotent subgroup with positive entropy contribution.
For this we assume that the metric on $G$ has the property that $B_1^G$
(and hence all compact subsets of $G$ and $X$) have finite upper box dimension\footnote{Box dimension is defined only on bounded sets, hence the restriction to $B_1^G$. Any other reasonable bounded set would have been equally good.}.
 
\begin{proposition}[Recurrence at polynomial rate]\label{simplerrecurrence}
	Let~$X=\Gamma\backslash G$ be as in Section~\ref{sec: the space} 
	and~$A<G$ be of class $\mathcal{A}'$. 
	Let~$\mu$ be an~$A$-invariant and ergodic probability
	measure on~$X$ and~$X'\subseteq X$ be measurable with positive measure. 
	Let~$a\in A$ and~$U<G_a^+$ be an~$A$-normalized unipotent subgroup with positive entropy contribution for~$a$.
	Then there exists some~$\kappa>0$ such that for~$\mu$-a.e.~$x\in X'$ and every sufficiently large~$n$
	there exists a~$U$-return within~$a^{n}B_1^Ua^{-n}$ of~$x$ to~$X'\cap xB_{e^{-\kappa n}}$.
\end{proposition}

As the set~$a^nB_1^Ua^{-n}$ expands at an exponential rate in~$n$ and we also 
consider the balls of exponentially shrinking radius, the proposition 
	amounts to a polynomial rate of recurrence, where~$\kappa$ controls 
	the degree of the polynomial recurrence. As we will see~$\kappa$
	depends on the entropy contribution of~$U$
	and the dimension of~$X$.

For the proof we will apply the following infinitely often (see also~\cite[\S4]{Lindenstrauss-03}).

\begin{lemma}[Recurrence to a set]\label{rec-to-a-set}
Let~$X=\Gamma\backslash G$ be as in Section~\ref{sec: the space}, 
let~$\mu$ be a probability measure on~$X$, 
and let~$U<G$ be a closed connected subgroup.  
Fix some~$T\geq 1$ and some $F_\sigma$-subset~$Q\subset U$ 
with compact closure. 
Let~$Y$ be an~$F_\sigma$-subset
of small diameter (with corresponding $r>0$ as in the definition of small diameter in~\S\ref{sec:ureturns}) such that
\begin{equation}\label{Q measure inequality}
    \mu_x^U\left(Q\right)\geq T \mu_x^U(B_1^U)
\end{equation} 
for~$x\in Y$. 
Then the set of elements
	\[
Y^{\mathrm{nr}}=
\bigl\{z\in Y\mid (QQ^{-1}\setminus B_{2r}^U).z\cap Y=
\emptyset\bigr\}
	\]
without~$U$-returns within~$QQ^{-1}$ to~$Y$	has measure
\[
\mu(Y^{\mathrm{nr}})\leq T^{-1}\frac{m_U(Q^2)}{m_U(Q)}.
\]
\end{lemma}

\begin{proof}
 Let~$Y=\bigcup_nY_n$ be as in the lemma with~$Y_n\subset X$ being compact.
 Let~$K\subset U$ be compact and let
 \[
  Y_{n,K}^{\mathrm{nr}}=\bigl\{z\in Y\mid K.z\cap Y_n=\emptyset\bigr\}.
 \] 
 Notice that~$z\in Y_{n,K}^{\mathrm{nr}}$ implies that~$K.z$ has positive
 distance to~$Y_n$, which implies that~$Y_{n,K}^{\mathrm{nr}}$ is 
 a relatively open subset of~$Y$. In particular,~$Y_{n,K}^{\mathrm{nr}}$ and
 \[
  Y^{\mathrm{nr}}_K=\bigcap_nY_{n,K}^{\mathrm{nr}}=\bigl\{z\in Y\mid  K.z\cap Y=\emptyset\bigr\}
 \]
 are measurable sets. Writing~$QQ^{-1}\setminus B_{2r}^U$ as a 
countable union of compact sets~$K_\ell$ we see that 
also~$Y^{\mathrm{nr}}=\bigcap_{\ell}Y^{\mathrm{nr}}_{K_\ell}$ 
is measurable.

We define the rectangle $R=X\times(Q^2)$. To prove the lemma we will consider
the measure $\mu\times m_U$ restricted to $R$ and estimate the conditional measures of $Y^{\mathrm{nr}}\times Q$ with respect to $\mathcal{C}_U|_R$, with $\mathcal C_U$ as in \S\ref{extended-space}.
For this we recall from Section \ref{anewconstruct} that the conditional measures $(\mu\times m_U|_R)_{(x_0,u_0)}^{\mathcal{C}_U}$ can be identified as the push forward of the measure $\mu_{x_0}^U$ to the
orbit $\Delta_U.(x_0,u_0)$, after further restricting
it to the atom $[(x_0,u_0)]_{\mathcal{C}_U}\cap R=(\Delta_U.(x_0,u_0))\cap R$,
and normalizing the resulting measure to be a probability measure.

Now fix some $(y_0,u_0)\in Y^{\mathrm{nr}}\times Q$.
Suppose also
\[
\Delta(u).(y_0,u_0)\in (Y^{\mathrm{nr}}\times Q)\cap [(y_0,u_0)]_{\mathcal{C}_U}
\]
for some $u\in U$. 
Looking at the second coordinate we have $uu_0\in Q$, 
which implies $u\in Qu_0^{-1}\subseteq QQ^{-1}$. However, as $y_0$ is
not having any $U$-returns within $QQ^{-1}$ and $u.y_0\in Y$
we conclude that $u\in B_{2r}^U\subset B_1^U$, where $r\in(0,\frac14)$
is as in the definition of small diameter.
Also note that $\Delta(Q).(y_0,u_0)\subset R$. 
We now combine these facts
with the assumption \eqref{Q measure inequality} to
obtain that
\begin{align*}
(\mu\times m_U|_{R})_{(y_0,u_0)}^{\mathcal{C}_U}\bigl(Y^{\mathrm{nr}}\times Q\bigr)
&= \frac{\mu_x^U(\{u\in U: \Delta(u).(y_0,u_0)\in Y^{\mathrm{nr}}\times Q\})}{\mu_x^U(\{u\in U: \Delta(u).(y_0,u_0)\in R\})}\\
&\leq 
\frac{\mu_x^U(B_{2r}^U)}{\mu_x^U(Q)}\\
&\leq T^{-1}
\end{align*}
for all $(y_0,u_0)\in Y^{\mathrm{nr}}\times Q$

Applying now the properties of the conditional measures and using that $\mu$ is a probability measure this implies
\[
 \mu\times m_U (Y^{\mathrm{nr}}\times Q)\leq T^{-1} m_U(Q^2),
\]
which gives the lemma.
\end{proof}

\begin{proof}[Proof of Propositions~\ref{simplerrecurrence}]
	Fix some~$\epsilon>0$. By assumption~$U$ has positive entropy contribution for~$a$. 
	Hence by the definition of
	the entropy contribution in~\eqref{volume decay} there exists some~$\alpha>0$ such that for a.e.~$x$
	there exists some~$N(x)$ with
	\begin{equation}\label{alpha-equation}
    \mu_x^U(a^n B_{1/2}^U a^{-n})	\geq e^{\alpha n}  \mu_x^U (B_1^U)
	\end{equation}
  for all~$n\geq N(x)$.
  We choose~$N_0\geq 1$ and some compact set~$X_\epsilon$
 of measure~$\mu(X_\epsilon)>1-\epsilon$
  such that $\mu_x^U$ is defined and~\eqref{alpha-equation} 
  holds for all~$x\in X_\epsilon$ and all~$n\geq N_0$.
   Let $X'$ be as in the proposition.
   
  Set $Q=a^nB_{1/2}^Ua^{-n}$. Note that by definition of $B_{r}^U$ the set $Q$ satisfies $Q=Q^{-1}$ and $Q^2 \subseteq a^nB_1^Ua^{-n}$.

	Let~$\kappa\in(0,\alpha/(2\dim G))$, with $\dim G$ 
 the box dimension of~$B_1^G$ with respect to our chosen metric $d$.
	Then the 
	compact set~$X_\epsilon$ can be covered for every sufficiently large~$n$
	by~$K_n\leq e^{\frac12\alpha n}$ many balls~$B_{n,k}=
	B_{e^{-\kappa n}/2}^X(x_{n,k})\subset X$ with~$x_{n,k}\in X_\epsilon$ for~$k=1,\ldots,K_n$; we may further assume that all these balls will be of small diameter in the sense of \S\ref{sec:ureturns}.
	
	Finally we apply Lemma~\ref{rec-to-a-set} to each
	\[
	 Y_{n,k}=B_{n,k}\cap X_\epsilon \cap X'.
	\]
	 Hence, we obtain using~\eqref{alpha-equation}
	for every~$n$ and~$k$ that the measure of the set~$Y_{n,k}^{\text{nr}}$
	of all~$x\in Y_{n,k}$ 
	with no~$a^nB_1^Ua^{-n}$-return to~$Y_{n,k}$ 
	is~$\ll e^{-\alpha n}$.
	Taking the union over all~$k=1,\ldots,K_n$ we get
	\[
	\mu\biggl(\bigcup_{k=1}^{K_n}Y_{n,k}^{\text{nr}}  \biggr)\ll e^{-\frac12\alpha n}.
	\]
	Using the easy half of Borel-Cantelli we see that a.e.~$x\in X_\epsilon \cap X'$ 
	only belongs to finitely
	many of the sets~$Y_{n,k}^{\text{nr}}$. For any of these~$x\in X_\epsilon$ 
	and all large enough~$n$ it follows that~$x$
	has~$U$-returns within~$a^nB_1^Ua^{-n}$ to~$B_{e^{-\kappa n}}^X(x)\cap X'$: if $k$ is such that $B_{n,k} \ni x$ then $B_{e^{-\kappa n}}^X(x)$
	 contains the~$B_{n,k}=B_{e^{-\kappa n}\!/2}^X(x_{n,k})$. 
	As $\mu(X_\epsilon)>1-\epsilon$ and~$\epsilon>0$ was arbitrary the proposition follows.	
\end{proof}

\subsection{Multiple recurrence}
Given a subset~$Y$ of small diameter
 and a subset~$Q\subseteq U$ with compact closure
we say that $x\in Y$ has \emph{~$N$ independent $U$-returns within~$Q$ to~$Y$}
if there are~$N$ points~$u_1,\ldots,u_N\in Q\setminus B_{2r}^U$
with~$ d(u_j^{-1},u_k^{-1})\geq 2r$ for~$1\leq j\neq k\leq N$
such that~$u_1.x,\ldots,u_N.x\in Y$ and
\[
 \{u\in Q\mid u.x\in Y\}\subseteq B_{2r}^U\cup\bigcup_{j=1}^N B_{2r}^U u_j.
\]
In this case we
also say that the elements~$u_1,\ldots,u_N$ \emph{represent 
the independent~$U$-returns} of~$x$ within~$Q$ to~$Y$.

We note that for a given~$U$-return~$u.x\in Y$ for some~$x\in Y$ and~$u\in Q$ we may have 
some nearby elements~$u'\in B_{2r}^Uu$ that also satisfy~$u'.x\in Y$, but due to our 
assumptions on the injectivity radius there will not be any~$u'\in (B_{9r}^U\setminus B_{2r}^U)u$
with~$u'.x\in Y$. This spacing property of occurrence of~$U$-returns 
implies that the number~$N$ of independent~$U$-returns 
is independent of the choice
of the elements~$u_1,\ldots,u_N\in Q$ that represent the independent~$U$-returns of~$x$ within~$Q$ to~$Y$.

\subsection{Multiple fast polynomial recurrence}
We now state the version of recurrence at polynomial rate that will be needed for the proof of 
Theorem~\ref{sl2-thm-final}. Since this does not entail any significant complications, we prove it for the setting presented in \S\ref{sec: the space}, where $\GG$ was a general linear algebraic groups defined over $\Q$ (rather than only forms of $\SL_2^k$ or $\PGL_2^k$).  We plan to make use of this more general statement in follow up work that will consider more general cases.  
In the outline \S\ref{outline} we had available a second element~$ b  \in A$ that was in the kernel of the Lyapunov weight corresponding to~$U$.
Unfortunately, in general this element may not exist, e.g.\ if $A\cong\Z^2$ and the coarse Lyapunov
weight defining $U$ has no rational representative. Hence we need to allow more general elements~$b\in A$, which we
will later choose, so that $U$ is in the weak unstable subgroup for $b$, but so that~$b$ does not expand~$U$ 
``too much''. 
We define for any radius~$r>0$ and time parameter~$t>0$ the Bowen $t$-ball (also called the Bowen-Dinaburg~$t$-ball) for~$ b  $
as
\begin{equation}\label{Bowen-Dinaburg}
 D_{r,t}^{ b  }=\bigcap_{k=-\lfloor t\rfloor}^{\lfloor t\rfloor}b^kB_r^Gb^{-k}.
\end{equation}

\begin{theorem}[Fast recurrence or positive entropy]\label{zero-or-fast-pol-rec}
	Let~$X=\Gamma\backslash G$ be as in \S\ref{sec: the space}, and~$A<G$ be a higher rank class $\mathcal{A}'$ subgroup. 
	Let~$\mu$ be an~$A$-invariant and ergodic probability
	measure on~$X$. 
	Let~$U$ be a coarse Lyapunov subgroup 
	with positive entropy contribution with respect to some (hence all) elements in $A$ expanding $U$. Fix such an element $a \in A$ and let $\alpha=h_\mu(a,U)$.
  Let~$ b  \in A$ and let $V\leq G^+_b$ be a product of coarse Lyapunov subgroups.
 Assume $U\cap V=\{e\}$, that $U$ normalizes $V$, and that
  \[
  G^+_b \leq U \ltimes V\leq G^+_bC_G(b).
  \]
  Let~$\mathcal A_V$ be 
 the~$A$-invariant~$\sigma$-algebra generated by the 
function~$x\mapsto\mu_x^{V}$. Then for every~$\kappa\geq 1$ we have either
\begin{itemize}[leftmargin=*,topsep=0pt,before=\vspace{0.5\baselineskip}]
	\item \textup{\bf non-negligible conditional entropy:}
	 $h_\mu( b  |\mathcal A_V)\geq \frac\alpha{4\kappa} $
\end{itemize}
\noindent
or
\begin{itemize}[leftmargin=*,topsep=0pt]
	\item
\textup{\bf fast multiple recurrence:} 
for $\mu$-a.e.~$x\in X$, 
 any sufficiently small~$r>0$, and
for all $n$ large enough we have at least~$e^{\frac\alpha{4} n}$ 
 independent~$U$-returns within~$a^n B_1^U a^{-n}$
of~$x$ to~$xD_{r,\kappa n}^{ b  }$.
\end{itemize}
\end{theorem}

In the case that the  conditional entropy $h_\mu( b  |\mathcal A_V)$ vanishes, the non-negligible entropy case cannot occur and then we are guaranteed to have fast multiple recurrence, which is much stronger than what is given in Proposition~\ref{simplerrecurrence}, 
as we can choose the `degree of the polynomial recurrence' arbitrarily by increasing $\kappa$. We also note that we assume
that~$r$ is sufficiently small only so that~$xD_{r,0}^b=B_r^X(x)$
is a ball of small diameter in the sense of \S \ref{sec:ureturns} and hence the notion 
of independent~$U$-returns is defined for~$x$ 
and subsets of~$xD_{r,0}^b$.

\begin{lemma}\label{lem: implication of product lemma}
    Under the conditions of Theorem~\ref{zero-or-fast-pol-rec}, there is a set of full measure~$X'$ so that $\mu_x^V=\mu_{u.x}^V$ whenever $x,u.x \in X'$.
\end{lemma}

\begin{proof}
 As indicated in \S\ref{sec;product} the leafwise measure for $UV\simeq U\ltimes V$ is a.s.\ the direct product
 of the leafwise measures for $U$ and for $V$. Morever, we may also use the groups in the reversed order, i.e.\ for the product map $\iota:(v,u)\in V\times U\mapsto vu$ we have $\mu_x^{UV}\propto\iota_*(\mu_x^V\times\mu_x^U)$ a.s.
 Indeed if $U\ltimes V=G_b^+$, this is precisely \cite[Thm.~8.4]{EinsiedlerKatokNonsplit} (see also \cite[\S~8]{Pisa-notes}). 
 If instead $U<C_G(b)$, we define $b'=b^na$, choose $n\in\N$ sufficiently large so that $U\ltimes V<G_{b'}^+$, and may again apply \cite[Thm.~8.4]{EinsiedlerKatokNonsplit}. 
  Since $U$ is a coarse Lyapunov subgroup, either $U<G^+_b$ or $U<C_G(b)$.
  
 This product structure implies now the lemma using double conditioning: The leafwise measure for $V$ can be obtained from the leafwise measure for $U\ltimes V$
 by conditioning on $V$-cosets. As the leafwise measure for $U\ltimes V$ is a product, moving from $x$ to $u.x$ within a set of full measure does not change the leafwise measure for $V$.
\end{proof}

 The following can be viewed as a sharper variant of Lemma~\ref{rec-to-a-set}. 

\begin{proposition}[Multiple recurrence to a set]\label{multiple-rec}	
	Let~$X=\Gamma\backslash G$ and~$A<G$ be as in Theorem \ref{zero-or-fast-pol-rec}. Let~$a\in A$ and~$U<G_a^+$ be an~$A$-normalized unipotent subgroup. Let~$\nu$ be
	a probability measure on~$X$.	
 Fix $n\in\mathbb{N}$, $T>1$, and suppose that~$Y\subset X$ is an~$F_\sigma$-subset of small diameter so that
	\begin{equation}\label{alpha-asymp}
   \nu_x^U(a^n B_{1/2}^U a^{-n}) \geq T \nu_x^U(B_1^U)
	\end{equation}
	for all~$x\in Y$. 
	Then for every~$N\in\N$ the set
	\[
	  Y^{\mathrm{fr}}=\bigl\{x\in Y\mid \mbox{$x$ has $<N$ independent~$U$-returns
	within~$a^nB_1^Ua^{-n}$ to~$Y$}\bigr\}
	\]
	of points with ``few returns back to~$Y$'' has measure~$\nu(Y^{\mathrm{fr}})\ll_U NT^{-1}$.
\end{proposition}

\begin{proof} 
Measurability of $Y^{\mathrm{fr}}$ follows
similarly to the argument in the first paragraph of the proof of
Lemma \ref{rec-to-a-set}.
We define the expanded ball
$Q=a^nB_{1/2}^Ua^{-n}=Q^{-1}$, and the rectangle
$R=X\times (Q^2)$. We will again estimate $\nu(Y^{\mathrm{fr}})$
by estimating instead $\nu\times m_U(Y^{\mathrm{fr}}\times Q)$
using the conditional measures of $\nu\times m_U$
restricted to $R$ and with respect to $\mathcal{C}_U$.

We again identify the conditional measure $(\nu\times m_U|_{R})_{(x,u)}^{\mathcal{C}_U}$
with the push-forward of the measure $\nu_x^U$ to the $\Delta(U)$-orbit of $(x,u)$,
restricting to the rectangle $R$, and normalizing the resulting measure
to a probability measure. Note that for $(x,u)\in Y\times Q$ we have $\Delta(Q).(x,u)\in R$. 
Our assumption \eqref{alpha-asymp} therefore gives that
\begin{equation}\label{eq:estimateforR2}
 (\nu\times m_U|_{R})_{(x,u)}^{\mathcal{C}_U}(\Delta(B_{4r}^U).(x,u))\leq T^{-1}
\end{equation}
for all $(x,u)\in Y\times Q$ and $r\in(0,1/4)$. 

Now let $(y_0,u_0)\in Y^{\mathrm{fr}}\times Q$ and $r\in(0,\frac{1}{4})$
as in the definition of small diameter. Let $h_1,\ldots,h_M$ represent
the independent $U$-returns of $y$ within $Q^2$ to $Y$, where $M< N$
by definition of $Y^{\mathrm{fr}}$.
Let $h_0=e$.

Now suppose $(u,u).(y_0,u_0)\in Y\times Q$ for some $u\in U$.
Looking at the second coordinate we see that $uu_0\in Q$, or equivalently
that $u\in Qu_0^{-1}\subseteq Q^2$. Hence we have $u\in B_{2r}^Uh_j$
for some $j\in\{0,1,\ldots,M\}$ by definition of $h_0,\ldots,h_M$ and so
\begin{multline*}
 (\nu\times m_U|_{R})_{(y_0,u_0)}^{\mathcal{C}_U}\bigl((\Delta(U).(y_0,u_0))\cap (Y\times Q)\bigr)\\
 \leq 
 \sum_{j}(\nu\times m_U|_{R})_{(y_0,u_0)}^{\mathcal{C}_U}(\Delta(B_{2r}^Uh_j).(y_0,u_0)),
\end{multline*}
where we only need to sum over those $j$
for which $\Delta(B_{2r}h_j).(y_0,u_0)\cap (Y\cap Q)$ contains some point $(x,u)$.
For those $j$ we can use \eqref{eq:estimateforR2} for this intersection point, which leads to 
\begin{equation}\label{eq:betterestimateR2}
     (\nu\times m_U|_{R})_{(y_0,u_0)}^{\mathcal{C}_U}\bigl((\Delta(U).(y_0,u_0))\cap (Y\times Q)\bigr)\leq
  (N+1)T^{-1}
\end{equation}
for any $(y_0,u_0)\in  Y^{\mathrm{fr}}\times Q$.

By the properties of conditional measures the estimate in \eqref{eq:betterestimateR2}
for all $(y_0,u_0)\in  Y^{\mathrm{fr}}\times Q$ implies
that
\[
(\nu\times m_U|_{R})\bigl(Y^{\mathrm{fr}}\times Q\bigr)\ll NT^{-1} m_U(Q^2).
\]
Dividing by $m_U(Q)$ this is leads to
\[
 \nu(Y^{\mathrm{fr}})\ll
 NT^{-1} \frac{m_U(Q^2)}{m_U(Q)}=NT^{-1} \frac{m_U\bigl((B^U_{1/2})^2\bigr)}{m_U(B^U_{1/2})}
\ll_U NT^{-1}.
\]
\end{proof}

\begin{proof}[Proof of Theorem~\ref{zero-or-fast-pol-rec}]
	Let~$X'\subseteq X$ be such that 
	the leaf-wise measure $\mu_x^{V}$
	is defined for all $x\in X'$
	and satisfies the almost sure properties 
  (which we will specify again in the course of the proof below).
	Due to~\eqref{eq:equivariance}
	the action of~$A$ descends to the factor~$\fancyZ$ so that $\mathcal A_V$ is an $A$-invariant $\sigma$-algebra. 

    By Lemma~\ref{lem: implication of product lemma} we may assume that $\mu_x^{V}=\mu_{u.x}^{V}$ 
    for $x,u.x\in X'$ and $u\in U$. 
    We claim that this implies
\begin{equation}\label{eq:leafofconditional}
 \bigl(\mu_{x_0}^{\mathcal A_V}\bigr)^U_x = \mu_x^U
\end{equation}
for $\mu$-a.e.\ $x_0$ and $\mu_{x_0}^{\mathcal A_V}$-a.e.\ $x$. Shrinking $X'$ if necessary and changing the definition
 of $(\mu_{x_0}^{\mathcal A_V})^U_x$ on a $\mu_{x_0}^{\mathcal{A}_V}$-conullset, we will assume
 that \eqref{eq:leafofconditional} holds for every $x_0,x\in X'$.

To see the claim suppose $\mathcal{C}$ is a~$U$-subordinate~$\sigma$-algebra on
a measurable subset~$Y\subseteq X'$. Then the atoms 
for~$\mathcal{C}$ belong to the atoms of 
the~$\sigma$-algebra~$\mathcal A_V$
corresponding to the map $x \to \mu_x^{V}$. 
Hence $\mathcal{C}\supseteq\mathcal{A}_V|_Y$ modulo~$\mu$.
By \cite[Prop.~5.20]{ET-book} this also shows
that $(\mu_{x_0}^{\mathcal A_V})_x^{\mathcal C}=\mu_x^{\mathcal{C}}$ for $\mu$-a.e.\ $x_0$ and $\mu_{x_0}^{\mathcal A_V}$-a.e.\ $x$.
	As the leaf-wise measures for $U$ are characterized by the conditional measures of these~$\sigma$-algebras we obtain the claim \eqref{eq:leafofconditional}.

We let~$\epsilon=\frac\alpha{5}$. 
By the definition of entropy contribution in \eqref{volume decay}
we have
  \[
   \mu_x^U(a^n B_{1/2}^U a^{-n}) \geq e^{(\alpha-\epsilon)n}\mu_x^U(B_1^U)
  \]
for a.e.~$x$, say for $x\in X'$, and for large enough $n$ (depending on $x$). By the above we may also phrase this as 
  \begin{equation}\label{alpha-asymp-again}
   \bigl(\mu_{x_0}^{\mathcal A_V}\bigr)_x^U(a^n B_{1/2}^U a^{-n}) \geq e^{(\alpha-\epsilon)n}\bigl(\mu_{x_0}^{\mathcal A_V}\bigr)_x^U(B_1^U)
  \end{equation}
for $x_0,x\in X'$ and for large enough $n$. 
 
 	Let~$\delta>0$ and let~$K\subseteq X'$ be a compact set with~$\mu(K)>1-\delta$.  
	Suppose~$r\in(0,\frac14)$ is such that~$10 r$
	is an injectivity radius for all~$x\in K$.
Shrinking~$K$ slightly if necessary we may assume that~$\mu(K)>1-\delta$
 and that there exists some fixed~$N_0$ such that~\eqref{alpha-asymp-again} holds 
 for all~$x\in K$ and all~$n\geq N_0$. This will be used
 as the input \eqref{alpha-asymp} to Proposition \ref{multiple-rec}
 after finding an ``efficient cover''.
 
\medskip
	Let~$\kappa\geq 1$ and let us assume that~$h_\mu( b  |\mathcal A_V)<\frac\alpha{4\kappa}$.
 	Let~$\mathcal E$
	be the~$\sigma$-algebra of~$ b  $-invariant so that  $\mu=\int_X\mu_{x}^{\mathcal E}\dee\mu(x)$
    is the ergodic decomposition of~$\mu$ with respect to the action of $ b $. 
	Then
	\[h_\mu( b  |\mathcal A_V)=\int_Xh_{\mu_x^{\mathcal E}}( b  |\mathcal A_V)d \mu(x);
	\]
	by ergodicity of $\mu$ under $A$ and e.g.~\eqref{eq:equivariance} applied to $G_a^+$
	the function~$x\mapsto h_{\mu_x^{\mathcal E}}( b  |\mathcal A_V)$ is invariant
	under~$A$. Hence it is constant and we obtain
	\[
	h_{\mu_x^{\mathcal E}}( b  |\mathcal A_V)=h_\mu( b  |\mathcal A_V)<\frac\alpha{4\kappa}
	\]
	for a.e.~$x$.

	 We will use a countable finite entropy partition~$\eta$ with the following properties\footnote{The construction of $\eta$ is straight-forward when $X$ is compact (where $\eta$ can be chosen to be finite), so the only somewhat delicate point is how to handle the part of the trajectory $\{ b ^j.x: |j|\leq n\}$ high up in the cusp.}
	 (see~\cite[\S7]{Pisa-notes} for the construction of such a partition): For any~$n\geq 0$ and~$x\in K$
	the partition element
	\[ 
	 [x]_{\eta_{-n}^n}\in\eta^n_{-n}=\bigvee_{j=-n}^n  b  ^j.\eta
	\]
	is contained in the Bowen~$n$-ball $xD_{r,n}^{ b  }$ (cf.~\ref{Bowen-Dinaburg}).  For convenience of 
	notation we define~$\eta_{-s}^s=\eta_{-\lfloor s\rfloor}^{\lfloor s\rfloor}$ for~$s\geq 0$.

 By the Shannon-McMillan-Breiman theorem conditioned on~$\mathcal A_V$ (see~\cite[Ch.~3]{ELW}) we have for~$\mu$-a.e.~$x\in X$ that
 \[
\frac1n  I_\mu(\eta^n_{-n}|\mathcal{A}_V)(x)=-\frac1n\log\mu_x^{\mathcal A_V}\left([x]_{\eta_{-n}^n}\right)\to 2h_{\mu_x^{\mathcal E}}( b ,\eta |\mathcal A_V)
<\frac\alpha{2\kappa}.
 \]
as~$n\to\infty$. In particular, this shows that for a.e.~$x$ and all sufficiently large~$n$ we have
\begin{equation}\label{big-atoms-equation}
\mu_x^{\mathcal A_V}\left([x]_{\eta_{-\kappa n}^{\kappa n}}\right)>e^{-\frac12\alpha n}.
\end{equation}
If necessary we increase~$N_0$ such that~\eqref{big-atoms-equation} holds for all~$n\geq N_0$ and all points~$x\in X_0$
for some measurable~$X_0\subset K$ of measure~$\mu(X_0)>1-\delta$. 

We fix some~$x_0\in X_0$ and
let~$\nu=\mu_{x_0}^{\mathcal{A}_V}$. Reinterpreting~\eqref{big-atoms-equation} we have
\begin{equation}\label{big-atoms-equation2}
 \nu\left([x]_{\eta_{-\kappa n}^{\kappa n}}\right)>e^{-\frac12\alpha n}
\end{equation}
for~$\nu$-a.e.~$x\in X_0$ and~$n\geq N_0$: indeed for~$x\in[x_0]_{\mathcal A_V}\cap X_0$ we have $\mu_x^{\mathcal{A}_V}=\mu_{x_0}^{\mathcal A_V}$
and so \eqref{big-atoms-equation} gives \eqref{big-atoms-equation2}. 
It follows
that for every~$n\geq N_0$ there are at most~$K_n\leq e^{\frac12\alpha n}$ many partition
elements~$P_{n,1},\ldots,P_{n,K_n}\in \eta_{-\kappa n}^{\kappa n}$ intersecting $X_0$.

Recall that \eqref{alpha-asymp-again} holds for $x_0,x\in X_0\subset X'$
and $n\geq N_0$,
which is now a statement for the leafwise measures of $\nu=\mu_{x_0}^{\mathcal A_V}$.  
By choice of~$r$ and construction of~$\eta$ the sets $P_{n,\cdot}$
have small diameter.
We now apply Proposition~\ref{multiple-rec}	
to each of these sets with the parameters $T=e^{(\alpha-\epsilon)n}$
and $N=e^{\tau n}$ for $\tau=\frac14\alpha$. Hence we obtain that the set
\begin{multline*}
	  Q_{n,k}=\Bigl\{x\in P_{n,k}\cap X_0\,\big|\, \mbox{$x$ has less than~$e^{\tau n}$ independent}\\
	  \mbox{$U$-returns within~$a^nB_1^Ua^{-n}$ to~$P_{n,k}$}\Bigr\}
\end{multline*}
satisfies~$\nu(Q_{n,k})\ll e^{(\tau+\epsilon-\alpha) n}$. Taking the union over all~$k=1,\ldots,K_n$ we obtain
\begin{equation}\label{nu-ineqwith4}
 \nu\left(\bigcup_{k=1}^{K_n} Q_{n,k}\right)\ll 
 e^{\frac12\alpha n} e^{(\tau+\epsilon-\alpha) n}= e^{(\epsilon-\frac14\alpha) n}=e^{-\frac1{20}\alpha}
\end{equation}
as $\epsilon=\frac15\alpha$.

Let us now return to the measure~$\mu$, define
\begin{multline*}
 Q_n=\Bigl\{x\in X_0\,\big|\,\mbox{$x$ has less than~$e^{\tau n}$ independent}\\\mbox{
 $U$-returns within~$a^nB_1^Ua^{-n}$ to~$X_0\cap (xD_{r,\kappa n}^{ b  })$}\Bigr\}.
	\end{multline*}
We claim that
\begin{equation}\label{equ;mainclaimformu}
\mu(Q_n)\ll e^{-\frac1{20}\alpha n}.
\end{equation}

To prove the claim it suffices to 
show
\begin{equation}\label{mu-ineqwith4}
	\mu_{x_0}^{\mathcal{A}_V}(Q_n)\ll e^{-\frac1{20}\alpha n}
\end{equation}
for a.e.~$x_0$, or 
since~$Q_n\subset X_0$ it suffices to consider a.e.~$x_0\in X_0$. 
Fix one such~$x_0$ and let again~$\nu=\mu_{x_0}^{\mathcal{A}_V}$. 
Define~$P_{n,k}\supseteq Q_{n,k}$ as above.
By our choice of the partition~$\eta$ we have~$P_{n,k}\subseteq xD_{r,\kappa n}^ b $ for all~$x\in P_{n,k}\cap X_0$. Hence every point of~$P_{n,k}$
that has at least~$e^{\tau n}$ independent~$U$-returns to~$P_{n,k}$
has at least that many to~$xD_{r,\kappa n}^ b $. Also~$X_0\setminus\bigcup_{k=1}^{K_n}P_{n,k}$ has~$\nu$-measure zero, which shows that~\eqref{nu-ineqwith4} implies~\eqref{mu-ineqwith4} and so the claim in~\eqref{equ;mainclaimformu}.
	
By the easy half of the Borel Cantelli lemma we see that a.e.~$x\in X_0$ can only belong to finitely many of the sets~$Q_n$ for~$n\geq N_0$.
It follows for a.e.~$x\in X_0$ and for large enough~$n$ that there 
are at least~$e^{\tau n}$ many~$U$-returns within~$a^nB_1^Ua^{-n}$ 
to~$xD_{r,\kappa n}^ b $. Since~$\mu(X_0)>1-\delta$ and~$\delta>0$ was arbitary, the theorem follows.
\end{proof}

\section{Unipotent invariance}\label{secunipotent}

We will consider in this section the first of the two possible outcomes of 
Theorem~\ref{zero-or-fast-pol-rec}, 
which in suitable circumstances leads to unipotent
invariance of the measure $\mu$. 

\subsection{Conditional positive entropy contribution implies invariance}

We continue with the setup of Theorem~\ref{zero-or-fast-pol-rec}, and in particular let~$X=\Gamma\backslash G$ be a locally homogeneous space defined by the lattice~$\Gamma$ arising
from a~$\Q$-structure~$\GG$ of the~$S$-arithmetic 
group~$G=\GG(\Q_S)$, and~$A<G$ a 
class~$\mathcal{A}'$-subgroup of higher rank.  
We suppose that~$U<G$ is a coarse Lyapunov subgroup, \ $ b \in A$, and $V \leq G_ b ^+$ a product of coarse Lyapunov groups satisfying the conditions of Theorem~\ref{zero-or-fast-pol-rec}, i.e. $U$ normalizes $V$, \ $U \cap V = \{e\}$ and 
 \[
  G^+_b \leq U \ltimes V\leq G^+_bC_G(b).
  \]
  Let~$\mathcal A_V$ be the~$A$-invariant~$\sigma$-algebra generated by the function~$x\mapsto\mu_x^{V}$.
  Furthermore, for every $x \in X$, let $I _ x ^ V$ denote the (closed) subgroup of $V$ defined by
\begin{equation}\label{eq: definition of I_x^V}
I _ x ^ V = \left\{ v \in V: [\mu _ x ^ V] = [\mu _ x ^ V.v] \right\};
\end{equation}
clearly $x \mapsto I^V_x$ is $\mathcal{A} _ V$-measurable.

A key observation that is the basis of the paper \cite{Einsiedler-Lindenstrauss-symmetry}  is that positivity of the relative entropy $h _ \mu (b, V| \mathcal A _ V)$ implies non triviallity of $I_x^V$. At its simplest, this is \cite[Lemma 3.1]{Einsiedler-Lindenstrauss-symmetry} which we now quote;  because of its simplicity and importance we also recall the proof.

\begin{lemma}[{\cite[Lemma 3.1]{Einsiedler-Lindenstrauss-symmetry}}]\label{pos-ent-implies-invariance-1}
Let $U$, $V$, $b$, $\mathcal A_V$ be as above. Then for $\mu$ a.e.~$x$ and for $\mu_x^{\mathcal A_V}$ a.e.~$y$, \ $\supp (\mu_x^{\mathcal A_V})_y^V < I_x^V$.  In particular, if $h_\mu( b ,V |\mathcal A_V)>0$ then $I_x^V$ is non-trivial a.s.
\end{lemma}

\begin{proof}
By the definition of conditional measures, for $\mu$ a.e.~$x$, for $\mu_x^{\mathcal A_V}$ a.e. $y$, we have that $[y]_{\mathcal A_V}=[x]_{\mathcal A_v}$, i.e. by definition of $\mathcal A_V$,\ $[\mu_x^V]=[\mu_y^V]$. Also for $\mu_x^{\mathcal A_V}$ a.e. $y$, for $(\mu_x^{\mathcal A_V})_y^V$ a.e. $v\in V$ we have that $[y]_{\mathcal A_V}=[v.y]_{\mathcal A_v}$ (since $[y]_{\mathcal A_V}$ is essentially constant with respect to $\mu_x^{\mathcal A_V}$ and that $v.y$ is in the conull set where \eqref{shift-formula} holds (since this $\mu$-conull set is $\mu_x^{\mathcal A_V}$ conull a.s.). Thus for $\mu_x^{\mathcal A_V}$ a.e.~$y$ there is a dense sequence of $v_i \in \supp (\mu_x^{\mathcal A_V})_y^V$ for which both
\[
\begin{aligned}
{[\mu_x^V]}&=[\mu_{v_i.y}^V]\qquad\text{and}\\
[\mu_{v_i.y}^V].v_i &= [\mu_y^V]
\end{aligned}
\]
hold. Since we also have $[\mu_x^V]=[\mu_y^V]$, it follows that $[\mu_{x}^V]=[\mu_{x}^V].v_i$ for all $i$. Since the group $\{v \in V: [\mu_x^V]=[\mu_x^V].v\}$ is closed, we may conclude that $\mu_x^V$ is indeed right invariant under the group generated by $\supp (\mu_x^{\mathcal A_V})_y^V$.
\end{proof}

In fact, the following stronger fact holds:
\begin{lemma}[{\cite [Cor.~3.4]{Einsiedler-Lindenstrauss-symmetry}}] For $\mu$-a.e.\ $x$, and $\mu_x^{\mathcal{A}}$-a.e.\ $y$, the leafwise measure $(\mu_x^{\mathcal{A}})_y^V$ is proportional to Haar measure on $I_x^V$.
\end{lemma}

While conceptually this follows from similar arguments as Lemma~\ref{pos-ent-implies-invariance-1}, showing this is slightly more delicate and we refer the reader to \cite{Einsiedler-Lindenstrauss-symmetry}. For our purposes, the simpler Lemma~\ref{pos-ent-implies-invariance-1} is sufficient.

The following simple proposition is standard (especially in the case that is relevant to this paper, i.e. when $V$ is a product of vector spaces over $\Q _ v$). 

\begin{proposition}\label{prop: local structure of closed groups}
    Let $V = \prod_ {s \in S} V _ s$ be a product of unipotent algebraic groups $V _ s$ over $\Q _ s$ respectively. Let $I$ be a closed subgroup of $V$. Then there is an open neighborhoods $B ^ V _ r$ of the identity in $V$ so that
\[
I \cap B ^ V _ r = \left (\prod_ {v \in S} I' _ v \right) \cap B ^ V _ r
\]
with each $I ' _ s$ the group of $\Q_s$-points of a (possibly trivial) algebraic subgroup of $V _ s$.
\end{proposition} 

We leave the proof to the reader (cf.\ also \cite[\S6]{EinsiedlerKatokNonsplit} for a relevant discussion).

\begin{lemma}
    Let $U$, $V$, $b$ be as in Theorem~\ref{zero-or-fast-pol-rec}, and let $\mathcal{A} _ V$ be the $\sigma$-algebra corresponding to the map $x \mapsto \mu _ x ^ V$ as above. Then under the assumptions of Theorem~\ref{zero-or-fast-pol-rec}, the group $I _ x ^ V$ is constant a.e.\ and is equal to a product $\prod_ {v \in S} I _ v$ with each $I_v$ (the $\Q_v$-points of) an $A$-normalized $\Q _ v$-algebraic subgroup of the $\Q _ v$-component of $V$.
\end{lemma}

\begin{proof}
By \eqref{eq:equivariance} and the definition of $I _ x ^ V$ it follows that for every $a \in A$ and $\mu$-a.x.~$x$
\begin{equation} \label{equivariance of I}
I _ {a . x} ^ V = a (I _ x ^ V) a^{-1}.
\end{equation}

That $I _ x ^ V$ is a.e.\ a product $\prod_ {v \in S} I _ {x,v}$ with each $I_{x,v}$  a $\Q _ v$-algebraic subgroup of the $\Q _ v$-component of $V$ follows from Proposition~\ref{prop: local structure of closed groups} which implies that the staement holds when restricting $I_x^V$ to a sufficiently small neighborhood of the identity, \eqref{equivariance of I} and Poincare recurrence along orbits of $b$ since $p$ expands $V$.

Elements $a_s \in G_s$ ($s \in S$) of class $\mathcal{A} '$ have the property that for every algebraic action of $G _ s$ on a projective variety $P$ over $\Q _ s$, for any point $p \in P$, we have that $a _ s ^ n p$ converge as $n \to \infty$ to an $a _ s$-invariant point of $P$. For every $s \in S$, the Lie algebra of $I _ {x,s}$ can be viewed as a point in an appropriate Grassmanian for $\mathfrak g _ s$ which is a point in a projective variety over $\Q _ s$. Employing Poincare recurrence once again, this time (individually) for every $a \in A$ we can deduce that the Lie algebra of $I _ {x,s}$ is fixed by $a$, hence $I _ {x, s}$ is normalized by $a$.

Thus for every $a \in A$ and $\mu$-a.e.~$x$ \eqref{eq:equivariance} implies that $I _ {a.x}^V= I _ x ^ V$, hence by ergodicity $I _ x ^ V$ is constant a.e.
\end{proof}

\begin{lemma}\label{pos-ent-implies-invariance}
Let $U$, $V$, $b$, $\mathcal A_V$ be as above and assume that~$h_\mu( b ,V |\mathcal A_V)>0$. Then~$\mu$ is invariant
under a nontrivial~$A$-invariant unipotent subgroup~$I<V$.
\end{lemma}

\begin{proof}
We have already established that under these conditions there is a nontrivial $A$-invariant unipotent subgroup $I < V$ so that for $\mu$-a.e.~$x$ the group $I _ x ^ V$ defined in \eqref{eq: definition of I_x^V} equals $I$.

The other hand, by the compatibility formula \eqref{shift-formula} there is a co-null set $X '$ so that if $x, v.x \in X'$ for $v \in V$ then 
$I _ {v.x}^V = v I _ {x} ^ V v ^{-1}$. However, $I _ {x} ^ V$ is constant a.e.\ and the only way to reconcile these two facts is that $I_x^V$ is a.e.\ normalized by the group generated by $\supp \mu _ x ^ V$. Thus $\mu _ x ^ V$ is invariant under $I$ \emph{both} from the left and from the right.

By \cite[Prop.~4.3]{Lindenstrauss-03} the a.s.\ left invariance of $\mu _ x ^ V$ by $I$ implies that $\mu$ is $I$-invariant.
\end{proof}

\subsection{The case of \texorpdfstring{$\mu$}{mu} invariant under an \texorpdfstring{$A$}{A}-invariant unipotent group}
Suppose now that $\mu$ is indeed invariant
under a nontrivial $A$-invariant unipotent subgroup~$I<G_b^+$. In view of Ratner's landmark Measure Classification Theorem this imposes rather strict restraints on the measure $\mu$.

Let~$H_u$ denote the subgroup generated by all one-parameter unipotent subgroups of~$G$ that preserve~$\mu$,
and note that~$A$ normalizes~$H_u$.

Using Ratner's measure classification theorem \cite{Ratner-Annals,Ratner-Acta} extended to  $S$-arithmetic homogeneous spaces
by Ratner \cite{Ratner-padic} and Margulis-Tomanov 
\cite{Margulis-Tomanov} 
we obtain that the ergodic components~$\mu_x^{\mathcal E_{H_u}}$
of~$\mu$ with respect to the action of~$H_u$ must be a homogeneous measure, i.e.
\begin{equation}\label{description-of-erg-comp-via-Rat}
	 \mu_x^{\mathcal E_{H_u}}=m_{L_x'.x}
\end{equation}
is a.s.\ the normalized Haar measure on a closed orbit $L_x'.x$ for some closed subgroup~$L_x'<G$. 

The following result, obtained by combining the results of~\cite{Margulis-Tomanov-almost-linear} and \cite{Tomanov-orbits} (which are based on the measure classification theorems referred to above) will be convenient for our purposes: 

\begin{theorem}\label{MTclassA}
 Let~$X=\Gamma\backslash G$ be an $S$-arithmetic quotient with $G\leq \GG(\Q_S)$ a finite index subgroup, 
 $\GG$ a connected algebraic group defined over~$\Q$, and $\Gamma$ an arithmetic lattice
 arising from $\GG$. Let~$A<G$ be a diagonalizable subgroup of class-$\mathcal A'$, and $H_u<G$ 
an~$A$-normalized subgroup generated by Zariski-connected unipotent subgroups. Let~$H=A H_u$ be the group generated by~$A$ and~$H_u$. 
Let~$\mu$ be an~$H$-invariant and ergodic probability measure on~$X$. 
Then there exists a connected algebraic~$\Q$-subgroup~$\LL\leq\GG$ and an open finite index 
subgroup~$L\leq\LL(\Q_S)$ such that a.e.~$x\in X$ has a representative~$g\in G$ 
such that the ergodic component~$\mu_x^{\mathcal{E}}$
of~$\mu$ for the point~$x=\Gamma g$ and for the action of~$H_u$ is algebraic and equals the normalized Haar measure~$m_{\Gamma Lg}$
on the closed orbit~$\Gamma L g$. Moreover, $H_u<g^{-1}Lg$, $\mu$ is supported on the closed orbit~$\Gamma N_{G}^1(\LL) g$, 
$A<N_G^1(g^{-1}\LL g)$, and
$A$ normalizes $g^{-1}Lg$. 
\end{theorem}

Here~$N^1_{G}(\LL)$ denotes the group
\[
N^1_{G}(\LL)=\{g\in G: \text{$g$ normalizes~$\LL$ and preserves the Haar measure on~$\LL(\Q_S)$}\}.
\]
Note that in the $S$-arithmetic setting $N^1_{G}(\LL)$
is in general \emph{not} the group of~$\Q_S$-points of an algebraic group. 

Let~$\Gamma g_0\in\supp\mu$. Applying conjugation
by~$g_0$ to~$A$ and~$H_u$, and 
pushing~$\mu$ by the action of~$g_0$,
we may also simplify the notation and suppose~$g_0=e$ is the identity,~$L<\LL(\Q_S)$ is a finite index subgroup
which can be used to describe the ergodic components
as in \eqref{description-of-erg-comp-via-Rat}
and~$\mu$ is supported on the orbit~$\Gamma N_G^1(\LL)\subset \Gamma N_G(\LL)$. 

\begin{proof}
Applying \cite[Thm.\ (a)]{Margulis-Tomanov-almost-linear} we find a closed subgroup $L'<G$ containing $H_u$, 
normalized by $A$
so that a.s.\
the ergodic component $\mu_x^{\mathcal{E}_{H_u}}$
is the normalized $L'$-invariant Haar measure on the closed orbit $L'.x$, i.e.\ the closed subgroup $L'_x$ appearing in \eqref{description-of-erg-comp-via-Rat} can be chosen uniformly a.s. 

We note that in \cite{Margulis-Tomanov-almost-linear} quotients of $G$ by more general closed subgroups $\Gamma$ were considered. Hence stabilizer of points could contain
the groups of $\Q_S$-points of algebraic subgroups defined over $\Q_S$. In our case $\Gamma$ is discrete, which removes this possibility and simplifies the statements in \cite{Margulis-Tomanov-almost-linear}. 
Applying \cite[Thm.\ (b)]{Margulis-Tomanov-almost-linear} with this in mind, there exists a point $x_0$ so that $\mu$
is concentrated on the orbit $N_G^1(L').x_0$
for
\[
N^1_{G}(L')=\{g\in G: \text{$g$ normalizes~$L'$ and preserves the Haar measure on~$L'$}\}.
\]
Moreover, also by \cite[Thm.\ (b)]{Margulis-Tomanov-almost-linear} we have $A<N^1_G(L')$. 

As the probability measure $\mu_x^{\mathcal{E}_{H_u}}$
is a.s.\ $H_u$-invariant and $H_u$-ergodic,
it follows from \cite[Thm.\ 2]{Tomanov-orbits}
(see also \cite[Prop.~1.1]{Borel-Prasad} for the real case) that it is algebraic. In other words the support of the measures satisfy for a.e.\ $x=\Gamma g$
that
\[
 \supp\mu_x^{\mathcal{E}_{H_u}}=\Gamma g L'=\Gamma L_g g
\]
for some finite index 
subgroup $L_g$ of the $\Q_S$-points of some Zariski connected~$\Q$-group~$\LL_g<\GG$ (that in general depends on the representative~$g\in G$ of~$x$).  Moreover, also by \cite[Thm.\ 2]{Tomanov-orbits} we have $H_u<g^{-1}L_gg$. Note that finite index implies that $L_g$ is open in $\LL_g(\Q_S)$. 

Now fix a point $x_1\in N_G^1(L').x_0$ with these properties,
let $g_1$ be a representative of $x_1=\Gamma g_1$, and
define $\LL=\LL_{g_1}$. Then by the above a.e.\ $x$ can be written as $x=\Gamma g_1 m$ for some $m\in N^1_G(L')$. Choosing the representative $g=g_1m$ of $x$ we now obtain by combining the above that a.s.\ 
\[
 \supp\mu_x^{\mathcal{E}_{H_u}}=\Gamma g L'=\Gamma g_1m L'=
 \Gamma g_1 L'm=\Gamma L_{g_1}g_1m=\Gamma L_{g_1} g. 
\]
Note that this implies that $gL'g^{-1}\cap L_{g_1}$ is open both in $gL'g^{-1}$ and in $L_{g_1}$, which in particular implies that they have the same Lie algebra as $\LL$. 
We now define the open subgroup $L=g_1L'g_1^{-1}\cap\LL(\Q_S)$ of $\LL(\Q_S)$,
and obtain that $\Gamma Lg_1\subseteq\Gamma L_{g_1}g_1$
is a closed subset with positive measure for the uniform measure $\mu_{x_1}^{\mathcal{E}_{H_u}}$ on the closed orbit $\Gamma L_{g_1}g_1$. Recall also that $H_u<g^{-1}_1Lg_1$ by the properties of $L'$ and of $L_{g_1}$, 
which implies that $\Gamma Lg_1$ is invariant under $H_u$. Using ergodicity for $H_u$ and $\mu_{x_1}^{\mathcal{E}_{H_u}}$ this forces $\Gamma Lg_1=\Gamma L_{g_1}g_1$. 
By construction
we have $A<N_G^1(L')< N_G^1(g_1^{-1}\LL g_1)$, and so that $A$ normalizes $g_1^{-1}Lg_1$.

It remains to prove that $\Gamma N_G^1(\LL)$
is closed. So suppose $\Gamma g_n\to\Gamma g$ as $n\to\infty$, where $g_n\in N_G^1(\LL)$
and $g\in G$. By definition this means that there exists a sequence $\gamma_n\in\Gamma$
so that $\gamma_ng_n\to g$ as $n\to\infty$.
We wish to show that $\Gamma g\in \Gamma N_G^1(\LL)$. 

Let $\mathfrak g$ denote the Lie algebra of $\GG$, let $e_1,\ldots,e_d\in\mathfrak g(\ZZ)$ be a rational basis of the Lie algebra of $\LL$
and consider
\[
 w=e_1\wedge\cdots\wedge e_d\in W=\bigwedge^d\mathfrak g. 
\]
Furthermore, let us write $\rho=\bigwedge^d\operatorname{Ad}$ for the
natural representation of $\GG$ on $W$. We identify $w$ as a point in $W(\Q_S)=\prod_{v\in S}W(\Q_v)$ in the usual way.

Write $g_n=(g_{n,v})_{v \in S}$. As the components $g_{n,v}$ of $g_n$ normalize
$\LL$, we see that $\rho(g_n)$ maps $w$
to a multiple of $w$. That is for every $v\in S$ there exists $c_{n,v}\in\Q_v^\times$ so that 
\[
\rho((g_n)_v)w=c_{n,v} w.
\]
We note that $|c_{n,v}|_v$ equals the modular character for the automorphism of $\LL(\Q_v)$ defined by conjugation by $(g_n)_v$. This implies that $\prod_{v\in S}|c_{n,v}|_v=1$ by the assumption that conjugation by $g_n$ preserves the volume in $L<\LL(\Q_S)$, which determines $c_{n,\infty}$ up to the sign as the inverse of $\prod_{p\in S}|c_{n,p}|_p$. 

Let $K=\{\pm 1\}\times\prod_{p\in S}\ZZ_p^\times$ be the compact subgroup of $\Q_S^\times$ consisting of all elements with norm $1$ in all places. Note that $K$ acts naturally on $W(\Q_S)$. Taking the quotient by $K$ we reinterpret the above fact concerning $c_{n,\infty}$ in the form
\[
K\rho(g_n)w
\in KW\bigl(\mathcal{O}_S\bigr).
\]
Now notice that $KW(\mathcal{O}_S)$ is a discrete subset of $K\backslash W(\Q_S)$.
As $\gamma_ng_n\to g$
as $n\to\infty$ we see that
\[
 K\rho({\gamma_ng_n})w\to K\rho(g) w
\]
as $n\to\infty$. It follows that
\[
K\rho(\gamma_ng_n)w= K\rho(g) w
\]
for all sufficiently large $n$. For any such $n$ we therefore have $K\rho(g_n)w=K\rho(\gamma_n^{-1}g)w$.
However, this implies that $\gamma_n^{-1}g$
normalizes $\LL$ and that the modular character of its associated automorphism agrees for every place $v\in S$ with the 
modular character of the automorphism of $g_n$. It follows that $\gamma_n^{-1}g\in N_G^1(\LL)$ and $\Gamma g\in \Gamma N_G^1(\LL)$ as desired.
\end{proof}

\section{Fast multiple recurrence cannot hold}\label{sec:not so fast}

Unless said otherwise we 
suppose in this section that~$\GG$ is a~$\Q$-almost simple~$\Q$-group that is a
form of~$\operatorname{SL}_2^k$ or~$\PGL_2^k$ for some~$k>1$, that~$\Gamma$ is an arithmetic lattice in~$G=\GG(\Q_S)$
arising from~$\GG$, that~$X=\Gamma\backslash G$ is the resulting
locally homogeneous space,  
that~$A<G$ is a class-$\mathcal{A}'$ subgroup of higher rank,
and that~$\mu$ is an~$A$-invariant and ergodic probability measure on~$X$ satisfying~$h_\mu(a)>0$
for some~$a\in A$. 

For simplicity, we treat the case of $\GG$ being a $\QQ$-form of $\SL_2^k$, the case of $\GG$ a $\QQ$-form of $\PGL_2^k$ is handled similarly. We recall that this implies that there is a number field $\KK$ with $k=[\KK:\Q]$, and a quaternion algebra (possibly $\KK$-split) $\mathbb D$ be over $\KK$, so that $\mathbb G = \operatorname{Res}_{\KK\mid\Q}(\mathbb G^\KK)$ where $\mathbb G^\KK$ is the $\KK$-group of elements of $\mathbb D$ of reduced norm 1. 
With these notations, we can identify $\GG(\Q_S) $ with
\begin{equation}
\label{gg in another way}
\prod_{v \in S}\prod_{\sigma|v} \mathbb \mathbb G^\KK(\KK_\sigma),
\end{equation}
where in the second product $\sigma$ runs over all places of $\KK$ lying over $v$. For a place $\sigma$ of $\KK$ lying over $v$ we let $|\cdot|_\sigma$ denote the corresponding absolute value extending the absolute value $|\cdot|_v$ on $\Q$; if $d_\sigma=[\KK_\sigma:\Q_v]$ then for any $k \in \KK$
\[
\sum_{\sigma} d_\sigma \log |k|_\sigma = 0,
\]
with $\sigma$ running over all places of $\KK$.
If $\mathbb{D}$ is split over $\KK$ we can identify $\mathbb G^\KK$ with $\SL_2$. For $\sigma$ a place of $\KK$ and $g \in \mathbb G^\KK(\KK)$ (or $\mathbb G^\KK(\KK_\sigma)$) denote by $\|g\|_\sigma$ the operator norm of $g$ as a $2\times 2$-matrix. If $\mathbb{D}$ is not split we can embedd $\mathbb G^\KK$ as a $\KK$-group in $\SL_4$, and let $\|g\|_\sigma$ denote the operator norm under this embedding. 

Let $S'$ be the set of all places of $\KK$ lying over $S$, i.e. $S'=\{\sigma: \sigma | v \text { for some $v \in S$}\}$.
Under the identification \eqref{gg in another way}, we may identify $\Gamma$ with a subgroup commensurable to $\mathbb G^\KK (\mathcal O_{\KK,S'})$ with $\mathcal O_{\KK,S'}$ denoting the ring of $S'$-integral elements of $\KK$.

\subsection{Arithmetic lattices}\label{easy-start}
Note that the assumption on~$\Gamma$ simply means that we may assume  
that~$\Gamma$ and $\mathbb G^\KK (\mathcal O_{\KK,S'})$ are commensurable. Replacing~$\Gamma$ with~$\Gamma'=\Gamma\cap \mathbb G^\KK (\mathcal O_{\KK,S'})$
we obtain a finite-to-one cover~$X'=\Gamma'\backslash G\rightarrow X$. It is now easy to define an~$A$-invariant
lift~$\mu'$ on~$X'$ by defining for any~$f'\in C_c(X')$
\[
 \int_{X'}f'\dee\mu'=\frac1{[\Gamma:\Gamma']}\int_X\sum_{\Gamma'g}f'(\Gamma'g)\dee\mu,
\]
where the sum is over all cosets~$\Gamma'g\in X'$ with~$\Gamma g=x\in X$ and so defines a continuous function~$f(x)$
of compact support on~$X$.
If~$\mu'$ is not~$A$-ergodic we may replace it by one of its~$A$-ergodic component. This shows that 
the study of~$A$-invariant and ergodic probability measures on~$X'$ includes the study of such measures on~$X$.
Therefore, we may and will always assume that $\Gamma<\mathbb G^\KK (\mathcal O_{\KK,S'})$.

\subsection{Choosing \texorpdfstring{$b$}{b} for
the application of Theorem~\ref{zero-or-fast-pol-rec}}\label{secchoosingb}
We can use \eqref{eq:entropycontributionsaddup} and our assumption $h_\mu(a)>0$
to find a coarse Lyapunov weight~$\chi_1$ 
and corresponding coarse Lyapunov subgroup~$U=G^{[\chi_1]}<G_a^+$
such that its entropy contribution 
\[
\alpha=h_\mu(a,U)>0
\]
is positive.

We note that for any $v \in S$ and $\sigma | v$, the direct factor $\mathbb G^\KK (\KK_\sigma)$ is either compact or isomorphic to $\SL_2(\KK_\sigma)$, and that $\KK_\sigma$ is a finite extension of $\Q_v$. Moreover, unless $A$ commutes with that factor, such a factor defines two nonzero Lyapunov weights corresponding to two opposite maximal unipotent subgroups of $\SL_2(\KK_\sigma)$. It follows that the coarse Lyapunov group~$U$ is the direct product of maximal unipotent subgroups of  $\mathbb G^\KK(\KK_\sigma)$ for all $\sigma$ in some subset $S'_1$ of $S'$.
Let~$G_1=\prod_{\sigma \in S'_1} \mathbb G^\KK (\KK_\sigma)$; then $U < G_1$.

Since~$A$ is of higher rank, there exists some nonzero Lyapunov
weight~$\chi_2$ that is inequivalent to~$\pm\chi_1$. 
Hence there exists a sequence~$ b _k\in A$ with
\begin{equation}\label{chi1ofbisclosetozero}
0\leq\chi_1( b _k)\to 0
\end{equation}
but
\begin{equation}\label{chi2ofbgetsbig}
\chi_2( b _k)\rightarrow\infty
\end{equation}
as $k\to\infty$. We will choose $ b $ from this sequence after some more preparations. 

Corresponding to the coarse Lyapunov weight $[\chi_2]$ there again corresponds a set $S'_2$ of places of $\KK$ lying over $S$, as well as 
a closed normal subgroup~$G_2=\prod_{\sigma \in S'_2} G^\KK(\K_\sigma)$ so that $G^{[\chi_2]}<G_2$. 
Multiplying~$\chi_2$ with a positive scalar if necessary,
we may assume~$\chi_2$ is a smallest weight
from its coarse Lyapunov equivalence class. In particular
all Lyapunov weights corresponding to~$G_2$ are of
the form~$t\chi_2$ for~$t\in\R$ with~$|t|\geq 1$. 

Since~$\chi_2$ and~$\pm\chi_1$ are inequivalent, 
we have~$G_1\cap G_2=\{e\}$. Let 
\[
S'_3 = \{\sigma|v:v \in S\}\setminus (S'_1\cup S'_2)
\]
and $G_3=\prod_{\sigma\in S_3'} \G^\KK (\KK_\sigma)$
so that~$G=G_1\times G_2\times G_3$. We note that 
while the groups~$G_1,G_2$ are
nontrivial and have no compact normal factors, 
the group~$G_3$ may be compact or trivial. 
Moreover, we note that the projection to $G_3$
of the element $b$ (to be defined below) may
be nontrivial.

Let~$e^{\Delta}$ be the largest (real or~$p$-adic) absolute value of an eigenvalue of the adjoint representation of~$a$ on the ambient matrix algebra.
Then
\begin{equation}\label{prep-delta}
 u\in a^n B_1^U a^{-n}\text{ implies }\|u\|\ll e^{\Delta n}
\end{equation}
for all~$n\geq 0$.

As $U=G^{[\chi_1]}$ is a coarse Lyapunov subgroup,
the amount of expansion of the conjugation map $\theta_b$ 
restricted to $U$ depends only on $\chi_1(b)$ for any $b\in A$. 
With this in mind \eqref{volume decay}
implies that the entropy contribution~$h_\mu( b ,U)$
of~$U$ for~$ b \in A$ with~$\chi_1(b)\geq 0$ is a certain
positive multiple of~$\chi_1(b)$. 
Hence we may use  \eqref{chi1ofbisclosetozero}-- \eqref{chi2ofbgetsbig} to choose $ b = b _k$ 
for large enough $k$ so that 
\begin{equation}\label{choiceofkappa1}
	h_\mu( b , U)<\frac1{4}\alpha
\end{equation}
and
\begin{equation}\label{choiceofkappa}
 \chi_2( b )\geq\lambda=5 |S|k\Delta
\end{equation}
holds. With our choice of~$\chi_2$ as the smallest weight
from its equivalence class we obtain for all sufficiently small $r>0$ that
the projection $h'$ to $G_2$ of any 
$h\in D_{r,n}^b$ satisfies
\begin{equation}\label{prep-lambda}
 d(h',a')\ll r e^{-\lambda n}
\end{equation}
for some~$a'\in C_{G_2}(b)$; note that by our choice of $G_2$, the group $C_{G_2}(b)$
is a diagonalizable subgroup of~$G_2$.

Finally we consider the horospherical subgroup~$G_ b ^+$. 
By~\eqref{chi1ofbisclosetozero} and since~$A$ is 
of class-$\mathcal{A}'$ we either have~$U<C_G( b )$
or~$U<G_ b ^+$. In the former case we set~$V=G_ b ^+$.
In the latter we note that any two coarse Lyapunov weight
groups -- except those of opposite weight --  commute in the case at hand.  Therefore we can find a direct 
product~$V$ of coarse Lyapunov weight groups
such that~$G_ b ^+$ is the direct product of~$U$ and~$V$. 
Furthermore we let~$\mathcal{A}_V$ be the~$\sigma$-algebra 
generated by~$x\mapsto\mu_x^{V}$. 
Applying Theorem~\ref{zero-or-fast-pol-rec} for 
our chosen~$ b \in A$ and~$\kappa=1$ we either
have~$h_\mu( b |\mathcal{A}_V)\geq \frac14\alpha$ or a strong form of recurrence.
However, the product formula of the leaf-wise
measures in the form of \eqref{eq:entropycontributionsaddup} implies 
(in either of the two cases) that
\[
 h_\mu( b |\mathcal{A}_V)=h_\mu( b ,G_ b ^+|\mathcal{A}_V)=
 h_\mu( b ,U|\mathcal{A}_V)+h_\mu( b ,V|\mathcal{A}_V)
\]
Using $h_\mu( b ,U|\mathcal{A}_V)\leq h_\mu( b ,U)$ 
and~\eqref{choiceofkappa1} 
we obtain from Theorem~\ref{zero-or-fast-pol-rec}
\begin{itemize}
	\item that $h_\mu( b ,V|\mathcal{A}_V)>0$ or
	\item that almost surely the 
fast multiple recurrence claim for the dynamics of~$U$ and the
Bowen balls at $x$ holds. 
\end{itemize}
We will analyse the
(desired) first possibility in \S\ref{easy-end}
after showing that the fast recurrence claim cannot hold.

\subsection{Fast multiple recurrence cannot hold}\label{fastreccannothold}

We will now show that the fast multiple recurrence claim of Theorem~\ref{zero-or-fast-pol-rec}
cannot hold. In fact we have chosen~$b$ in a way so that certain Diophantine inequalities will be
impossible to satisfy in a nontrivial fashion, which will allow us to upgrade
certain inequalities to equations. This is where the arithmetic structure of the lattice, i.e.\ 
the assumption~$\Gamma<\GG(\mathcal{O}_S)$, will be used.

Recall that~$D_{r,n}^{b}$
is the Bowen ball as defined in~\eqref{Bowen-Dinaburg}. We assume (for the purpose of a contradiction)
that for~$\mu$-a.e.~$x\in X'$, 
 any sufficiently small~$r>0$ and
 all sufficiently large~$n$ there are at least $e^{\frac\alpha4n}$ independent $U$-returns within~$a^n B_1^U a^{-n}$
of~$x$ to~$(xD_{r,n}^{ b  })\cap X'$. 
We fix some $g\in G$ and $n_0\geq 1$
so that $x=\Gamma g$ satisfies this claim for all $n\geq n_0$. 

Suppose now that $n\geq n_0$
and some~$u\in a^n B_1^U a^{-n}\setminus B_{2r}^U$ is such that $xu^{-1}=xh$ for some~$h\in D_{r, n}^{b }$. Then there exists some~$\gamma\in\Gamma$ with  $\gamma g u^{-1}=g h$, which we may also write as 
\begin{equation}
\gamma=ghug^{-1}\in g D_{r,n}^b (a^nB_1^U a^{-n}) g^{-1}
\label{definitionofgamma}.
\end{equation}
As $u\notin B_{2r}^U$ and $h\in B_r^G$,
it follows that $\gamma$ must be nontrivial.
Moreover, we claim that we obtain
at least $e^{\frac\alpha4n}$ different elements of $\Gamma$
in this way all satisfying \eqref{definitionofgamma}.

Indeed we have at least  $e^{\frac\alpha4n}$ independent $U$-returns of $x$ within $a^nB_1^Ua^{-n}$ to $xD_{r,n}^b$
and we show that no two such returns give rise to the same lattice element. So suppose that $xu_1^{-1}$ and $xu_2^{-1}$ are independent $U$-returns within $a^nB_1^Ua^{-n}$ to $x D_{r,n}^b$
that give rise to the same lattice element $\gamma$. 
Then $d(u_1^{-1},u_2^{-1})\geq 2r$ by definition of independent $U$-returns, but also $\gamma=gh_1u_1g^{-1}=gh_2u_2g^{-1}$
for some $h_1,h_2\in D_{r,n}^b$. 
However, the latter implies 
$u_1u_2^{-1}=h_1^{-1}h_2\in B_{2r}^G$, which contradicts
$d(u_1^{-1},u_2^{-1})\geq 2r$.

We now concern ourselves with the properties of the various elements~$\gamma$ appearing above.
For the places $\sigma \in S'_1$ of $\KK$ corresponding to the factor~$G_1$ we use the estimate of~$u$ in~\eqref{prep-delta}
and find that
\begin{equation}\label{boundoncomplexity}
	\|\gamma\|_\sigma \ll_g e^{\Delta n}
\end{equation}
while for all other places of $\KK$ we have trivially
\begin{equation}\label{trivialboundcomplexity}
\|\gamma\|_\sigma \ll_g 1.
\end{equation}
We will combine these estimates with the information coming from the factor~$G_2$. 

\begin{lemma}[Semi-simple elements]\label{semisimple-gamma}
	If~$x=\Gamma g$ satisfies fast multiple recurrence for~$b$ as in~\eqref{choiceofkappa},~$r$ is sufficiently small, and~$n$ is sufficiently large,
	then any element~$\gamma$ 
 as in \eqref{definitionofgamma} is semi-simple and non-central. 
\end{lemma}

\begin{proof}
Recall that we may identify $\Gamma$ with a subgroup of $\mathbb G^\KK(\KK)$.

If $\mathbb D$ is a division algebra over $K$, then $\gamma$ is automatically semisimple. If $\mathbb D$ is split, i.e.\ is isomorphic to the ring of $2\times 2$-matrices over $\KK$, the Jordan decomposition gives that~$\gamma=s_1u_1$ is the product
of a semisimple element~$s_1$ and a unipotent element~$u_1$ that commute with each other. If~$s_1$ does not belong to the center,
then we must have~$u_1=I$. Now recall that~$u \in G_1$, \eqref{definitionofgamma}, and that the projection of~$g^{-1}\gamma g$ to~$G_2$ belongs
to~$D_{r, n}^{ b  }\subset B_r^G$. In particular we see that the eigenvalues
of $\gamma$ are close to~$1$ according to $|\cdot|_\sigma$ for any $\sigma \in S_2'$. 

This implies that if $s_1$ is central, $s_1 = I$.
Since $\gamma$ is by definition not the identity, if the statement of the lemma is false, $\mathbb D$ is split, $\GG = \operatorname{Res}_{\KK\mid\Q}\SL_2$ (so $\GG(\Q_S)$ can be identified with $\prod_{\sigma \in S'} \SL_2(\KK_\sigma)$) and
$\gamma$ is unipotent.

For $\sigma\in S_2'$ let $b_\sigma \in \SL_2(\KK_\sigma)$ denote the projection of $b$ to $\SL_2(\KK_\sigma)$.
By \eqref{prep-lambda}, for any $\sigma\in S_2'$ there is some $a' \in \C_{\SL_2(\KK_\sigma)}(b)$ so that 
\[
\|\gamma - ga'g^{-1}\|_\sigma \ll re^{-\lambda n}.
\]
If $\gamma$ is unipotent, it follows that $|\operatorname{trace} (a') -2|_\sigma \ll re^{-\lambda n}$. Thus
$\|a'-e\|_\sigma \ll r^{1/2}e^{-\lambda n/2}$ and hence
\begin{equation}\label{anotherineqforh}
\|\gamma-e\|_\sigma \ll_g r^{1/2}e^{-\lambda n/2}.
\end{equation}

For every place $\sigma \in S'_1$ of $\KK$ we use
use~\eqref{prep-delta} to get
\[
 \|\gamma\|_\sigma\ll_{g} e^{\Delta n}, 
\]
and for every $\sigma \in S'_3$
\[
 \|\gamma\|_\sigma\ll_{g,r} 1.
\]

Since $\sum_{\sigma |v} d_\sigma = [\KK:\Q]=k$ and since there is at least one $\sigma \in S'_2$, it follows that 
\[
\prod_{\sigma \in S'} \|\gamma -I \|_\sigma ^{d_\sigma} \ll_{g,r} e^{\Delta k |S| n - \lambda n /2}  \leq  e^{-\Delta |S'| n/2}.
\]
However for any nonzero $x \in \mathbb{D}(\mathcal O_{\K,S'})$,
\[
\prod_{\sigma \in S'} \|x \|_\sigma^{d_\sigma} \geq 1
\]
hence we get $\gamma=I$ and so a contradiction.
\end{proof}

\begin{lemma}[Commuting]\label{samecommutator}
  If~$x=\Gamma g$ satisfies fast multiple recurrence for~$b$ as in~\eqref{choiceofkappa},~$r$ is sufficiently small, and $n$ is sufficiently large, then 
  any two elements
  of $\Gamma$ satisfying \eqref{definitionofgamma} commute.
\end{lemma}

\begin{proof}
Let $\gamma,\eta\in\Gamma$ be two 
elements satisfying \eqref{definitionofgamma}.
We will again make use of~\eqref{prep-delta} and \eqref{prep-lambda} for the 
study of the commutator
\[
 [\gamma,\eta]=\gamma^{-1}\eta^{-1}\gamma\eta.
\]
In fact~\eqref{prep-delta} gives
\[
 \|[\gamma,\eta]\|_\sigma \ll_g e^{4\Delta n}
\]
for $\sigma \in S'_1$. We again have trivially
\[
\|[\gamma,\eta]\|_\sigma \ll_g 1
\]
for all $\sigma \in S'_3$. Due to~\eqref{prep-lambda} we have
\[
\|[\gamma,\eta]-I\|_\sigma\ll_g e^{-\lambda n}
\]
for all $\sigma\in S'_2$.

It follows that
\[
x=[\gamma,\eta]-I
\]
is an element of the ring $\mathbb{D} ( \mathcal O_{\KK,S'})$ with
\[
\prod_{\sigma \in S'} \|x\|_\sigma^{d_\sigma} \ll e^{4\Delta k |S|n-\lambda n} \leq e^{-\Delta k |S|n}.
\]
For $n$ large this can only happen if $x=0$, i.e. $\gamma$ and $\eta$ commute as claimed. 
\end{proof}

Now note that any two elements of~$\SL_2$, that do not belong to the centre of~$\SL_2$ (or any two nontrivial elements
of~$\PGL_2$) commute
if and only if they have the same centralizer. Hence by Lemmas \ref{semisimple-gamma}--\ref{samecommutator}
we have found for every large enough $n$ a torus $\TT<\GG$ defined over $\Q$ with $e^{\frac\alpha4n}$ many elements in $\TT(\mathcal{O}_S)$ satisfying \eqref{boundoncomplexity}--\eqref{trivialboundcomplexity}. 
It is a standard fact that this is impossible. 

\begin{lemma}\label{fewlementsintorus}
Let $\GG$ be a $\Q$-group. Let $S$ be a finite set of places of $\Q$ containing $\infty$. 
 There exist constants $C$ and $d$ (depending on $\GG$ and $S$ only) so that for any $\Q$-torus $\TT<\GG$ and $t\geq1$ there are at most $Ct^d$ many elements $\gamma\in\TT(\mathcal{O}_S)$
 satisfying
 \begin{equation}\label{eq:height bound}
 \max_{\sigma\in S}\|\gamma\|_\sigma<e^t. 
 \end{equation}
\end{lemma}

\begin{proof}
 Let $\GG<\SL_d$ and $\TT<\GG$ be a $\Q$-torus. This means that over $\overline \Q$ the group $\TT$ is isomorphic to a product of $k<d$ copies of the multiplicative group $\G_m$, and moreover this isomorphism can be defined over a number field $\K$ with $D=[\K:\Q]$ bounded in terms of $d$. Let $\overline S$ denote the set of places of $\K$ over $S$, and let $\phi : \G_m^k \to \TT$ be the aforementioned isomorphism.

For $\alpha=(\alpha_1,\dots,\alpha_k)\in\G_m^k(\overline \Q)$, let $\hat h(\alpha)=h(\alpha_1)+\dots+h(\alpha_k)$ with $h(\cdot)$ denoting the usual logarithmic height, cf.~\cite[Ch. 3]{Zannier}.
Note that if $\gamma\in\TT(\mathcal{O}_S)$ then $\phi^{-1} (\gamma) \in \bigl({\mathcal O}^{\times}_{\K,\overline S}\bigr)^k$, and if $\gamma$ further satisfies 
 \[
 \max_{\sigma\in S}\|\gamma\|_\sigma<e^t
 \]
 then $\hat h (\phi^{-1} (\gamma)) \ll t$ with the implied constant depending only on $d$.

The $S$-arithmetic version of the Dirichlet unit theorem (see \cite[\S 3.2.2]{Zannier}) gives that the group of $\overline S$-units $\mathcal O^{\times}_{\K,\overline S} \simeq \Z^{\#\overline{S}-1}\times \mathrm{Tors}$ with $\mathrm{Tors}$ a group of roots of unity of bounded order. Let $M=k(\#\overline{S}-1)$ and $\alpha_1,\dots,\alpha_{M}$ be generators of $\bigl({\mathcal O}^{\times}_{\K,\overline S}\bigr)^k$ modulo torsion. Then by \cite[\S3.4]{Zannier}, the map
\[
(n_1,\dots,n_m)\goesto \hat h(\alpha_1^{n_1}\dots\alpha_M^{n_M})
\]
extends to a norm $\norm{\cdot}_{*}$ on $\R^M$.\footnote{Note that this is not just a simple consequence of the functorial properties of the height, but requires an argument.}
Since there is a uniform lower bound on the height of an algebraic number of degree $D$ that is not a root of unity, the number of $M$-tuples $\mathbf n \in \Z^M$ with $\norm{\mathbf n}_{*} < Ct$ is $\ll C^Mt^M$. Since for any $\gamma\in\TT(\mathcal{O}_S)$ we have that $\phi^{-1} (\gamma) \in \left({\mathcal O}^{\times}_{\K,\overline S}\right )^k$ we may conclude that
\[
\# \{\gamma\in\TT(\mathcal{O}_S) \quad \text{satisfying \eqref{eq:height bound}}\}\ll t^M
\]
 \end{proof}

Combining Lemmas \ref{semisimple-gamma}--\ref{fewlementsintorus} it follows that
the case of fast multiple recurrence
in Theorem \ref{zero-or-fast-pol-rec} cannot hold. 

\section{Conclusion of the proof of Theorem~\ref{sl2-thm-final}}
\label{easy-end}

Theorem \ref{zero-or-fast-pol-rec} and the discussions in 
\S\ref{sec:not so fast} 
prove that
\[
 h_\mu(b,V|\mathcal{A}_V)>0. 
\]
However, this is the
assumption of \S\ref{secunipotent} and Lemma \ref{pos-ent-implies-invariance}. Applying Theorem~\ref{MTclassA} we find an algebraic group~$\LL$ which 
describes the ergodic components for the unipotently
generated subgroup preserving~$\mu$. 
We will describe now which possibilities we have for
the~$\Q$-subgroup~$\LL<\GG$ by using that~$\GG$ is~$\Q$-almost simple
and a form of~$\SL_2^k$ or~$\PGL_2^k$.

\begin{proposition}\label{subgroup-L-of-sl2}
	Let~$\GG$ be a~$\Q$-almost simple~$\Q$-group
	that is a form of~$\SL_2^k$ (or of~$\PGL_2^k$). 
	Let~$\LL<\GG$ be a nontrivial Zariski-connected~$\Q$-subgroup.
	Then 
 the projection of $\LL$ to any of the almost 
 simple factors over $\overline{\Q}$ is nontrivial
 and one of the following statements holds.
	\begin{itemize}
		\item[(S)] $\LL$ is semi-simple, and the normalizer of~$\LL$ contains~$\LL$ as a finite index subgroup.
		 Moreover, $\LL$ projects onto 
  each of the almost simple factors over $\overline{\Q}$. 		\item[(U)] $\LL$ is unipotent and abelian,
 the connected component of the centralizer of $\LL$ is the unipotent
 radical of a Borel subgroup of $\GG$, 
 and the normalizer~$N_{\GG}(\LL)^\circ$ is solvable. In particular,~$\GG$ has positive~$\Q$-rank.
		
		\item[(T)] $\LL=\LL_t\LL_u$ is a semidirect product
		of a nontrivial~$\Q$-torus~$\LL_t<\LL$ and its (possibly trivial) unipotent radical~$\LL_u\lhd\LL$.
	\end{itemize}
  The subgroup~$\LL$ from \S\ref{secunipotent} is never as in~(T).
\end{proposition}

\begin{proof}
We recall that the absolute Galois group of~$\Q$ acts on the various algebraic~$\overline\Q$-subgroups of~$\GG$
and that~$\GG$ and~$\LL$ as~$\Q$-groups are fixed by this action. 
Since~$\GG$ is isomorphic to~$\SL_2^k$ or~$\PGL_2^k$ over~$\overline\Q$, we see that the almost simple normal $\overline{\Q}$-subgroups
are permuted under the Galois action. Moreover, since~$\GG$ is assumed to be almost simple over~$\Q$
the action of the Galois group is transitive on the almost simple~$\overline \Q$ factors of~$\GG$ (for otherwise
the product over all factors in an orbit would be a normal~$\Q$-subgroup of~$\GG$). We enumerate
the $\overline{\Q}$-almost-simple normal factors of~$\GG=\HH_1\times\cdots\times\HH_k$, and write~$\pi_j:\GG\to\HH_j$
for the projection maps for~$j=1,\ldots,k$. Let $\mathfrak{h}_j$ be the Lie algebra of~$\H_j$ for~$j=1,\ldots,k$. 
Since~$\HH_1$ is isomorphic to~$\SL_2$ or to~$\PGL_2$ 
the subgroup~$\pi_1(\LL)<\HH_1$ is either trivial, all of~$\HH_1$, a unipotent subgroup, a torus subgroup, or
the Borel subgroup of~$\HH_1$. By transitivity of the Galois action the subgroups~$\pi_j(\LL)<\HH_j$
are for every~$j=2,\ldots,k$ of the same type as the subgroup~$\pi_1(\LL)$.

If~$\pi_j(\LL)$ is trivial for all~$j=1,\ldots,k$ then~$\LL<\pi_1(\LL)\times\cdots\times\pi_k(\LL)$ is trivial too, which 
contradicts our assumption.

If~$\pi_j(\LL)$ is a torus subgroup for all~$j$,
then~$\LL<\pi_1(\LL)\times\cdots\times\pi_k(\LL)$ is also a torus subgroup. 
Similarly if~$\pi_j(\LL)$ is the Borel subgroup for all $j$, then $\LL$ is solvable. The remaining claim in (T) follows
from the Levy decomposition of $\Q$-groups.

If~$\pi_j(\LL)$ is unipotent for all~$j=1,\ldots,k$, then~$\LL<\pi_1(\LL)\times\cdots\times\pi_k(\LL)$
is also unipotent as in~(U). We also have
in this case that $\pi_1(\LL)\times\cdots\times\pi_k(\LL)$
is the connected component of the centralizer of $\LL$ and the unipotent radical of a Borel subgroup of $\GG$. 
Moreover,~$\LL(\mathcal{O}_S)=\GG(\mathcal{O}_S)\cap \LL(\Q_S)$
is a lattice in~$\LL(\Q_S)$ and so nontrivial,
which implies that~$\GG(\mathcal{O}_S)$ has unipotent elements, the locally homogeneous space~$X=\GG(\mathcal{O}_S)\backslash \GG(\Q_S)$
is noncompact, or equivalently that~$\GG$ has positive~$\Q$-rank. 
 The claim in~(U) regarding~$N_{\GG}(\LL)^\circ$ follows
by applying the case (T) above also to~$N_{\GG}(\LL)^\circ$.

So let us assume now that~$\pi_j(\LL)=\HH_j$ for~$j=1,\ldots,k$ and let~$\mathfrak{l}<\sl_2^k$ 
denote the Lie algebra of~$\LL$. We claim that~$\mathfrak{l}$ is a diagonally embedded copy of~$\sl_2^\ell$
for some $\ell$ dividing $k$. More precisely after permuting the indices and applying
some inner automorphism over $\overline{\Q}$ (if necessary) $\mathfrak{l}$ takes the form
\[
 \bigl\{(v_1,\ldots,v_\ell,v_1,\ldots,v_\ell,\ldots,v_1,\ldots,v_\ell)\mid
v_1,\ldots,v_\ell\in\sl_2\bigr\}
\]
Note that the connected component of the normalizer of a semisimple Lie group (or Lie algebra)
equals the Lie group times the centralizer of the Lie group. Therefore, the claim implies the statement in~(S). 

To prove the claim we let~$A$ be a minimal 
subset of~$\{1,\ldots,k\}$ with the property
that the intersection
\[
 \mathfrak{l}\cap \mathfrak{h}_A\neq 0 \qquad \mathfrak{h}_A  = \sum_{j\in A}\mathfrak{h}_j.
\]
Let~$B$ be another minimal subset with this intersection property 
and suppose that~$A$ and~$B$ intersect nontrivially.
We wish to show that this implies that~$A=B$. So suppose~$j_0\in A\cap B$ and let~$v\in\mathfrak{l}\cap\mathfrak h_A$
be a nonzero vector. Taking the commutators with any~$w\in\mathfrak{l}$
and using that~$\pi_{j_0}(\mathfrak{l})=\mathfrak h_{j_0}$ is semisimple
we see that the component~$v_{j_0}=\pi_{j_0}(v)\in\mathfrak h_{j_0}$ can be any vector in~$\mathfrak h_{j_0}$. For that reason there 
exists some~$w\in\mathfrak{l}\cap\mathfrak h_B$ so that~$w_{j_0}=\pi_{j_0}(w)$ does not commute with~$v_{j_0}$. 
In particular, we see that~$[v,w]\in\mathfrak{l}\cap\mathfrak h_{A\cap B}$ has a nonzero~${j_0}$-th component. 
As~$A$ and~$B$ were minimal subsets with the intersection property this implies~$A=B$. 

Using the same argument it follows that $\pi_j(\mathfrak{l}\cap\mathfrak h_A)=\sl_2$ for any $j\in A$. 
Also by definition~$\mathfrak{l}\cap\mathfrak h_A$ has trivial
intersection with any subspace~$\mathfrak h_B$ for a proper 
subset $B\subsetneq A$. 
Using only linear algebra this implies 
that~$\mathfrak{l}\cap\mathfrak h_A$ must be a diagonally embedded 
three-dimensional subspace of~$\mathfrak h_A$. As it is also 
a Lie algebra, it equals the simulteneous graph of~$|A|-1$
many Lie algebra automorphisms of~$\sl_2$ with common 
domain equal~$\mathfrak h_{j_0}$ for some fixed~$j_0\in A$.
As all automorphisms of $\sl_2$ are inner over $\Q$, we may also apply
some inner automorphism to $\mathfrak{l}$ and assume 
that $\mathfrak{l}\cap\mathfrak h_A$ equals 
a diagonally embedded copy of~$\sl_2$.

For any~$j_0\in\{1,\ldots,k\}$ there exists a unique minimal subset~$A_{j_0}\subset\{1,\ldots,k\}$
with the intersection property. As the Galois action is transitive on the indices
we see that the subsets are all of equal size and constitute a partition of~$\{1,\ldots,k\}$.
This implies the claim.

Let us now prove the final claim of the proposition: 
In fact, we have found the subgroup~$\LL$ in \S\ref{secunipotent}
as the algebraic group describing the~$H$-ergodic components of~$\mu$ where~$H$
was generated by unipotents. However, in the case~(T) the unipotent elements
of~$\LL$ are all contained in~$\LL_u\lhd\LL$ 
and these do not act ergodically on the 
orbit~$\Gamma\LL(\Q_S)'$ (as it factors onto
a locally homogeneous space defined by~$\LL_t(\R)'$ 
on which~$H<\LL_u(\Q_S)$ acts trivially). 
\end{proof}

Note that it also follows from the proof that if $\GG$ is a form of $\SL_2^k$ and $\LL$ is as in~(S) then
$\LL$ is simply connected, as $\LL(\overline\Q)$ is isomorphic to a product of $\SL_2$'s.

\subsection{Proof of Theorem~\ref{sl2-thm-final} for \texorpdfstring{$\LL$}{L} semi-simple}\label{easy-2}

We suppose first that the subgroup
$\LL$ in Theorem \ref{MTclassA} 
is semi-simple as in 
Proposition \ref{subgroup-L-of-sl2} (S). 
In particular, $\LL$ equals the Zariski connected component of its normalizer.
Also recall that~$\mu$ is~$L$-invariant for a finite 
index subgroup $L<\LL(\Q_S)$ and both of these groups 
are normalized by $A$. Since $A<N_G(\LL)$ it follows that 
$A\cap L$ has finite index in $A$. Since $A$ is 
class-$\mathcal{A}'$ we obtain $A<\LL(\Q_S)$ (e.g.\ by considering the action of $A$ in the Chevalley representation for $\LL$). 
Therefore, $\mu$ is invariant under $L'=LA<\LL(\Q_S)$ 
and is supported on the closed orbit~$\Gamma L'$,
which proves the first possibility in 
Theorem~\ref{sl2-thm-final} in the case at hand.

Note that if $\GG$ is a form of $\SL_2^k$, then by the remarks following the proof of Proposition~\ref{subgroup-L-of-sl2} the group $\LL$ is simply connected, hence as $\LL(\Q_S)$ is non compact we have $\LL(\mathcal O_S) L'=\LL(\Q_S)$
by strong approximation \cite[\S7]{Platonov-Rapinchuk}. It follows that in this case we may take $L=L'=\LL(\Q_S)$.

We also note that if $X$ is compact,
then $\Gamma$ cannot contain any 
unipotent elements. This in turn
implies that there are no unipotent $\Q$-subgroups of $\GG$. 
By Proposition~\ref{subgroup-L-of-sl2} we conclude that~$\LL$
is semi-simple and the above discussion applies.

\subsection{Proof of Theorem~\ref{sl2-thm-final} for \texorpdfstring{$\LL$}{L} unipotent}\label{sec:embedding}

By Theorem \ref{MTclassA} 
and Proposition~\ref{subgroup-L-of-sl2}
the only other case to consider is
when $\mu$ is invariant under a finite index subgroup $L<\LL(\Q_S)$
for a unipotent $\Q$-group $\LL$. 
As $\LL(\Q_S)$ has no finite index subgroups, 
we have in fact $L=\LL(\Q_S)$. 
Moreover, by Proposition~\ref{subgroup-L-of-sl2}
$N_G(\LL)$ is a solvable subgroup of $G$.
Finally, $\mu$ is supported on a single closed orbit of the Haar measure preserving normalizer~$N_G^1(\LL)$.

To prove Theorem \ref{sl2-thm-final} it remains to upgrade this to the statement that $\mu$ is supported on a compact orbit. 
For this we first note that for $\GG$ as in Theorem~\ref{sl2-thm-final} positive $\Q$-rank implies that~$\GG=\operatorname{Res}_{\K\mid\Q}\operatorname{SL}_2$ 
or~$\GG=\operatorname{Res}_{\K\mid\Q}\operatorname{PGL}_2$
for some number field~$\K\mid\Q$, see \cite[Ch.~4]{Platonov-Rapinchuk}. 
We define $\mathbb{U}$
as the connected component of the centralizer of $\LL$. By Proposition \ref{subgroup-L-of-sl2} $\UU$ is the unipotent radical of a Borel subgroup $N_\GG(\UU)$. 

Let $\mathbb{B}$ denote the upper triangular
Borel subgroup of $\SL_2$ resp.\ $\PGL_2$ so that $\operatorname{Res}_{\K|\Q}\mathbb{B}<\GG$ defines a $\Q$-subgroup which is also a Borel subgroup of $\GG$. 
Note that any two Borel subgroups (as minimal parabolic subgroups)
defined over $\Q$ are conjugated by an element of $\GG(\Q)$, which allows us to pass statements (up to finite index issues) from one Borel subgroup to the other.
In particular, Dirichlet's $S$-unit theorem now implies that the orbit $\Gamma N_G^1(\UU)$ is in fact compact.
By construction we have $\N_\GG(\LL)<N_\GG(\UU)$ and the modular characters
of conjugation on $\LL$ and conjugation on $\UU$ both define nontrivial $\Q$-characters. As $\GG$ has $\Q$-rank one, these need to be rationally dependent, which implies that $N_G^1(\LL)<N_G^1(\UU)$. Put together
we now obtain that $\Gamma N_G^1(\LL)$
is closed by Theorem \ref{MTclassA}
and contained in the compact orbit $\Gamma N_G^1(\UU)$. This concludes the proof of Theorem \ref{sl2-thm-final}.

\section{The Adelic Measure Classification}

As in \S\ref{s:adelic},
let $\KK$ be a number field, and let
\[
X_{\Adel_\K}=\SL_2(\K)\backslash \SL_2(\Adel_\K).
\]
We set $\GG=\operatorname{Res}_{\KK|\QQ}\SL_2$ and $G=\GG(\Adel_\Q)$ which we identify with $\SL_2(\Adel_K)$.

Let $\TT<\SL_2$ be the group of diagonal matrices, and let $\Lambda <\Adel_\KK ^\times$ be a discrete subgroup satisfying
\begin{enumerate}[label=\textup{(A\arabic*)},leftmargin=*]
\item \label{item:big-2} There are $C,\kappa>0$ so that for any $t>0$,
\[
|\{ k \in \Lambda: h(k)\leq t \}| \geq C e^{\kappa t}
\]
\end{enumerate}
Set, for $k=(k_\sigma)\in \Adel_\KK$,
\[
\sa k= (\sa {k_\sigma})_\sigma \text{ \ with \ }\sa {k_\sigma}=\begin{pmatrix}
k_\sigma&\\&k_\sigma^{-1}
\end{pmatrix},
\]
and $A = \left\{ \sa k: k \in \Lambda \right\}<\TT(\Adel_\KK)$.
Let~$\mu$ be $A$-invariant and ergodic probability measure $\mu$ on $X_{\Adel_\K}$.
Let $S_\KK$ denotes the set of places of $\KK$, \ $S_{\KK,\infty}$ the set of $\sigma \in S_\KK$ with $\sigma \mid \infty$, and $S_{\KK,\mathrm{f}} = S_\KK \setminus S_{\KK,\infty}$. Also, for any place $\sigma \in S_\KK$ we let $\mathcal O_{(\sigma)}$ denote the maximal order in the local field $\KK_\sigma$ (e.g. if $\KK=\QQ$, \ $\mathcal O_{(p)}=\Z_p$).
We say that $\Lambda$ is bounded at $\sigma \in S_\KK$ if $\norm {k} _ \sigma = 1$ for every $k \in \Lambda$. We say that it is unbounded at $\sigma$ otherwise.

\medskip

Suppose first that this $A$-invariant and ergodic probability measure $\mu$ 
has zero entropy. We want to show (at least if $\Lambda$ is rich enough) that there is a $\Q$-torus $\HH$ and a $g \in \GG(\Adel_\Q)$ so that
    \[
    \mu(\Gamma \HH(\Adel_\Q)g)=1
    \]
    with $A < g^{-1}\HH(\Adel_\Q)g$.

\subsection{Decay rates of Bowen balls}\label{sec:defofa0}

For now fix some $k_0\in\Lambda$ and
let $a_0=\sa{k_0}\in A$. 
For $r\in(0,1]$ we define the Bowen ball for $a_0$ and $n\in\N$ by
\begin{equation}\label{adelicBowen}
     D^{a_0}_{r,n}=\bigcap_{k=-n}^n a_0^k B_r^{G} a^{-k}_0,
\end{equation}
where we will assume that the metric satisfies
\[
 B_1^{G}\subset\prod_{\sigma \in S_{\KK,\infty}} \SL_2(\KK_\infty)\times\prod_{\sigma\in S_{\KK,\mathrm{f}}}\SL_2(\mathcal O_{(\sigma)}).
\]
We chose $r$ small enough so that the intersection of $B_{10r}(G)$ with any conjugate of~$A$ in $\SL_2(\Adel_\KK)$ is the identity.

As in \S\ref{recurrencesection}, the measure of such Bowen balls relate to the entropy of $a_0$ with respect to an $A$-invariant and ergodic probability measure $\mu$. In fact for $\mu$-a.e.\ $x$ and any sufficiently small $r>0$ we have that
\[
-\tfrac1n\log\mu(x D^{a_0}_{r,n})\rightarrow 
h_\mu(a_0)\qquad\text{as $n\to\infty$}.
\]
Assume now that the entropy of $a_0$ with respect to $\mu$ vanishes. Then in particular for $\alpha>0$ arbitrary 
we have that for $\mu$-a.e.\ $x$
\begin{equation}\label{eq:adelic-0-decay}
\mu\bigl(x D^{a_0}_{r,n}\bigr)>e^{-\alpha n} \qquad\text{for all sufficiently large $n$.}
\end{equation}

\subsection{Using the growth condition on \texorpdfstring{$A$}{A}}

Starting with a point $x=\Gamma g$ and $n$ satisfying \eqref{eq:adelic-0-decay} for some $\alpha>0$, consider the union
\[
\bigcup_{k\in \Lambda :h(k)<\delta n} x D_{r,n}^{a_0}\sa k,
\]
where $\delta = 2\alpha/ \kappa$,  with $\kappa$ as in \ref{item:big-2}. All the sets appearing in the union have the same measure which, if $n$ is large enough,  is at least $e^{-\alpha n}=e^{-\delta\kappa n/2}$ by  \eqref{eq:adelic-0-decay}. By \ref{item:big-2} we have $\gg e^{\kappa\delta n}$ such
sets included in the union. Hence  for large enough $n$ at least two of these sets must overlap. 
In other words for each such $n$ there exist two distinct $k_1, k_2\in \Lambda$ with $h(k_1),h(k_2)<\delta n$  (depending on $n$) and $\gamma_n\in\Gamma$ so that for $a_i=\sa {k_i}$ it holds that
\begin{equation*}
     \gamma_n g h_1 a_1= g h_2 a_2
\end{equation*}
for some $h_1,h_2\in D_{r,n}^{a_0}$, or equivalently
\begin{equation}\label{new-nice-gamma}
    \gamma_n=g (h_2 a_{2}a_{1}^{-1} h_1^{-1}) g^{-1}\in\SL(\KK).
\end{equation}
Note that our choice of $r$ guarantees that $\gamma_n\neq e$ as $a_1\neq a_2$. Set $k_{(n)}=k_2k_1^{-1}$, $a_{(n)}=a_{2}a_{1}^{-1}=\sa {k_{(n)}}$, and note that $h(k_{(n)})\leq 2\delta n$.

\subsection{Analysing the lattice elements}\label{sec:lattice-torus}
Similar to the discussion in Section \ref{sec:not so fast} we want to show that \eqref{new-nice-gamma} imposes rather sever restrictions on the $\gamma_n$.
Specifically, we show that analogues of Lemmas~\ref{semisimple-gamma}--\ref{samecommutator} holds. We note however a key difference between the $S$-arithemtic setting in \S\ref{sec:not so fast} and the Adelic torus we study here: in contrast to Lemma~\ref{fewlementsintorus}, which gives polylog bounds on the number of $S$-integral points in a $\Q$-torus, an adelic torus contains polynomially in $T$ many point of ``size'' $<T$. 

\begin{lemma}[Semi-simple elements]\label{semisimple-gamma-adelic}
	If~$x=\Gamma g$ satisfies~\eqref{eq:adelic-0-decay} for~$\alpha>0$ small enough, and if~$r>0$ is sufficiently small, and~$n$ is sufficiently large,
	then $\gamma_n$ is semi-simple and non-central. 
\end{lemma}

Before we set to prove Lemma~\ref{semisimple-gamma-adelic}, we need a few preliminaries. For $g _ \sigma \in \SL _ 2 (\KK _ \sigma)$ we let $\norm g _ \sigma$ denote the operator norm of $g$ with respect to the usual norm on $\KK _ \sigma ^2$ (i.e.\ the Euclidean norm if $\sigma$ is infinite, and the maximum norm if $\sigma$ is finite). Since for any $2\times 2$ matrix $\begin{pmatrix} m_{ij} \end{pmatrix}$
\[
\norm {\begin{pmatrix} m_{ij} \end{pmatrix}} _ \sigma \leq \max \absolute {m _ {ij}} _ \sigma
\]
it holds that for any nonzero $z \in M_2(\KK)$,
\begin{equation}\label{eq:norm product}
\prod_ {\sigma \in S_\KK} \norm z _ \sigma ^ {d _ \sigma} \geq 1
.\end{equation}

\begin{proof} [Proof of Lemma~\ref{semisimple-gamma-adelic}]
Let $S_1$ be the set of places $\sigma \in  S_\KK$ for which $\absolute {k _ 0} _ \sigma \neq 1$, $S _ 3$ be the finite set of those places in $S_\KK \setminus S _ 1$ which either are infinite places or places in which the $\sigma$-component of $g \in \SL _ 2 (\Adel_{\KK})$ is not in $\SL _ 2 (\mathcal{O}_{(\sigma)})$, and let $S _ 2$ be all other places.

As in the proof of Lemma~\ref{semisimple-gamma}, the statement reduces easily to proving that $\gamma _ n$ cannot be a unipotent element of $\SL _ 2 (\KK)$.
Suppose in contradiction that $\gamma _ n$ is a nontrivial unipotent element. 

It follows from \eqref{new-nice-gamma} that for every $\sigma \in S _ 2$ for which $\absolute {k_{(n)}}_\sigma=1$, $\gamma _ n \in \SL _ 2 (\mathcal{O}_{(\sigma)})$, hence
\begin{equation*}
\norm {\gamma _ n - I} _ \sigma \leq 1
,\end{equation*}
and more generally
\begin{equation*}
\log \norm {\gamma _ n - I} _ \sigma \leq  \absolute{\log | {k_{(n)}}|_\sigma}.
\end{equation*}
While we cannot be as precise for $\sigma \in S _ 3$, we can similarly say that in this case
\begin{equation*}
\log \norm {\gamma _ n - I} _ \sigma \leq  \absolute{\log | {k_{(n)}}|_\sigma} + O(1)
,\end{equation*}
with the $O(1)$ depending only on $\KK,g$.

Similar estimates hold for $\sigma \in S _ 1$, but for these we can do better using the assumption that $\gamma _ n$ is unipotent. Taking the trace of both sides of \eqref{new-nice-gamma} we see that for any $\sigma \in S _ 1$ the  component of $h _ 2 a _ {(n)} h _ 1$ 
corresponding to $\sigma$ (in the notations of \eqref{new-nice-gamma}) has trace 2 and is $\ll e^{-cn}$ close to $\T (\KK _ \sigma)$, where $c > 0$ is some constant that depends on $a_0$. This implies that for $\sigma \in S _ 1$
\begin{equation} \label{eq:S_1 places}
\norm {\gamma_n -I} _ \sigma \ll _ {g} e ^ {- c n / 2}
.\end{equation}

Setting $z=\gamma_n-I$, taking the logarithm of \eqref{eq:S_1 places}, and summing over all places, we get using \eqref{eq:norm product} that
\begin{equation}\label{calculation giving contradiction}
    \begin{aligned}
0 & \leq \sum_ {\sigma \in S_\KK} d _ \sigma \log \norm z _ \sigma \\
& \leq 
- \frac{cn}{2} \sum_ {\sigma \in S _ 1} d _ \sigma+\sum_ {\sigma \in S _ 2 \cup S _ 3} d _ \sigma\absolute{\log | {k_{(n)}}|_\sigma} + \sum_ {\sigma \in S _ 1\cup S_3} O(1)  \\
& \leq - \frac{cn}{2}+2 h (k_{(n)}) + O(1) \leq \bigl(4 \delta-\frac{c}2\bigr) n + O(1)
,\end{aligned}
\end{equation}
with the $O (1)$ independent of n. Choosing $\alpha$ small enough so that $\delta = 2\alpha/\kappa < c / 8$ gives a contradiction once $n$ is sufficiently large.
\end{proof}

A similar calculation, this time evaluating $\norm {z}_\sigma$ for $z= \gamma _ n \gamma _ {n +1} - \gamma _ {n +1} \gamma _ n$, and again using \eqref{new-nice-gamma} gives the following analogue of Lemma~\ref{samecommutator}. We leave the details to the reader.

\begin{lemma}[Commuting]\label{samecommutator-adelic}
	If~$x=\Gamma g$ satisfies~\eqref{eq:adelic-0-decay} for~$\alpha>0$ sufficiently small, and if~$r>0$ is sufficiently small, and~$n$ sufficiently large,
  then the elements
  $\gamma_n$ and $\gamma_{n+1}$ defined 
  using \eqref{new-nice-gamma}
  commute.
\end{lemma}

Note that any two commuting non-trivial
semi-simple elements of $\SL_2$ have the same
centralizer. Hence the centralizer of $\gamma_n$ stabilizes for $n$ large enough, giving a $\KK$-torus $\HH_g=C_{\SL_2}(\gamma_n)$. Moreover, any choice of $\gamma_n$ satisfying \eqref{new-nice-gamma} is in $\HH_g$ and we obtain the following:

\begin{lemma}[A single torus subgroup]\label{single torus lemma}
	If~$x=\Gamma g$ satisfies~\eqref{eq:adelic-0-decay}for~$\alpha>0$ sufficiently small, $\delta = 2\alpha/\kappa$, and if~$r>0$ is sufficiently small, then there exists
      some integer $n_0$ and a torus subgroup $\HH_g<\SL_2$
      defined over $\K$ so that if $k \in \Lambda$ and $\gamma\in \SL_2(\KK)$ with $\gamma\neq I$ satisfy for $n>n_0$ that 
      \begin{equation}\label{defining equation for H}
          h(k)\leq 2\delta n\qquad\text{and}\qquad
          \gamma \in g D^{a_0}_{r,n} \sa k (D^{a_0}_{r,n}) ^{-1} g^{-1}
      \end{equation}
      then $\gamma\in \HH_g(\KK)$ and $\HH_g=C_{\SL_2}(\gamma)$. For any $n>n_0$ at least one such pair $(k,\gamma)$ exists. The group $\HH_g(\KK)$ is uniquely determined by $x$ up to conjugation by $\SL_2(\KK)$ and uniquely determined once the representative $g$ of the coset $x=\Gamma g$ is chosen.
\end{lemma}

We remark that in the above lemma $\alpha>0$ (or equivalently $\delta>0$) and $r>0$ were small enough depending on $a_0$, but for any choice of $\delta,r>0 $ there would be sufficiently many $\gamma$ to nail down $\HH_g$.

For $k, k ' \in \KK$, we say that $k \llcurly k'$ if for every $\sigma \in S_\KK$ one has that \[\absolute {\log|k| _ \sigma} \leq \absolute {\log|k'| _ \sigma}.\]
It follows that if $k \llcurly k'$ and $a = \sa k$ and $a ' = \sa k '$ then $D^{a}_{r,n} \subseteq D^{a'}_{r,n} $. Thus if in addition $h_\mu (a) = h_\mu (a ') = 0$ the group $\HH_g$ one obtains for $a_0=a'$ automatically satisfies the defining conditions for $\HH_g$ for $a_0=a$. Since for every $k, k ' \in \KK ^ \times$ one can find an element $k'' $ in the semigroup generated by $k$ and $k'$ so that both $k \llcurly k''$ and $k' \llcurly k''$ we obtain the following important observation:

\begin{lemma}[Independence from the diagonal element]\label{torus does not depend}
Suppose $h_\mu (a) = 0$ for every $a \in A$. Then the $\KK$-torus $\HH_g$ defined in Lemma~\ref{single torus lemma} does not depend on $a _ 0$.
\end{lemma}

\subsection{Analysing the \texorpdfstring{$\KK$-torus $\HH_g$}{K-torus}}

Having found the subgroup $\HH_g$ we now wish to understand its
relationship with the initial point $x$ better.

\begin{lemma}\label{lemma: on the diagonal}
 Let $x=\Gamma g$ be a point satisfying \eqref{eq:adelic-0-decay} for sufficiently small $\alpha>0$, and let $\HH_g$
 be the torus subgroup found in \S\ref{sec:lattice-torus}. 
 Then $g^{-1}\HH_g g$ is the diagonal subgroup at all places $\sigma\in S_{\K}$ for which $|a_0|_\sigma\neq 1$. 
\end{lemma}

\begin{proof}
Let $\gamma _ n$ and $k$ be as in \eqref{defining equation for H} for $n$ large enough, and $\beta > 0$ small enough depending on $a _ 0$ to be determined later. Let $\sigma \in S_\KK$ be so that $\absolute {a _ 0}_\sigma \neq 1$, and let $g _ \sigma$ denote the $\sigma$-component of $g$.

We claim that for any $\beta>0$, if $\alpha$ is sufficiently small, for all $n$ large enough
\begin{equation}\label{eq:estimate at sigma}
\norm {\gamma _ n - I} _ \sigma \geq e ^ {-  \beta n},
\end{equation}
and so also  $\norm {g_ \sigma ^{-1} \gamma _ n g_ \sigma - I} _\sigma\gg_{g} e ^ {-  \beta n}$.
Indeed, suppose \eqref{eq:estimate at sigma} were false. Arguing analogously to \eqref{calculation giving contradiction} in the proof of Lemma~\ref{semisimple-gamma-adelic}, with the negation of \eqref{eq:estimate at sigma} used instead of \eqref{eq:S_1 places} and $\{\sigma\}$ instead of $S_1$ we see that the negation of \eqref{eq:estimate at sigma} leads to a contradiction for large $n$ once $\alpha$ (and hence $\delta$) are small enough.

On the other hand by \eqref{defining equation for H} we have that $g_ \sigma ^{-1} \gamma _ n g_ \sigma$ is very close to a diagonal matrix.
Indeed, if $h \in D ^ {a _ 0} _ {r, n} \subset \SL _ 2 (\Adel_{\KK})  $ then the $\sigma$-components of $h$ is within $e^{-cn}$ of some element in the unit ball in $\T (\KK _ \sigma)$, where $c>0$ is a constant depending only on $a _ 0$.
Since $h(k) < 2\delta n$
that there is a $t _ \sigma \in \T (\KK _ \sigma)$ with $\norm {t _ \sigma} _ \sigma \ll e ^ {2\delta n}$ so that $\norm {g_ \sigma ^{-1} \gamma _ n g_ \sigma - t _ \sigma} \ll e^{-cn+2\delta n}$. 

Combining the last two paragraphs it follows that the unit vectors $\left(\begin{smallmatrix} 1 \\ 0 \end{smallmatrix}\right)$ and $\left(\begin{smallmatrix} 0 \\ 1 \end{smallmatrix}\right)$ are within $e^{-cn+O(\max(\beta,\delta))n}$ of the eigenvectors of $g_ \sigma ^{-1} \gamma _ n g_ \sigma$. But by Lemma~\ref{single torus lemma} the eigenvectors of $g_ \sigma ^{-1} \gamma _ n g_ \sigma$ do not depend on $n$. Since we can take both $\beta$ and $\delta$ very small (in comparison to $c$), and $n$ as large as we wish, the eigenvalues of $g_ \sigma ^{-1} \gamma _ n g_ \sigma$ are \emph{exactly} $\left(\begin{smallmatrix} 1 \\ 0 \end{smallmatrix}\right)$ and $\left(\begin{smallmatrix} 0 \\ 1 \end{smallmatrix}\right)$. Thus $g_ \sigma ^{-1} \gamma _ n g_ \sigma$ is a nontrivial element of $\T(\KK_\sigma)$, and hence $\HH_g(\KK_\sigma)$, the centralizer of $\gamma_n$ in $\SL_2(\KK_\sigma)$, is equal to $g_ \sigma \T(\KK_\sigma)g^{-1}$.
\end{proof}

\subsection{Wrapping up the zero entropy case of Theorem~\ref{thm:adelic}}
With this we can now conclude the zero entropy case. For the statement of the following theorem, we make note of the following additional possible assumption on $\Lambda$:

\begin{enumerate}[label=\textup{(A\arabic*)},start=2,leftmargin=*]
\item \label{item:indep-2} For any (possibly infinite) set of places $S \subsetneq S_\KK$, there is an element $k \in \Lambda$ so that $\prod_{\sigma\in S}|k|_\sigma \neq 1$. 
\end{enumerate}
Note that \ref{item:indep-2} in particular implies that $\Lambda$ is unbounded at \emph{all} places.

\begin{theorem}[Zero entropy]\label{adelic-zeroent} Let $\KK$ be a number field, $\Gamma = \SL_2(\K)$, \ $\TT<\SL_2$ be the subgroup of diagonal matrices,
\[
X_{\Adel_\K}= \Gamma \backslash \SL_2(\Adel_\K),
\]
and $\Lambda <\Adel_\KK ^\times$ a discrete subgroup satisfying \ref{item:big-2}.
Let $A = \left\{ \sa k: k \in \Lambda \right\}<\TT(\Adel_\KK)$ and let $\mu$ be an $A$-invariant and ergodic probability measure. Then either there is some $a _ 0 \in A$ for which $h _ \mu (a _ 0) > 0$, or all of the following hold (assuming $\Lambda$ satisfies the corresponding successively stronger conditions):
\begin{enumerate}[label=\textup{(Z\arabic*)}]
\item\label{prop item 1} there is a $\KK$-torus $\HH < \SL _ 2 $ and $g \in \SL _ 2 (\Adel_{\KK})$ so that $\mu$ is equivalent-by-compact to an $A '$-invariant measure $\mu '$ satisfying
\begin{equation*}
\supp \mu ' \subseteq \Gamma \HH (\Adel_\KK) g \text{ \ and \ } g  A ' g ^{-1} \subseteq \HH (\Adel_{\KK})
.\end{equation*}
\item\label{prop item 2}
Assume further that the set of places $\sigma$ at which $\Lambda$ is unbounded is of density $> 50\%$. Then we may take $\HH = \T$ and the element $g$ in \ref{prop item 1} to be in the normalizer of $\TT(\Adel_\K)$.
\item\label{prop item 3} Assume further that $\Lambda < \KK^\times$, and that it satisfies \ref{item:indep-2}. Then there exists a point $x _ 0 \in \T (\KK) \backslash \T (\Adel_{\KK})$ so that $\mu = \delta _ {x _ 0}$. 
\end{enumerate}
\end{theorem}

\begin{proof}
Assume $h _ \mu (a) = 0$ for every $a \in A$.

We first establish \ref{prop item 1}. Let $S_\Lambda \subseteq S_\KK$ be the set of places in which $\Lambda$ is unbounded. Let $\Omega < \TT(\Adel_\KK)$ be the compact group 
\begin{equation}\label{eq: Omega}
\Omega = \{\sa k: \absolute{k_\sigma}_\sigma = 1 \text{ for all $\sigma\in S_\KK$}\}.
\end{equation}

Lift $\mu$ from $X _ \KK$ to a left $\Gamma$-invariant measure $\tilde \mu$ on $\SL _ 2 (\Adel_{\KK})$ in the obvious way. By Lemmas~\ref{single torus lemma}--\ref{lemma: on the diagonal}  
for $\tilde\mu$-almost every $g = (g _ \sigma) \in \SL _ 2 (\Adel_{\KK})$ there is a $\KK$-torus $\HH _ g$ so that 
\begin{equation}\label{eq: HH determines g}
    \HH_g (\KK _ \sigma) = g_\sigma \T (\KK _ \sigma)  g_\sigma^{-1} \qquad \text{for every $\sigma\in S_\Lambda$}.
\end{equation}
The group $\HH _ g$ depends in a measurable way on $g$, and furthermore it follows from the definition that for any $a \in A$ it holds that $\HH _ {ga} = \HH _ {g} $.

Since there are only countably many $\KK$-groups, and since $\mu$ is ergodic, it follows that there is a subgroup $\HH$ so that for $\mu$-almost every $x$, if $x = \Gamma g$, then $\HH _ g$ can be conjugated by an element of $\Gamma$ to $\HH$.

For $\mu$ a.e.\  $x=\Gamma g$ there exists a $\check g \in  \SL_2(\Adel_\KK)$ 
with $\Gamma g=\Gamma\check g$ so that  $\HH _ {\check g} = \HH$. 
Then by \eqref{eq: HH determines g}, $\HH (\Adel_\KK) \check g \Omega$ is $A$-invariant. Given $x$, the element $ \check g$ is uniquely determined up to left multiplication by an element of the normalizer of $\HH$ in $\SL_2(\KK)$. This normalizer is either $\HH (\KK)$ if $\HH$ is $\KK$-anisotropic or can be written as $\HH (\KK) \sqcup w \HH (\KK)$ for an appropriate element $w \in \SL_2(\KK)$, so given $x$ either $\HH (\Adel_\KK) \check g \Omega$ or the pair
\[
(\HH (\Adel_\KK) \check g \Omega, w \HH (\Adel_\KK) \check g \Omega)
\] is well defined, and is constant on the $A$-orbit of $x$. By ergodicity this map is constant a.e., which implies 
in either case that $\mu$ is supported on $\Gamma \HH (\Adel_\KK) \check g \Omega$ for some $\check g\in\SL_2(\Adel_\KK)$. This implies \ref{prop item 1} readily.

We now establish \ref{prop item 2}. Let $\HH$
 and $g$ be as in \ref{prop item 1}. It follows from \eqref{eq: HH determines g} that for every $\sigma \in S _ \Lambda$ the $\KK$-torus $\HH$ splits in $\KK _ \sigma$. If $\HH$ is $\KK$-split then it has to be conjugate in $\SL _ 2 (\KK)$ to $\T$; since $\HH$ was defined only after conjugation by elements of $\SL _ 2 (\KK)$ we may as well take it to be equal to $\T$.  
Moreover, \eqref{eq: HH determines g} implies that if $\HH = \T$ then for every $\sigma \in S _ \Lambda$ we have that the $\KK _ \sigma$ component $g _ \sigma$ satisfies
\[
g_\sigma \in \TT(\K_\sigma)\sqcup w \TT(\K_\sigma)
\]
for  $w= \left(\!\begin{smallmatrix} 0& 1 \\ -1& 0 \end{smallmatrix}\!\right)$. For $\sigma \in S_\Lambda$,
let $w_\sigma = I$ if $g_\sigma \in  \TT(\K_\sigma)$ and $w_\sigma = w$ if $g_\sigma \in  w\TT(\K_\sigma)$. For $\sigma \not\in S_\Lambda$, let $w_\sigma = I$. Let $\mathbf w = (w_\sigma)$. Then $\mathbf w$ normalizes $\TT(\Adel_\KK)$ and $\mu$ is equivalent-by-compact to an invariant measure $\mu'$ supported on~$\Gamma \T (\Adel_{\KK})\mathbf w$.

Thus if the conclusion of \ref{prop item 2} does not hold, $\HH$ anisotropic over $\KK$. Let $\mathbb L$ be the splitting field of $\HH$; then $[\mathbb L: \KK] = 2$, and $\HH$ splits at $\sigma \in S _ \KK$ if and only if $\mathbb L$ can be embedded in $\KK _ \sigma$, in other words if and only if there are two places of $\mathbb L$ lying over $\sigma$. By the Chebotarev Density Theorem \cite[\S5.13]{Iwaniec-Kowalski}, the density of such places (ordering places by the cardinality of the residue field) is \%50, in contradiction to the assumptions of \ref{prop item 2}.

\medskip

Finally, assume $\Lambda$ satisfies the conditions of \ref{prop item 3}. Then in particular $\Lambda$ is unbounded at all places, hence  we can arrange that $\HH = \T$. Let $\mathbf w=(w_\sigma)$ be as in the previous paragraph.
If either $w_\sigma=I$ for  all $\sigma \in S_\KK$ or $w_\sigma=w$ for all $\sigma \in S_\KK$ then $\T(\KK)$ fixes the points in $\Gamma\T(\Adel_K)\mathbf w = \Gamma\T(\Adel_K)$, hence by ergodicity there is a $x _ 0 \in \T (\KK) \backslash \T (\Adel_{\KK})$ so that $\mu = \delta _ {x _ 0}$.

Otherwise, take \[S = \{\sigma \in S_\KK: w _ \sigma =I)\}\]and let $k \in \Lambda$ be such so that 
\[\prod_{\sigma \in S} |k|_{\sigma}\neq 1
\]
Then $x \sa k^n$ escapes to infinity for all $x \in \supp\mu$, in contradiction to Poincare recurrence.
This establishes \ref{prop item 3}.

\end{proof}

\subsection{The positive entropy case}

\begin{theorem}[Positive entropy]\label{positive entropy one}
 Let $\KK$, $\Gamma$, $\TT<\SL_2$ and $X_{\Adel_\K}$ be as in Theorem~\ref{adelic-zeroent}.
Let $\Lambda < \Adel_ \KK^{\times}$ be of higher rank and let $A=\{\sa k : k \in \Lambda\}$. If $\mu$ is an $A$-invariant and ergodic probability measure with $h_\mu(a_0)>0$ for some $a_0 \in A$ then $\mu$ is equivalent-by-compact to an $A'$-invariant measure $\mu'$ so that one of the following holds true.
\begin{enumerate}[label=\textup{(\textit{\roman*})},leftmargin=*]
	\item \label{item-algebraic-pos-adelic}There is a semisimple $\Q$ algebraic subgroup $\HH \leq \GG=\operatorname{Res}_{\K|\Q}\SL_2$ so that $\mu'$ is the uniform measure on $\Gamma \HH(\Adel_\Q)g$.
	
    \item\label{item-solvable-pos-adelic} $\mu'$ is supported on a single orbit $\Gamma N_G^1 (\UU) g$, where $\UU$ is the $\KK$-group $\twobytwo1*.1$ and 
   \begin{equation}
   \label{def of adelic N}
    N^1_G(\UU)= \left\{(n_\sigma)_\sigma: n_\sigma =\twobytwo {a_\sigma} {b_\sigma} . {a_\sigma^{-1}} , \prod_{\sigma\in S_\KK} |a_\sigma|_\sigma = 1\right\}.
    \end{equation}
    Moreover, $\mu'$ is invariant under a conjugate of the $\Adel_\Q$ points of a nontrivial unipotent $\Q$-subgroup of $\operatorname{Res}_{\KK|\Q}\UU$ 
    and for every $\sigma \in S_\KK$ in which $\Lambda$ is unbounded, the corresponding component~$g_\sigma$ of $g$ satisfies
    \[ g_\sigma\TT(\K_\sigma)g_\sigma^{-1} < \twobytwo {*} {*} . {*} .\]
\end{enumerate}
\end{theorem}

\medskip

If we make more stringent assumptions on $\Lambda$ we can nail down the possible measures $\mu$ further, without having to resort to classifying invariant measures up to equivalence-by-compact. For this, we again make note of another possible assumption on $\Lambda$:
\begin{enumerate}[label=\textup{(A\arabic*)},start=3,leftmargin=*]
    \item\label{item:isolate a place} there is a $\sigma \in S_\K$ lying over some $v \in S_\Q$ and a course Lyapunov subgroup $U^{[\alpha]}$ of $\SL_2(\Adel_K)$ so that (\textit{a}) $\KK_\sigma =\Q_v$, \ (\textit{b}) $U^{[\alpha]}$ intesects $\SL_2(\K_\sigma)$ non-trivially, and (\textit{c}) $U^{[\alpha]}$ does \emph{not} intersect $\SL_2(\K_{\sigma'})$ for any other $\sigma' |v$.
\end{enumerate}  

\begin{theorem}[Positive entropy]\label{positive entropy two}
 Let $\KK$, $\Gamma$, $\TT<\SL_2$ and $X_{\Adel_\K}$ be as in Theorem~\ref{adelic-zeroent}.
Let $\Lambda < \KK^{\times}$ be of higher rank satisfying both \ref{item:indep-2} and \ref{item:isolate a place}, and let $A=\{\sa k : k \in \Lambda\}$. Then any $A$-invariant and ergodic probability measure $\mu$ with $h_\mu(a_0)>0$ for some $a_0 \in A$ satisfies either
\begin{enumerate}[label=\textup{(\textit{\roman*'})},leftmargin=*]
\item\label{item-algebraic-pos-adelic'} $\mu$ is the uniform Haar measure on $X_{\Adel_\K}$.
\item \label{item-solvable-pos-adelic'} $\mu$ is the uniform Haar measure on the closed orbit $\Gamma\UU(\Adel_\KK) a$, where $a \in\TT(\Adel_\K)$ and $\UU$ is is one of the two unipotent $\K$-subgroups  $\twobytwo 1 * . 1$ or $\twobytwo 1 . * 1$.
\end{enumerate}
\end{theorem}

\begin{proof} [Proof of Theorem~\ref{positive entropy one}]
Let $S_{\mathrm{ram}}$ denote the set of ramified places of $\KK$. For any $k \in \Adel_{\KK} ^ \times$ and every finite place  $\sigma \in S_{\KK,\mathrm{f}} \setminus S_{\mathrm{ram}}$ we have that $\absolute {k _ \sigma}_\sigma \in p ^ \Z $ where $p$ is the rational prime divided by $\sigma$. Since $S_{\mathrm{ram}}$ is finite, it follows that there is a finite index subgroup $\Lambda ' \leq \Lambda$ so that for every $k \in \Lambda '$ and every finite place $\sigma$ we have that $\absolute {k _ \sigma} \in p^ \Z$ for the corresponding rational prime $p$.

Let $\Omega < \T (\Adel_{\KK})$ be the compact group given by \eqref{eq: Omega}. For $k \in \Lambda '$ let
\[
\ssa k = \left (\twobytwo {\absolute k _ \sigma} .. {{\absolute k _ \sigma}^{-1}} \right)_{\sigma } \in \TT(\Adel_\KK),
\]
and let $A ' = \left\{ \sa k: k \in \Lambda ' \right\}$ and $\mathring A = \left\{ \ssa k: k \in \Lambda ' \right\}$.
Then $\mathring A$, considered as a subgroup of $\GG (\Adel_\Q)$ for~$\GG = \operatorname{Res} _ {\KK| \Q} \SL _ 2$, is of class-$\cA'$, and for every $k \in \Lambda '$
\[
\sa k (\ssa k) ^{-1} \in \Omega
.\]
In particular, both $A'$ and $\mathring{A}$ are cocompact subgroups of $A\Omega$.

Replacing $a _ 0$ with $a _ 0 ^ \ell$ if necessary, we may assume that $a _ 0 = \sa {k _ 0}$ with $k _ 0 \in \Lambda '$. Let $\mathring a_0=\ssa{k_0}$.
Let $\bar \mu$ be the image of $\mu$ on $X_{\Adel_\K}/\Omega$, and note that both $A'$ and $\mathring A$ act on $X_{\Adel_\K}/\Omega$ and indeed these actions coincide. Since $A '$ is of finite index in $A$ and $\mu$ is $A$-ergodic, the ergodic decomposition of $\mu$ with respect to $A '$ has at most finitely many ergodic components. Let $\mu_1$ be one of these and let $\bar \mu_1$ denote the corresponding measure on $X_{\Adel_\K}/\Omega$. Then 
\[\mu=\textstyle{\frac 1{[\Lambda:\Lambda']}\sum_{a \in A/A'} a.\mu_1 \qquad \bar \mu=\frac 1{[\Lambda:\Lambda']}\sum_{a \in A/A'} a. \bar \mu_1.}\]
As the image of an $A'$-ergodic measure on $ X_{\Adel_\K}$ the measure $\bar \mu_1$ is $A'$-ergodic. Let $\bar{\bar\mu}_1$ be an $\mathring A$-ergodic and invariant lift of $\bar \mu$ to $ X_{\Adel_\K}$. Since $X_{\Adel_\K} \to X_{\Adel_\K}/\Omega$ is a compact extension
in the sense of Furstenberg (with an isometric action on the fibres), 
\begin{equation*}
h_{\mu }(a_0)=h_{\mu_1 }(a_0)=h_{\bar \mu_1}(a_0)=h_{\bar \mu}(\mathring a_0)=h_{\bar {\bar \mu}_1}(\mathring a_0)
.\end{equation*}
Moreover if $m_\Omega$ is the Haar measure on $\Omega$ then
\begin{multline*}
\int_ {A \Omega / A} a. \mu \, da = \mu*m_\Omega = \frac{1}{[\Lambda:\Lambda']} \sum_ {a_1 \in A / A '} a_1. (\mu_1 *m_\Omega) = \\
\frac{1}{[\Lambda:\Lambda']} \sum_ {a_1 \in A / A '} \int_ {\mathring{A} \Omega / \mathring{A}} a_1a. \bar{\bar \mu}_1\, da = \int_ {\mathring A \Omega / \mathring{A}} a. \bar{\bar \mu}_1\, da,
\end{multline*}
so $\mu$ is equivalent-by-compact to the $\mathring{A}$-invariant and ergodic measure $\bar{\bar \mu}_1$ that has positive entropy under $\mathring a _ 0$.

Let $S_i \nearrow S_\Q$ be an increasing set of places,
\[\bar S_i = \{\sigma \in S_\K:\sigma |v \text{ for some $v\in S_i$}\},\]
and assume $S_1 \ni \infty$ and moreover
\begin{enumerate}[label=\textup{(\alph*)},leftmargin=*]
     \item\label{enough places for a_0} if $a_0 = \sa {k_0}$ then all $\sigma \in S_\K$ for which $\absolute{k_0}_\sigma\neq 1$ are in $\bar S_1$
    \item\label{higher rank already at 1} $\mathring A \cap \GG(\Q_{S_1})$ is of higher rank
\end{enumerate}
\noindent
(note our assumptions of $A$ being of higher rank implies that \ref{higher rank already at 1} is satisfied as long as $S _ 1$ includes enough places.)

Set
\begin{align*}
X_i &= \GG(\mathcal O_{S_i}) \backslash \GG(\Q_{S_i}) = \SL_2(\mathcal O_{\K,\bar S_i}) \backslash \SL_2(\K_{\bar S_i}),\\
\mathring A _i &= \mathring A \cap \GG(\Q_{S_i}).
\end{align*}
Let $\int_ {X _ {\Adel_{\KK}}} \bar{\bar \mu} _ x ^ {(i)} \,d \bar{\bar \mu} (x)$ is ergodic decomposition of $\bar{\bar \mu}$ with respect to $\mathring A _ i$. Since $\mathring A$ is abelian and acts ergodically on $\bar{\bar \mu}$, and $\mathring a _ 0 \in \mathring A _ i$ we have that 
\[h _ {\bar{\bar \mu} _ x ^ {(i)}} (\mathring a _ 0) = h _ {\bar{\bar \mu}} (\mathring a _ 0) \qquad \text{a.e.}
\]
For $j \geq i$, \ let $\mu _ x ^ {(i,j )}$ denote the push forward of $\bar{\bar \mu} _ x ^ {(i)}$ under the projection map~$\pi_j :X _ {\Adel_{\KK}} \to X _ j$.

Note that condition \ref{enough places for a_0} implies that the action of $\mathring a_0$ on $X_{\Adel_\KK}$ is a compact extension of the action of the projection $\mathring a^{(j)}$ of~$\mathring a_0$ to $\GG(\Q_{S_j})$ on  $\SL_2(\mathcal O_{\K,\bar S_i}) \backslash \SL_2(\K_{\bar S_i})$, and hence $h_{\mu_x^{(i,j)}}(\mathring a^{(i)})$ a.s.\ equals $h _ {\bar{\bar \mu} _ x ^ {(i)}} (\mathring a _ 0)$, hence to $h _ \mu (a _ 0) >0$.

Applying Theorem~\ref{sl2-thm-final} to $\mu^{(i,j)}_x$ (cf. also \S\ref{sec:embedding} for more explicit description of the solvable case), we see that either $\mu^{(i,j)}_x$ is the uniform measure on $\GG(\mathcal O_{S_j}) \HH (\Q _ {S _ j})g_j$ for some $\Q$-semisimple group $\HH$ that conceivably could depend on $x \in X_{\Adel_\KK}$ and on $i,j$, or $\mu^{(i,j)}_x$ is supported on a single compact orbit $\GG(\mathcal O_{S_j}) N^1_{\GG(\Q_{S_j})} (\UU) g_j$, is invariant under the $\Q_{S_j}$-points of a conjugate of a nontrivial unipotent $\Q$-subgroup of $\GG$ and furthermore 
\[
g_j \mathring {A} _i g_j^{-1} < N^1_{\GG(\Q_{S_j})} (\UU).
\]
For $j'>j$ the image of $\mu _ x ^ {(i,j' )}$ under the natural projection $X_{j'} \to X_j$ is $\mu _ x ^ {(i,j )}$; e.g. by considering for every $\sigma \in S_\KK$ what are the $\KK _ \sigma$-unipotent subgroups of $\SL _ 2 (\KK _ \sigma)$ preserving $\mu _ {j'}$ one sees that
\begin{itemize}
\item  whether $\mu^{(i,j)}_x$ satisfies the semisimple or solvable case of Theorem~\ref{sl2-thm-final} does not depend on $j$;
\item in the semisimple case $\HH$ can be chosen to be independent of $j$;
\item the $g_j$ can be chosen so that $g_j = \pi_j(g)$ for a fixed $ g \in \GG (\Adel_{\KK})$. Moreover, since a.s.\ $x \in \supp\bar{\bar\mu}^{(i)}_x$, the element  $g$ can be chosen so that $\SL_2(\KK)g=x$.
\end{itemize}
It follows that that either 
\begin{enumerate}[label=\textup{(${\bar{\arabic*}}_x$)},leftmargin=*]
    \item\label{the semisimple case at x} $\bar{\bar\mu}^{(i)}_x$ is the uniform measure on $\GG(\KK) \HH (\Adel_{\KK})g$ for a $\Q$-semisimple group $\HH_x$, possibly depending on $x$, 
    \item\label{the solvable case at x} $\bar{\bar\mu}^{(i)}_x$ is supported on a single compact orbit $\GG(\KK) N^1_{G} (\UU) g$, is invariant under a conjugate of the $\Adel_\Q$-points of a nontrivial unipotent $\Q$-subgroup of $\GG$ and furthermore
\[
g \mathring {A} _i g^{-1} < N^1_{G} (\UU).
\]
\end{enumerate}

Assume for the moment \ref {the semisimple case at x} holds for $\bar{\bar\mu}^{(i)}_x$. The group $\HH _ x$ is uniquely defined up to conjugation by $\GG(\KK) $. If the semisimple case holds for $\bar{\bar\mu}^{(i)}_x$, it also has to hold for $\bar{\bar\mu}^{(i)}_{a.x}$ and up to conjugation by an element of $\GG(\KK) $, \ $\HH _ {x}$ is the same as $\HH _ {a.x}$ for $a \in A$.

It follows from ergodicity of $\bar{\bar \mu}$ under $\mathring A$ that either a.s.~$\bar{\bar\mu}^{(i)}_{x}$ satisfies \ref{the semisimple case at x} or a.s.~$\bar{\bar\mu}^{(i)}_{x}$ satisfies \ref{the solvable case at x}, and (since there are only countably many $\Q$-subgroups of $\GG$) if \ref {the semisimple case at x} holds a.s.\ one may take $\HH_x$ to be independent of $x$ (it may still depend on $i$ at this stage). Let $\HH_i$ denote this group.

Since $\mathring A _ i < \mathring A _ j$ for $i < j$, for $\bar{\bar\mu}$-almost every $x$, the $\mathring A _ i$-ergodic component of $\bar{\bar\mu}^{(j)}_x$ coincides with $\bar{\bar\mu}^{(i)}_x$. It is easy to see this implies that at least for $i$ large enough which case of Theorem~\ref{sl2-thm-final} applies a.s.\ to $\bar{\bar\mu}^{(i)}_x$ does not depend on $i$. 

If \ref{the semisimple case at x} holds, possibly after modifying $\HH_i$ by conjugating it with an element of $\SL_2(\KK)$, we have that for $i<j$, \ $\HH_i \leq \HH_{j}$. Thus, in this case, for $i$ large $\HH_i=\HH$ does not depend on $i$, establishing \ref{item-algebraic-pos-adelic} of Theorem~\ref{positive entropy one}.

If, on the other hand, \ref{the solvable case at x} holds, then
fixing $x$ we see that $\bar{\bar\mu}^{(i)}_x$ is supported on single compact orbit $\GG(\KK) N^1_{G} (\UU) g$ where $g$ has to be independent of $i$ because $\bar{\bar\mu}^{(i)}_x$ is an ergodic component of $\bar{\bar\mu}^{(j)}_x$ for $j>i$. By the martingale convergence theorem and ergodicity of $\bar{\bar\mu}$ under $\mathring A$, $\bar{\bar\mu}^{(i)}_x \to \bar{\bar\mu} $ in the weak$^*$ topology, hence $\bar{\bar\mu} $ is supported on the single compact orbit $\GG(\KK) N^1_{G} (\UU) g$. Since $g \mathring {A} _i g ^{-1} < N^1_{G} (\UU)$ for all~$i$ 
\[
g \mathring {A} g^{-1} < N^1_{G} (\UU).
\]
It follows that for every $\sigma\in S_\K$ for which $\Lambda$ is unbounded, hence the projection of $\mathring A$ to $\TT(\K_\sigma)$ is Zariski dense,
\[ g_\sigma\TT(\K_\sigma)g_\sigma^{-1} < \twobytwo {*} {*} . {*} .\]
Since each $\bar{\bar\mu}^{(i)}_x$ is invariant under  a conjugate of the $\Adel_\Q$-points of a nontrivial unipotent $\Q$-subgroup of $\GG$ the same holds for the limiting measure $\bar{\bar\mu}$. This establishes \ref{item-solvable-pos-adelic} of Theorem~\ref{positive entropy one}.
\end{proof}

Assuming that $\Lambda<\KK^\times$ and that it satisfies both \ref{item:indep-2} and \ref{item:isolate a place} we can now establish Theorem~\ref{positive entropy two}:

\begin{proof} [Proof of Theorem~\ref{positive entropy two}]
Let $\bar{\bar\mu}$, $\mathring A$ and $\Omega$ be as in the proof of Theorem~\ref{positive entropy one}.

Suppose first that $\bar{\bar\mu}$ satisfies \ref{item-algebraic-pos-adelic} of that theorem.
The group $\GG = \operatorname{Res} _ {\KK | \Q}\SL_2$ is (by definition) isomorphic over $\KK$ to $\SL _ 2 ^ {[\KK: \Q]}$. By Proposition~\ref{subgroup-L-of-sl2}, $\HH$ projects onto each of these $\SL _ 2$-factors. Let $v \in S_\QQ$ and $\sigma \in S_\KK$ be as in \ref{item:isolate a place}. Then $\HH (\Q _ v)$ projects onto the factor $\GG (\KK _ \sigma)$ of $G$.

Let $U ^ {[\alpha]}<G$ be a coarse Lyapunov subgroup as in \ref{item:isolate a place}. Then $\bar{\bar\mu}$ is invariant under $g ^{-1} \HH (\Q _ v) g$ and so $\bar{\bar\mu}$-a.s.\ we have that $\bar{\bar\mu} ^ {[\alpha]} _ x$ is invariant under $g ^{-1} \HH (\Q _ v) g \cap U ^ {[\alpha]}$.  By \ref{item:isolate a place}
\[
0 \lneq U ^ {[\alpha]} \cap \GG (\Q _ v) \leq U ^ {[\alpha]} \cap \SL _ 2 (\KK _ \sigma).
\]
Thus $\bar{\bar\mu}$ is invariant under a nontrivial unipotent subgroup of $\SL _ 2 (\KK _ \sigma)$. Since it is also invariant under the group $g ^{-1} \HH (\Q _ v) g$ that projects onto $\SL _ 2 (\KK _ \sigma)$, we have that $\bar{\bar\mu}$ is invariant under $\SL _ 2 (\KK _ \sigma)$. By strong approximation, $\bar{\bar\mu}$ has to be the uniform measure $m_{X_{\Adel_\K}}$ on $X_{\Adel_\KK}$.

It follows that the original $A$-invariant measure $\mu$ satisfies that $\mu * m _ \Omega = m_{X_{\Adel_\K}}$. Since $A \Omega$ is abelian, this presents $m_{X_{\Adel_\K}}$ as the average of the $A$-invariant measures $a.\mu$ for $a \in \Omega$. Since $m_{X_{\Adel_\K}}$ is $A$-ergodic, the only way this can happen is if $a.\mu = m_{X_{\Adel_\K}}$ for almost every $a$, hence $\mu = m_{X_{\Adel_\K}}$ establishing \ref{item-algebraic-pos-adelic'} of Theorem~\ref{positive entropy two}.

\medskip

Now suppose that $\bar{\bar\mu}$ satisfies \ref{item-solvable-pos-adelic} of Theorem~\ref{positive entropy one}. Multiplying $g$ by an element of $\UU(\Adel_{\KK})$ on the left, we may ensure that
for every $\sigma \in S_\K$ for which $\Lambda$ is unbounded,
\[ g_\sigma\TT(\K_\sigma)g_\sigma^{-1} < \twobytwo {*} . . {*} .\]
Since \ref{item:indep-2} implies that $\Lambda$ is unbounded at all places this is true for all $\sigma \in S_\KK$.
It follows  that for any place $\sigma \in S_\KK$, we have that the $\KK _ \sigma$ component $g _ \sigma$ of $g$ satisfies $g_\sigma \in w_\sigma \TT(\K_\sigma)$ with $ w _ \sigma = \twobytwo 1..1$ or $\twobytwo .1{-1}.$.

Now arguing as in the last paragraph of the proof of Theorem~\ref{adelic-zeroent}, it follows from \ref{item:indep-2} that either for all places $\sigma$ we have that $ w _ \sigma = \twobytwo 1..1$ or for all places $\sigma$ we have that $ w _ \sigma =\twobytwo .1{-1}.$. In either case $\mathbf w=(w_\sigma)\in \GG(\K)$ hence
in the first case $\bar{\bar\mu}$ is supported on a single orbit of
$\Gamma N^1_G(\UU)a$ for some $a\in \TT(\Adel_\K)$, and in the second on a single orbit of $\Gamma N^1_G(\tilde \UU)a$ for $\tilde \UU= \twobytwo 1.*1 = \mathbf w\UU \mathbf w$ (and again $a\in \TT(\Adel_\K)$).

Assume for simplicity we are in the case $\mathbf w = I$; the other case is handled identically. Using the assumption \ref{item:isolate a place} similarly to the semisimple case, we may conclude that if $\sigma$ is as in \ref{item:isolate a place}, the measure $\bar{\bar\mu}$ is invariant under $\UU(\KK _ \sigma)$ hence by a simple cases of strong approximation under all of $\UU(\Adel_{\KK})$.

Passing back to $\mu$ what we have discovered about $\bar{\bar\mu}$ implies that $\mu$ is supported on the single orbit $\Gamma N^1_G(\UU)a$ (since this set is $\Omega$-invariant), as well as $\UU(\Adel_{\KK})$-invariant (since this group is normalized by $\Omega$).

Now
\begin{equation} \label{foliating solvable by unipotent}
\Gamma N^1_G(\UU)a \subset \bigcup_ {a' \in \T (\Adel_{\KK})} \Gamma\UU(\Adel_{\KK})a',
\end{equation}
and one can identify the sets we are taking union over as atoms of some countably generated Borel $\sigma$-algebra of subsets of $\Gamma N_G(\UU)$.

Since $\Lambda \leq K ^ \times$, each element in $A$ fixes each of the sets on the right hand side of \eqref{foliating solvable by unipotent}. Since $\mu$ is ergodic, it must be supported on a single $\Gamma\UU(\Adel_{\KK})a'$, and since it is $\UU(\Adel_{\KK})$-invariant it is the uniform measure on this set.

Taking care of the case $\mathbf w = \twobytwo .1{-1}.$ in an identical fashion, this establishes \ref{item-solvable-pos-adelic'}.

\end{proof}

\section{A Diopantine Application}

In this section we prove Theorem~\ref{thm:diophantine} (Theorem \ref{dioph-cor} from the introduction).

\begin{theorem}[A Diophantine result]\label{thm:diophantine}
For any $\varepsilon>0$, and integer $\ell>0$ there exists an $N\in\N$, so that for every $\alpha\in\R$,
for every $Q\geq 10$ there exist some integers $n\leq N$ and $q\leq Q$ for which
\[
\langle qn^{2\ell}\alpha\rangle\leq \min\bigl(\tfrac{\varepsilon}q,\tfrac{1+\varepsilon}{Q}\bigr)
\]
\end{theorem}

The proof relies on Theorem \ref{thm:adelic}, applied on the space
\[
X_{\Adel}=\SL_2(\Q)\backslash\SL_2(\Adel).
\]
Let $S$ be the set of places of $\QQ$.
For $\alpha \in \R$, let $h(\alpha)=(h_v(\alpha))_{v \in S}$ be the element
\[
h_v(\alpha)= \begin{cases}
\twobytwo 1 . \alpha 1
& v=\infty\\
\hspace{0.5cm} e & v\neq\infty\end{cases}
\]
and $[h(\alpha)]$ the corresponding coset in $X_{\Adel}$.
For $k\in \Q$, set $m(k)=\twobytwo {\tfrac1{k}} . . k,$ embedded diagonally in $\SL_2(\Adel)$, and let 
\[
A=\left\{
    m(k^\ell)
:
m\in\Q^\times\right\}.
\]
Note that if $k \in \Q^{\times}$, since $m(k)\in \SL_2(\Q)$,
\begin{equation}\label{eq:action of m}
m(k).[h(\alpha)]=[h(\alpha)]m(k)^{-1} = [m(k)h(\alpha)m(k)^{-1}]=[h(k^2\alpha)].
\end{equation}

By unique factorization into prime powers $\Q^\times$ is isomorphic to the abelian group $\Z/(2)\times\sum_{p}\Z$. Hence we can define a F{\o}lner sequence $(F_j)_j$ in $\Q^\times$ as follows. We define $F_j$ to consist of all  $n\in\Z$ for which the prime factorisation of $n$ only uses primes $p\leq j$ with multiplicity $\leq j$. This implies that $F_j$ is almost invariant under $p$ for all primes $p\leq j$. 

Let $a_t=\twobytwo{e^t} . . {e^{-t}} \in \SL_2(\R)$, considered as an element in $\SL_2(\mathbb A)$ with all entries at finite primes the identity. Then
 \[
\min_{1\leq q\leq Q} \langle q\alpha\rangle\leq \tfrac{c}Q
\]
if and only if the lattice 
\[
\Z^2 h_\infty(\alpha) a_{\log Q}=\left\langle (Q,0),(Q\alpha,Q^{-1})\right\rangle
\]
has a nonzero vector in $[-c,c]\times [-1,1]$, i.e. if and only if there is a nonzero \[(w_v)_{v \in S} \in \Q^2 h(\alpha)a_{\log Q}\] with $w_\infty \in [-c,c]\times [-1,1]$ and $w_p \in \Z_p^2$ for all primes $p$.

\begin{proof}[Proof of Theorem \ref{thm:diophantine}]
Suppose Theorem \ref{thm:diophantine} fails for some $\varepsilon>0$ and $\ell$. This implies that for any $j\geq 1$, there is  an $\alpha_j \in \R$ and $t_j\geq 1$ so that for all $n \in F_j$,
\[
\langle qn^{2\ell}\alpha_j\rangle> \min\bigl(\tfrac{\varepsilon}q,(1+\varepsilon)e^{-t_j}\bigr) \qquad\text{for all $q\leq e^{t_j}$}.
\]
In particular, we have 
\begin{equation}\label{eq:less interesting bound}
\min_{q\leq e^{t_j}}\langle qn^{2\ell}\alpha_j\rangle>{\varepsilon}e^{-t_j},
\end{equation}
for all $n \in F_j$. By the correspondence discussed before the proof we have equivalently that
for all $n\in F_j$ the point 
\[a_{-t_j}m(n^\ell).[h(\alpha_j)]\]
belongs to the compact subset
\[
\Omega  = \bigl\{[h]\in X_{\mathbb A}: \Q^2 h \cap C_\varepsilon = \{0\}\bigr\} 
\]
with \[C_\varepsilon=\{(w_v)_{v \in S}: w_\infty \in (-\varepsilon, \varepsilon)\times(-1,1), w_p \in \Z_p^2 \quad\text{for all primes $p$}\}.\]

Let $\mu_j$ be the probability measure
\[
\mu_j = \frac1{|F_j|}\sum_{n \in F_j} \delta_{a_{-t_j}m(n^\ell).[h(\alpha_j)]}.
\]
As all the measures $\mu_j$ are supported on the compact set $\Omega$, there is a subsequence, say $\{\mu_j\}_{j \in J}$, which converges weak$^{*}$ to a probability measure $\mu$ on $\Omega$.

As $F_j$ is a F\o lner sequence in $\Q^{\times}$, the sequence of sets $\{m(n^\ell):n \in F_j\}$ form a F\o lner sequence in $A$. Thus $\mu$ is $A$-invariant.
Let $\mu_x^{\mathcal E}$ be an $A$-ergodic component of $\mu$ for which $\supp \mu_x^{\mathcal E} \subseteq \supp \mu$ (this holds for a.e.\ ergodic component).
Applying Theorem~\ref{thm:adelic} we conclude that $\mu_x^{\mathcal E}$ is one of
three types of measures, all having in common that there is a $y \in \TT(\mathbb A_K)$ so that the corresponding point $[y] \in \TT(\Q)\backslash\TT(\mathbb A_K)$ is in $\supp \mu_x^{\mathcal E} \subseteq \supp \mu$.

Thus there is for any $j \in J$ a choice of $n_j \in F_j$ so that $a_{-t_j}m(n_j^\ell).[h(\alpha_j)] \to [y]$. Let $y_\infty \in \TT(\R)$ be the
point (unique up to sign) so that
\[
y\in \{y_\infty\}\times\prod_{p\ \mathrm{prime}} \TT(\Z_p).
\]
As $\supp\mu \subseteq \Omega$ we have 
$[y] \in \Omega$, which implies that
\[
\Z^2 y_\infty \cap (-\varepsilon, \varepsilon)\times(-1,1) = \{0\}.
\]
Since $y_\infty$ is a diagonal matrix of determinant 1, it follows that $(b,0)\in \Z^2y_\infty$ for some $0<|b|\leq 1$.

Since for $j\in J$ large enough $a_{-t_j}m(n_j^\ell).[h(\alpha_j)] = a_{-t_j}.[h(n_j^{2\ell}\alpha_j)]$ is very close to $[y]$, the lattice 
\[
\Z^2 \begin{pmatrix}
     1&\\ n^{2\ell}\alpha&1
 \end{pmatrix} a_{t_j} = \Z^2 \begin{pmatrix}e^{t_j}&0\\e^{t_j} n_j^{2\ell}\alpha_j& e^{-t_j}\end{pmatrix}
\]
has a vector, say $\bigl(e^{t_j}(qn_j^{2\ell}\alpha_j-p),e^{-t_j}q\bigr)$ with $q>0$, very close to $(b,0)$. In particular, for $j$ large enough, 
\begin{align*}
e^{t_j}\langle qn_j^{2\ell}\alpha_j\rangle &< 1+\varepsilon\\
e^{-t_j}q &<\varepsilon.
\end{align*}
So for $Q_j = e^{t_j}$ there is a $q<\epsilon Q$ so that
\[
\langle qn_j^{2\ell}\alpha_j\rangle < \tfrac {1+\varepsilon}{Q_j}
\]
in contradiction to the definition of $\alpha_j$
at the beginning of the proof.
\end{proof}


\begin{thebibliography}{10}

\bibitem{Aka-Einsiedler-Shapira}
M.~Aka, M.~Einsiedler, and U.~Shapira.
\newblock Integer points on spheres and their orthogonal lattices.
\newblock {\em Invent. Math.}, 206(2):379--396, 2016.

\bibitem{Aka-Einsiedler-Wieser}
M.~{Aka}, M.~{Einsiedler}, and A.~{Wieser}.
\newblock {Planes in four space and four associated CM points}.
\newblock {\em arXiv e-prints}, page arXiv:1901.05833, Jan 2019.

\bibitem{Boshernitzan}
M.D.~Boshernitzan, \emph{Quantitative recurrence results}, Invent. Math., 113(3): 617–631, 1993.

\bibitem{Blomer-Bourgain-Radziwill-Rudnick}
V.~Blomer, J.~Bourgain, M.~Radziwi\l\l, and Z.~Rudnick.
\newblock Small gaps in the spectrum of the rectangular billiard.
\newblock {\em Ann. Sci. \'{E}c. Norm. Sup\'{e}r. (4)}, 50(5):1283--1300, 2017.

\bibitem{Borel-Prasad}
A.~Borel and G.~ Prasad.
\newblock \emph{Values of isotropic quadratic forms at S-integral points}. 
\newblock Compositio Math. 83 (1992), no. 3, 347–372.

\bibitem{Bourgain-Lindenstrauss}
J.~Bourgain, E.~Lindenstrauss, \emph{Entropy of quantum limits},
{Comm. Math. Phys.}, 233(1):153-171, 2003.

\bibitem{Brin-Katok}
M.~Brin and A.~Katok.
\newblock On local entropy.
\newblock In {\em Geometric dynamics ({R}io de {J}aneiro, 1981)}, volume 1007
  of {\em Lecture Notes in Math.}, pages 30--38. Springer, Berlin, 1983.

\bibitem{Carmon-quadratic}
D.~Carmon.
\newblock Evenly divisible rational approximations of quadratic
  irrationalities.
\newblock {\em Israel J. Math.}, 223(1):441--448, 2018.

\bibitem{Cassels-Swinnerton-Dyer}
J.~W.~S. Cassels and H.~P.~F. Swinnerton-Dyer.
\newblock On the product of three homogeneous linear forms and the indefinite
  ternary quadratic forms.
\newblock {\em Philos. Trans. Roy. Soc. London. Ser. A.}, 248:73--96, 1955.

\bibitem{EFS}
M.~Einsiedler, L.~Fishman, U.~Shapira.
\newblock Diophantine approximations on fractals.
\newblock {\em Geom. Funct. Anal.}, 21 (2011), no.~1, 14-35.

\bibitem{Einsiedler-Katok}
M.~Einsiedler and A.~Katok.
\newblock Invariant measures on {$G/\Gamma$} for split simple {L}ie groups
  {$G$}.
\newblock {\em Comm. Pure Appl. Math.}, 56(8):1184--1221, 2003.
\newblock Dedicated to the memory of J\"urgen K. Moser.

\bibitem{EinsiedlerKatokNonsplit}
M.~Einsiedler and A.~Katok, \emph{Rigidity of measures -- the high entropy
case, and non-commuting foliations}, Probability in mathematics.
 Israel J. Math.  148  (2005), 169--238.
 
\bibitem{Einsiedler-Katok-Lindenstrauss}
M.~Einsiedler, A.~Katok, and E.~Lindenstrauss.
\newblock Invariant measures and the set of exceptions to {L}ittlewood's
  conjecture.
\newblock {\em Ann. of Math. (2)}, 164(2):513--560, 2006.

\bibitem{EL-action-on-torus}
M.~Einsiedler and E.~Lindenstrauss.
\newblock \emph{Rigidity properties of $\mathbb{Z}^d$-actions on tori and solenoids}. 
\newblock Electron. Res. Announc. Amer. Math. Soc. 9 (2003), 99–110.

\bibitem{Einsiedler-Lindenstrauss-ICM}
M.~Einsiedler and E.~Lindenstrauss.
\newblock Diagonalizable flows on locally homogeneous spaces and number theory.
\newblock In {\em International Congress of Mathematicians. Vol. II}, pages
  1731--1759. Eur. Math. Soc., Z\"urich, 2006.

\bibitem{Einsiedler-Lindenstrauss-joining}
M.~Einsiedler, E.~Lindenstrauss, \emph{Joinings of higher-rank diagonalizable actions on locally homogeneous
 spaces.}
 Duke Math. J.  138  (2007),  no. 2, 203--232.

\bibitem{Low-entropy}
M.~Einsiedler, E.~Lindenstrauss. \emph{On measures invariant under diagonalizable actions: the rank-one case
 and the general low-entropy method},
 J. Mod. Dyn.  2  (2008),  no. 1, 83--128.

\bibitem{Pisa-notes}
M.~Einsiedler, E.~Lindenstrauss. \emph{Diagonal actions on locally homogeneous spaces}.
 Homogeneous flows, moduli spaces and arithmetic, 
 155--241, Clay Math. Proc., 10, Amer. Math. Soc., Providence, RI,  2010.

\bibitem{full-torus-paper}
M.~Einsiedler, E.~Lindenstrauss.
\newblock On measures invariant under tori on quotients of semisimple groups.
\newblock {\em Ann. of Math. (2)}, 181(3):993--1031, 2015.


\bibitem{Einsiedler-Lindenstrauss-joinings-2}
M.~Einsiedler, E.~Lindenstrauss.
\newblock Joinings of higher rank torus actions on homogeneous spaces.
\newblock {\em Publ. Math. Inst. Hautes \'{E}tudes Sci.}, 129:83--127, 2019.

\bibitem{Einsiedler-Lindenstrauss-symmetry}
Manfred Einsiedler and Elon Lindenstrauss.
\newblock Symmetry of entropy in higher rank diagonalizable actions and measure
  classification.
\newblock {\em J. Mod. Dyn.}, 13:163--185, 2018.


\bibitem{ELW}
M.~Einsiedler, E.~Lindenstrauss, T.~Ward, \emph{Entropy in ergodic theory and homogeneous dynamics},
in preparation, see \url{http://maths.dur.ac.uk/~tpcc68/entropy/welcome.html}.

\bibitem{ET-book}
M.~Einsiedler, T.~Ward,  \emph{Ergodic theory with a view towards number theory},
Graduate Texts in Mathematics, 259. Springer-Verlag London, Ltd., London,  2011. xviii+481 pp.

\bibitem{Furstenberg-disjointness-1967}
H.~Furstenberg.
\newblock Disjointness in ergodic theory, minimal sets, and a problem in
  {D}iophantine approximation.
\newblock {\em Math. Systems Theory}, 1:1--49, 1967.

\bibitem{Iwaniec-Kowalski}
H.~Iwaniec and E.~Kowalski.
\newblock {\em Analytic number theory}, volume~53 of {\em American Mathematical
  Society Colloquium Publications}.
\newblock American Mathematical Society, Providence, RI, 2004.


\bibitem{KatokSpatzier96}
A.~Katok and R.~Spatzier, \emph{Invariant measures for higher-rank
hyperbolic
abelian actions}, Ergodic Theory Dynam. Systems \textbf{16} (1996), no.~4,
751--778.

\bibitem{Khayutin-CM-joinings}
I.~Khayutin.
\newblock Joint equidistribution of {CM} points.
\newblock {\em Ann. of Math. (2)}, 189(1):145--276, 2019.

\bibitem{Lindenstrauss-adelic}
E.~Lindenstrauss.
\emph{Adelic dynamics and arithmetic quantum unique ergodicity}.
{Current developments in mathematics, 2004}, 111--139, 2006,

\bibitem{Lindenstrauss-03}
E.~Lindenstrauss. \emph{Invariant measures and arithmetic quantum unique
ergodicity}. Ann. of Math. (2)  163  (2006),  no. 1, 165--219.

\bibitem{Margulis-book} G. A. Margulis, \emph{Discrete Subgroups of Semisimple Lie Groups},
Ergebnisse der Mathematik und ihrer Grenzgebiete (3), 17. Springer-Verlag, Berlin, 1991. x+388 pp. 

\bibitem{Margulis-conjectures}
G. A. Margulis.
\newblock Problems and conjectures in rigidity theory.
\newblock In {\em Mathematics: frontiers and perspectives}, pages 161--174.
  Amer. Math. Soc., Providence, RI, 2000.

\bibitem{Margulis-Tomanov} G. A. Margulis and G. M. Tomanov,
\emph{Invariant measures for actions of unipotent groups over
local fields on homogeneous spaces}, Invent. Math. 116 (1994), no. 1-3, 347--392.

\bibitem{Margulis-Tomanov-almost-linear}
G.A.~Margulis, G.M.~Tomanov, \emph{Measure rigidity for almost linear groups and its applications.}
 J. Anal. Math.  69  (1996), 25--54.

\bibitem{Pesin-book}
Y.~Pesin, \emph{Dimension theory in dynamical systems},
   {Chicago Lectures in Mathematics},
   {University of Chicago Press},
   {Chicago, IL},
   {1997},
   {xii+304}.

\bibitem{Platonov-Rapinchuk}
V.~Platonov, A.~Rapinchuk,
       \emph{Algebraic groups and number theory},
 {Pure and Applied Mathematics},
 {Academic Press Inc.},
 {Boston, MA},
 {1994},
 Vol. {139}.
 

\bibitem{Ratner-joining}
M.~Ratner, \emph{Horocycle flows, joinings and rigidity of products}, Ann. of Math. 118:277--313, 1983.

\bibitem{Ratner-Acta}
M.~Ratner,
\newblock \emph{On measure rigidity of unipotent subgroups of semisimple groups},
\newblock {Acta Math.}, 165(3-4):229--309, 1990.

\bibitem{Ratner-Annals}
M.~Ratner,
\newblock \emph {On {R}aghunathan's measure conjecture}.
\newblock  Ann. of Math. (2), 134(3):545--607, 1991.

\bibitem{Ratner-padic}
M.~Ratner, \emph{Raghunathan's conjectures for Cartesian products of real and p-adic
 Lie groups.}
 Duke Math. J.  77  (1995),  no. 2, 275--382.


 \bibitem{Rudolph-2-and-3}
Daniel~J. Rudolph.
\newblock {$\times 2$} and {$\times 3$} invariant measures and entropy.
\newblock {\em Ergodic Theory Dynam. Systems}, 10(2):395--406, 1990.

\bibitem{Springer} T. A. Springer, {\it Linear algebraic groups}, Second edition, Birkh\"auser Boston, Boston, MA, 1998

\bibitem{Tomanov-orbits}
G.~Tomanov, \emph{Orbits on homogeneous spaces of arithmetic origin and
approximations.}
Analysis on homogeneous spaces and representation theory of Lie groups,
Okayama–Kyoto (1997), 
265--297, Adv. Stud. Pure Math., 26, Math. Soc. Japan, Tokyo,  2000.

\bibitem{Zannier}
U.~Zannier,
\newblock {\em Lecture notes on {D}iophantine analysis}, volume~8 of {\em
  Appunti. Scuola Normale Superiore di Pisa (Nuova Serie) [Lecture Notes.
  Scuola Normale Superiore di Pisa (New Series)]}.
\newblock Edizioni della Normale, Pisa, 2009.
\newblock With an appendix by Francesco Amoroso.

\end{thebibliography}
\end{document}